\documentclass{amsart}

\usepackage{ragged2e}
\usepackage{geometry}

\usepackage[utf8]{inputenc}
\usepackage{enumerate}

\usepackage{amsmath}
\usepackage{amsfonts}
\usepackage{tikz-cd}
\usetikzlibrary{braids}
\usepackage{mathabx}

\usepackage{bbm}
\usepackage{mathrsfs}
\usepackage{amsthm}
\usepackage{verbatim}
\usepackage{relsize}
\usepackage{dsfont}
\usepackage{graphicx}
\usepackage{amssymb}
\usepackage{mystix2}
\usepackage{xcolor}

\newcommand{\eps}{\varepsilon}
\newcommand{\ov}{\overline}
\newcommand{\id}{\textnormal{id}}
\newcommand{\mc}{\mathcal}

\newcommand{\mrm}{\mathrm}
\newcommand{\mscr}{\mathscr}
\newcommand{\mf}{\mathfrak}
\newcommand{\msf}{\mathsf}
\newcommand{\I}{\mathbbm{1}}

\newcommand{\vp}{\varphi}

\newcommand{\md}{\operatorname{d}\!}
\newcommand{\cst}{\ifmmode \mathrm{C}^* \else $\mathrm{C}^*$\fi}

\newcommand{\la}{\langle}
\newcommand{\ra}{\rangle}

\newcommand{\NN}{\mathbb{N}}
\newcommand{\RR}{\mathbb{R}}
\newcommand{\KK}{\mathbb{K}}
\newcommand{\CC}{\mathbb{C}}
\newcommand{\ZZ}{\mathbb{Z}}
\newcommand{\GG}{\mathbb{G}}
\newcommand{\HH}{\mathbb{H}}
\newcommand{\TT}{\mathbb{T}}

\newcommand{\QQ}{\mathbb{Q}}

\newcommand{\vv}{\mathrm{V}}
\newcommand{\Vv}{\mathds{V}}
\newcommand{\vV}{\text{\reflectbox{$\Vv$}}\:\!}

\newcommand{\ww}{\mathrm{W}}
\newcommand{\WW}{{\mathds{V}\!\!\text{\reflectbox{$\mathds{V}$}}}}
\newcommand{\Ww}{\mathds{W}}
\newcommand{\wW}{\text{\reflectbox{$\Ww$}}\:\!}

\newcommand{\wot}{\ifmmode \textsc{wot} \else \textsc{wot}\fi}
\newcommand{\sot}{\ifmmode \textsc{sot} \else \textsc{sot}\fi}
\newcommand{\sots}{\ifmmode \textsc{sot}^* \else \textsc{sot}$^*$\fi}
\newcommand{\ssot}{\ifmmode \sigma\textsc{-sot} \else $\sigma$-\textsc{sot }\fi}
\newcommand{\ssots}{\ifmmode \sigma\textsc{-sot}^* \else $\sigma$-\textsc{sot }$^*$\fi}
\newcommand{\swot}{\ifmmode \sigma\textsc{-wot} \else $\sigma$-\textsc{wot}\fi}

\newcommand{\Linf}{\operatorname{L}^{\infty}(\GG)}
\newcommand{\Linfd}{\operatorname{L}^{\infty}(\whG)}

\newcommand{\wh}{\widehat}

\newcommand{\wt}{\widetilde}
\newcommand{\whG}{\widehat{\GG}}
\newcommand{\whH}{\widehat{\HH}}
\newcommand{\LdG}{\operatorname{L}^{2}(\GG)}
\newcommand{\LdH}{\operatorname{L}^{2}(\HH)}

\newcommand{\oon}{\operatorname}

\newcommand{\braid}[2]{{}^{#1}{\fdiagovrdiag}^{#2}}
\newcommand{\braidu}[2]{{}^{#1}{\rdiagovfdiag}^{#2}}

\renewcommand{\restriction}{\mathord{\upharpoonright}}
\newcommand{\rest}{\restriction}

\DeclareMathOperator{\lin}{span}

\DeclareMathOperator{\B}{B}

\DeclareMathOperator{\M}{M}

\DeclareMathOperator{\N}{N}

\DeclareMathOperator{\Dom}{Dom}

\DeclareMathOperator{\LL}{L}

\DeclareMathOperator{\Rep}{Rep}

\DeclareMathOperator{\CB}{CB}

\DeclareMathOperator{\Aut}{Aut}
\DeclareMathOperator{\Ad}{Ad}
\DeclareMathOperator{\vN}{vN}

\newtheorem{theorem}{Theorem}[section]

\newtheorem{proposition}[theorem]{Proposition}
\newtheorem{lemma}[theorem]{Lemma}
\theoremstyle{definition}
\newtheorem{corollary}[theorem]{Corollary}
\newtheorem{remark}[theorem]{Remark}

\newtheorem{definition}[theorem]{Definition}

\numberwithin{equation}{section}

\begin{document}
\title{Braided tensor product of von Neumann algebras}

\author{Kenny De Commer}
\address{Vrije Universiteit Brussel\\
Pleinlaan 2\\
1050 Brussels\\
Belgium
}
\email{kenny.de.commer@vub.be}

\author{Jacek Krajczok}
\address{Vrije Universiteit Brussel\\
Pleinlaan 2\\
1050 Brussels\\
Belgium
}
\email{jacek.krajczok@vub.be}





\date{}

\begin{abstract}
We introduce a definition of braided tensor product $\M\ov\boxtimes\N$ of von Neumann algebras equipped with an action of a quasi-triangular quantum group $\GG$ (this includes the case when $\GG$ is a Drinfeld double). It is a new von Neumann algebra which comes together with embeddings of $\M,\N$ and the unique action of $\GG$ for which embeddings are equivariant. More generally, we construct braided tensor product of von Neumann algebras equipped with actions of locally compact quantum groups linked by a bicharacter. We study several examples, in particular we show that crossed products can be realised as braided tensor products. We also show that one can take the braided tensor product $\vartheta_1\boxtimes\vartheta_2$ of normal, completely bounded maps which are equivariant, but this fails without the equivariance condition.
\end{abstract}

\maketitle

\section{Introduction}

Existence of fruitful connections between theories of locally compact groups and operator algebras has been well established throughout the years. In particular, one can study symmetries of a given von Neumann algebra $\M$ by considering actions on it by locally compact groups, described by point-$w^*$-continuous homomorphisms $G\rightarrow \Aut(\M)$. For example, Connes showed in \cite{ConnesClassification} that deep structural properties of $\M$ are encoded in the action of $\RR$ on $\M$ via modular automorphisms. As an another example, given an action $G\curvearrowright \M$, one can consider the associated crossed product von Neumann algebra $G\ltimes \M$. Algebras of this kind were among the first examples of type III factors (\cite{TakesakiII}).

A very broad generalisation of the theory of locally compact groups was introduced by Kustermans and Vaes (\cite{KustermansVaes, KustermansVaesVN}); objects they define are called \emph{locally compact quantum groups}. Every locally compact group gives rise to a locally compact quantum group (these are often called \emph{classical}), and one can take the dual $\whG$ of any locally compact quantum group $\GG$. Consequently, the theory of Kustermans and Vaes contains an extension of the classical Pontryagin duality. In particular, one can take the dual of any classical (not necesarilly abelian) locally compact group $G$. The quantum group $\wh{G}$ is described by objects studied in abstract harmonic analysis; for example $\LL^{\infty}(\wh{G})=\oon{vN}(G)$ is the group von Neumann algebra and $\mrm{C}_0^u(\wh{G})=\mrm{C}^*(G)$ is the full group \cst-algebra.

In the realm of actions of locally compact groups, a basic construction of a new action is given by the tensor product of actions. More precisely, if we have two actions $\alpha^{\M}\colon G\curvearrowright \M,\alpha^{\N}\colon G\curvearrowright \N$, then we can consider a ``diagonal'' action on $\M\bar\otimes \N$ defined via $G\ni g\mapsto \alpha^{\M}(g)\otimes \alpha^{\N}(g)\in \Aut(\M\bar\otimes \N)$. It is characterised by the property that the canonical embeddings $\M,\N\rightarrow \M\bar\otimes \N$ are equivariant.

As with classical groups, one can consider actions of locally compact quantum groups on von Neumann algebras $\GG\curvearrowright \M$. Unlike in the classical case mentioned above, it is however not necessarily true that if $\GG$ acts on $\M,\N$, then there is an action on $\M\bar\otimes \N$ for which the canonical embeddings are equivariant. An easy counterexample is given by two copies of the translation action $\Delta_{\wh{F_2}}\colon \wh{F_2}\curvearrowright \LL^{\infty}(\wh{F_2})=\oon{vN}(F_2)$. Our work shows, that if we have actions of $\GG$ and its dual $\whG$ on $\M,\N$ which satisfy the \emph{Yetter-Drinfeld} condition (Definition \ref{def4}), then one can define a twisted version of tensor product, called the \emph{braided tensor product} $\M\ov\boxtimes\N$ (Definition \ref{def2}). It is a von Neumann algebra together with canonical embeddings $\M,\N\rightarrow \M\ov\boxtimes\N$ and actions $\GG,\whG\curvearrowright \M\ov\boxtimes\N$ such that the embeddings are equivariant (Proposition \ref{prop20}). The same conclusion holds more generally, if one replaces actions of $\GG,\whG$ by an action of a quasi-triangular quantum group. This is indeed a generalisation, as pairs of actions of $\GG,\whG$ which satisfy the Yetter-Drinfeld condition correspond bijectively to actions of the Drinfeld double $D(\GG)$, which is a quasi-triangular quantum group.

In fact, we define the braided tensor product $\M\ov\boxtimes\N$ more generally 
for any two von Neumann algebras $\M,\N$ equipped with actions $\HH\curvearrowright \M,\GG\curvearrowright \N$ of locally compact quantum groups $\GG,\HH$ which are linked via a bicharacter $\wh{\mc{X}}\in \LL^{\infty}(\whH)\bar\otimes \LL^{\infty}(\whG)$. At this level of generality however, we don't obtain any action on $\M\ov\boxtimes\N$ (the quasi-triangular case corresponds to $\HH=\GG$ and $\wh{\mc{X}}=\wh{\oon{R}}$ being the $\oon{R}$-matrix).

Let us mention that the construction of the braided tensor product is well-known and of fundamental importance in Hopf algebra theory: we refer to \cite{Maj95,HS20} for many instances and applications of `braided' mathematics in an algebraic setting. Similar constructions have also appeared in the \cst-algebraic setting: a minimal version of the braided tensor product in \cite{MeyerRoyWoronowiczI, MeyerRoyWoronowiczII, NestVoigt}, and a maximal version in \cite{RoyTimmermann}. Our definition is close to the one of \cite{MeyerRoyWoronowiczI} (up to conventions), but because of the fact that we work with von Neumann algebras our arguments are very different. For example, one defines the braided tensor product $\M\ov\boxtimes\N$ to be a $\swot$-closed subspace spanned by elements of the form $\iota_{\M}(m)\iota_{\N}(n)$ (analogs of simple tensors), but it is not obvious why this subspace is closed under multiplication. We reach this conclusion (after some preparations) in Theorem \ref{thm1}. Our argument uses the biduality theorem (\cite[Theorem 2.6]{VaesUnitary}) to reduce the problem to the case of dual actions, where a direct calculation is possible. In the restricted setting of Yetter-Drinfeld actions of discrete quantum groups, a version of braided tensor product appeared also in \cite{Moakhar}. A special case of our construction with $\GG=\RR,\HH=\wh{\RR}$ and $\wh{\mc{X}}$ equal to the Kac-Takesaki operator appeared in the work of Houdayer on type $\oon{III}_1$ factors \cite{Houdayer} (see Example \ref{example3}).

\cst-algebraic braided tensor products are used in the theory of ``braided \cst-quantum groups'' (\cite{BQGRoy, BraidedSU2, RahamanRoy}) and we hope to use in the future our construction to study ``braided locally compact quantum groups'' in the setting of von Neumann algebras and weights.\\

Let us describe how the paper is structured. In the next section we recall the necessary parts of the theory of locally compact quantum groups and Drinfeld doubles. In Section \ref{sec:braidedflip} we define braided flip operators $\braid{U}{V}$ which depend on the chosen bicharacter. When $\GG$ is a quasi-triangular quantum group and braided flip operators are defined using $\oon{R}$-matrix as the bicharacter, then $\braid{U}{V}$ make $\Rep(\GG)$ into a braided monoidal category (Proposition \ref{prop2}). In Section \ref{sec:btp} we fix locally compact quantum groups $\GG,\HH$, a bicharacter $\wh{\mc{X}}$, actions $\HH\curvearrowright\M,\GG\curvearrowright\N$ and their implementations. Then we define $\M\ov\boxtimes\N$ using braided flip operators, show that it is a von Neumann algebra (Theorem \ref{thm1}) and that it does not depend on the chosen implementations (Proposition \ref{prop3}). As an important technical tool, we introduce a universal lift of the action $\alpha^{\M}\colon \M\rightarrow \Linf\bar\otimes \M$ to a map $\alpha^{\M,u}\colon \M\rightarrow \mrm{C}_0^u(\GG)^{**}\bar\otimes \M$ with similar properties. In Section \ref{sec:maps} we show that one can take the braided tensor product of maps $\vartheta_1\boxtimes\vartheta_2$ if they are normal, CB and equivariant (Proposition \ref{prop8}). Then we use it to prove restricted stability of approximation properties under the braided tensor product construction (Proposition \ref{prop9}). In the next section we assume that $\GG=\HH$ is a quasi-triangular quantum group with an $\oon{R}$-matrix $\wh{\mc{X}}=\wh{\oon{R}}$, construct the canonical action on the braided tensor product (Proposition \ref{prop20}) and prove associativity of $\ov\boxtimes$ (Proposition \ref{prop11}). In fact, we obtain more general statements by considering actions of canonical quantum subgroups $\GG_1,\GG_2\subseteq\GG$ (if $\GG=D(\HH)$ is the Drinfeld double, then $\GG_1=\HH$ and $\GG_2=\whH$). Section \ref{sec:inf} is devoted to the construction of an infinite braided tensor product. It also depends on the choice of invariant states, similarly as the usual infinite tensor product. In Section \ref{sec:examples} we describe several examples, in particular we show that  the braided tensor product of (non-equivariant) functionals might fail to exist (Proposition \ref{prop13}), it can happen that $\M\ov\boxtimes \N\nsimeq \N\ov\boxtimes\M$ (Proposition \ref{prop18}), and that crossed products can be realised as braided tensor products (Proposition \ref{prop17}). This allows us to easily conclude that the braided tensor product behaves differently than the usual tensor product when considering types, $w^*$ CBAP or factoriality of the involved von Neumann algebras. The Appendix contains two lemmas concerning locally compact quantum groups.

We write $\chi$ for the flip map on tensor product of algebras and $\Sigma$ for the flip of Hilbert spaces. Symbol $\ov{\lin}^{\,\tau}$ will denote the closure of the linear span in topology $\tau$, while $\ov{\lin}$ means the norm-closure. Throughout the paper we use leg numbering notation (with flips), so e.g.~$x_{[32]}=\I\otimes \chi(x)$. For a von Neumann algebra, $\CB^\sigma(\M)$ means the space of completely bounded, normal maps equipped with the CB norm. Whenever $\omega$ is a n.s.f.~weight on a von Neumann algebra, we write $\mf{N}_{\omega}$ for the left ideal of square integrable elements $\mf{N}_{\omega}=\{x\in \M\mid \omega(x^*x)<+\infty\}$, $\msf{H}_\omega$ for the GNS Hilbert space, $\Lambda_\omega$ for the GNS map and $J_\omega$ for the modular conjugation. We will typically denote the Hilbert space on which von Neumann algebra $\M$ is represented (in a not necessarily standard way) by $\msf{H}_{\M}$, and similarly, representation $U$ will typically act on $\msf{H}_U$.


\section{Preliminaries}\label{sec:preliminaries}
\subsection{General preliminaries}
We work in the setting of locally compact quantum groups, as defined by Kustermans and Vaes (\cite{KustermansVaes, KustermansVaesVN}). A \emph{locally compact quantum group} $\GG$ is described by a von Neumann algebra $\LL^{\infty}(\GG)$ together with a \emph{comultiplication} $\Delta_{\GG}$, which is an injective, normal, unital $*$-homomorphism $\Delta_{\GG}\colon \Linf\rightarrow \Linf \bar\otimes \Linf$ satisfying the coassociativity condition $(\Delta_{\GG}\otimes \id)\Delta_{\GG}=(\id\otimes\Delta_{\GG})\Delta_{\GG}$. Furthermore, by definition on $\Linf$ there are two n.s.f.~weights $\vp_{\GG},\psi_{\GG}$ called \emph{left/right Haar integral}, which satisfy respectively a left/right invariance condition. An important result states that they are unique up to a positive multiple. We denote by $J_{\GG}$ the modular conjugation of $\vp_{\GG}$ and by $\LL^2(\GG)$ the GNS Hilbert space of $\vp_{\GG}$. One can canonically identify the GNS Hilbert space of $\psi_{\GG}$ with $\LL^2(\GG)$ and then $J_{\psi_{\GG}}$ is equal to $J_{\vp_{\GG}}=J_{\GG}$ up to a scalar multiple. An important role in the theory is played by two unitary operators $\ww^{\GG},\vv^{\GG}$ acting on $\LL^2(\GG)\otimes\LL^2(\GG)$ (\emph{Kac-Takesaki operators}). They are characterised by
\[
((\omega\otimes\id)\ww^{\GG *})\Lambda_{\vp_{\GG}}(x)=
\Lambda_{\vp_{\GG}} ((\omega\otimes\id)\Delta_{\GG}(x)),\quad
((\id\otimes \omega)\vv^{\GG })\Lambda_{\psi_{\GG}}(y)=
\Lambda_{\psi_{\GG}} ((\id\otimes\omega)\Delta_{\GG}(y))
\]
for $\omega\in \B(\LdG)_*, x\in \mf{N}_{\vp_{\GG}},y\in \mf{N}_{\psi_{\GG}}$. Both operators implement the comultiplication via
\begin{equation}\label{eq58}
\Delta_{\GG}(x)=\ww^{\GG *}(\I\otimes x)\ww^{\GG}=
\vv^{\GG}(x\otimes \I) \vv^{\GG *}\quad(x\in\Linf).
\end{equation}

With every locally compact quantum group $\GG$ one can associate its \emph{dual} $\whG$, which is also a locally compact quantum group. By construction $\Linfd=\ov{\lin}^{\,\swot}\{(\omega\otimes\id)\ww^{\GG}\mid \omega\in \B(\LdG)_*\}\subseteq \B(\LL^2(\GG))$, comultiplication $\Delta_{\whG}$ is implemented by $\ww^{\whG}=\ww^{\GG *}_{[21]}$ as in \eqref{eq58} and we can identify $\LL^2(\whG)=\LL^2(\GG)$. \emph{Pontryagin duality} states that the dual of $\whG$ is equal to $\GG$. One can show that $J_{\whG}\Linf J_{\whG}=\Linf$, hence one can consider map $R_{\GG}=J_{\whG}(\cdot)^* J_{\whG}$ on $\Linf$, called the \emph{unitary antipode}. We denote by $\nu_{\GG}>0$ the \emph{scaling constant}. We will also use the self-adjoint unitary $u_{\GG}=\nu_{\GG}^{i/8} J_{\GG} J_{\whG}$ and the map $j_{\GG}=J_{\GG}(\cdot)^* J_{\GG}$, the canonical anti-isomorphism $\LL^{\infty}(\GG)\rightarrow \LL^{\infty}(\GG)'$. For technical reasons, it is convenient to associate with $\GG$ two more quantum groups (\cite[Section 4]{KustermansVaesVN}). The \emph{opposite} quantum group $\GG^{op}$ with $\LL^{\infty}(\GG^{op})=\Linf$ and $\Delta_{\GG^{op}}=\chi\Delta_{\GG}$ and the \emph{commutant} quantum group $\GG'$ defined by $\LL^{\infty}(\GG')=\LL^{\infty}(\GG)'$ and $\Delta_{\GG'}=(J_{\GG}\otimes J_{\GG})\Delta_{\GG}(J_{\GG}(\cdot)J_{\GG})(J_{\GG}\otimes J_{\GG})$. One has $(\GG^{op})'=(\GG')^{op}$, and $u_{\GG}$ implements an isomorphism between $(\GG^{op})'$ and $\GG$.

There is also a \cst-algebraic description of $\GG$. One can prove that $\ov{\lin}\, \{(\id\otimes \omega)\ww^{\GG}\mid \omega\in \B(\LdG)_*\}$ is a weak$^*$-dense \cst-subalgebra of $\Linf$ and comultiplication restricts to a non-degenerate $*$-homomorphism $\mrm{C}_0(\GG)\rightarrow \M(\mrm{C}_0(\GG)\otimes \mrm{C}_0(\GG))$ (which we denote by the same symbol). We have
\begin{equation}\label{eq54}
\ww^{\GG}\in \M(\mrm{C}_0(\GG)\otimes \mrm{C}_0(\whG))\subseteq \LL^{\infty}(\GG)\bar\otimes \Linfd,\quad 
\vv^{\GG}\in \M(\mrm{C}_0(\wh{\GG}')\otimes \mrm{C}_0(\GG))\subseteq \LL^{\infty}(\wh{\GG}')\bar\otimes \Linf
\end{equation}
and
\begin{equation}\label{eq55}\begin{split}
(\Delta_{\GG}\otimes \id)\ww^{\GG}&=\ww^{\GG}_{[13]}\ww^{\GG}_{[23]},\quad
(\id\otimes \Delta_{\whG})\ww^{\GG}=\ww^{\GG}_{[13]}\ww^{\GG}_{[12]},\\
(\id\otimes \Delta_{\GG})\vv^{\GG}&=\vv^{\GG}_{[12]}\vv^{\GG}_{[13]},\quad
(\Delta_{\wh{\GG}'}\otimes \id)\vv^{\GG}=\vv^{\GG}_{[13]}\vv^{\GG}_{[23]}.
\end{split}
\end{equation}
Furthermore, the unitaries $\ww^{\GG},\vv^{\GG}$ are related via
\begin{equation}\label{eq56}
\vv^{\GG}=
(j_{\whG}\otimes R_{\GG})\ww^{\GG}_{[21]}
\end{equation}
(see part $(6)$ of Remark \ref{remark2}). We will also use a universal version of the \cst-algebraic picture (\cite{Kustermans}). There is a \cst-algebra $\mrm{C}_0^u(\GG)$ together with surjective $*$-homomorphism $\pi_{\GG}\colon \mrm{C}_0^u(\GG)\rightarrow\mrm{C}_0(\GG)$ called the \emph{reducing map}. A lot of the above objects have their universal versions, e.g.~there is a universal version of the comultiplication $\Delta_{\GG}^u\colon \mrm{C}_0^u(\GG)\rightarrow \M(\mrm{C}_0^u(\GG)\otimes \mrm{C}_0^u(\GG))$, of the unitary antipode $R_{\GG}^u$ and of the map $j_{\GG}^u$. They are linked to their reduced versions by $(\pi_{\GG}\otimes \pi_{\GG})\Delta_{\GG}^u=\Delta_{\GG}\pi_{\GG}$, $\pi_{\GG} R_{\GG}^u=R_{\GG} \pi_{\GG}$ and $\pi_{\GG'}j_{\GG}^u=j_{\GG}\pi_{\GG}$. What will be important for us, is the existence of lifted versions of the canonical unitaries:
\[
\wW^{\GG}\in \M(\mrm{C}_0(\GG)\otimes \mrm{C}_0^u(\whG)),\quad 
\Ww^{\GG}\in \M(\mrm{C}_0^u(\GG)\otimes \mrm{C}_0(\whG)),\quad 
\WW^{\GG}\in \M(\mrm{C}_0^u(\GG)\otimes \mrm{C}_0^u(\whG))
\]
which satisfy expected identities, e.g.~$(\pi_{\GG}\otimes \id)\WW^{\GG}=\wW^{\GG}$, $(\id\otimes \pi_{\whG})\WW^{\GG}=\Ww^{\GG}$ and $(\id\otimes \pi_{\whG})\wW^{\GG}=\ww^{\GG}$. There are also lifted versions of $\vv^{\GG}$. They satisfy relations analogous to \eqref{eq55} and \eqref{eq56}. On the \cst-algebra $\mrm{C}_0^u(\GG)$ there is a canonical character $\eps_{\GG}^u$ called the \emph{counit}. It is characterised by $(\eps_{\GG}^u\otimes\id)\Ww^{\GG}=\I$. Lifted versions of the Kac-Takesaki operators allow us to introduce half-lifted comultiplications:
\begin{equation}\label{eq57}
\begin{split}
\Delta_{\GG}^{u,r}\colon \mrm{C}_0(\GG)\ni x &\mapsto 
\Ww^{\GG *}(\I\otimes x)\Ww^{\GG}\in \M(\mrm{C}_0^u(\GG)\otimes 
\mrm{C}_0(\GG)),\\
\Delta_{\GG}^{r,u}\colon \mrm{C}_0(\GG)\ni x &\mapsto 
\vV^{\GG }(x\otimes \I)\vV^{\GG *}\in \M(\mrm{C}_0(\GG)\otimes 
\mrm{C}_0^u(\GG)).
\end{split}
\end{equation}

A (left, unitary) \emph{representation} of $\GG$ on a Hilbert space $\msf{H}$ is a unitary element $U\in \M(\mrm{C}_0(\GG)\otimes \mc{K}(\msf{H}))$ satisfying $(\Delta_{\GG}\otimes \id)U=U_{[13]}U_{[23]}$. An important result of Kustermans (\cite[Proposition 5.2]{Kustermans}) states that every representation of $\GG$ is of the form $U=(\id\otimes \phi_{U})\wW^{\GG}$ for a non-degenerate $*$-homomorphism $\phi_{U}\colon \mrm{C}_0^u(\whG)\rightarrow \B(\msf{H})=\M(\mc{K}(\msf{H}))$. This establishes a $1$-$1$ correspondence between representations of $\GG$ and (non-degenerate) representations of $\mrm{C}_0^u(\whG)$; one can think of $\mrm{C}_0^u(\whG)$ as a full group \cst-algebra of $\GG$. If $U,V$ are representations of $\GG$ on $\msf{H}_U,\msf{H}_V$, then an \emph{intertwiner} between $U$ and $V$ is a bounded, linear map $T\colon \msf{H}_U\rightarrow \msf{H}_V$ satisfying $(\I\otimes T)U=V(\I\otimes T)$. We let $\Rep(\GG)$ be the category with representations of $\GG$ as objects and intertwiners as morphisms; there are natural notions of direct sum $U\oplus V$ (acting on $\msf{H}_U\oplus \msf{H}_V$) and tensor product $U\otop V=U_{[12]}V_{[13]}$ (acting on $\msf{H}_U\otimes \msf{H}_V$).

Let $\M$ be a von Neumann algebra. A (left) \emph{action} of $\GG$ on $\M$ is an injective, normal $*$-homomorphism $\alpha^{\M}\colon \M\rightarrow \Linf\bar\otimes \M$ satisfying $(\Delta_{\GG}\otimes\id)\alpha^{\M}=(\id\otimes \alpha^{\M})\alpha^{\M}$. We will write in this situation $\alpha^{\M}\colon \GG\curvearrowright \M$. If $\M$ acts on a Hilbert space $\msf{H}_{\M}$ and $U^{\M}=(\id\otimes \phi_{\M})\wW^{\GG}$ is a representation of $\GG$ on $\msf{H}_{\M}$ which satisfies $\alpha^{\M}(m)=U^{\M *}(\I\otimes m)U^{\M}\,(m\in\M)$, then we say that $U^{\M}$ (or equivalently $\phi_{\M}$) \emph{implements} $\alpha^{\M}$. There always exists at least one implementation of $\alpha^{\M}$ on the standard Hilbert space $\LL^2(\M)$ of $\M$, called the \emph{standard} implementation (\cite[Definition 3.6]{VaesUnitary}, see also Lemma \ref{lemma9} and the discussion before it). 

\subsection{Quantum subgroups and bicharacters}
Let $\GG,\HH$ be locally compact quantum groups. $\HH$ is said to be a (closed) \emph{quantum subgroup} of $\GG$ (in the sense of Vaes), if there exists an injective, normal, unital $*$-homomorphism $\gamma_{\HH\subseteq\GG}\colon\LL^{\infty}(\whH)\rightarrow\Linfd$ such that $\Delta_{\whG}\gamma_{\HH\subseteq \GG}=(\gamma_{\HH\subseteq\GG}\otimes \gamma_{\HH\subseteq\GG})\Delta_{\whH}$ (see \cite[Definition 3.1 and Theorem 3.3]{DKSS_ClosedSub}). In this situation one can find a non-degenerate $*$-homomorphism $\theta_{\HH\subseteq \GG}\colon \mrm{C}_0^u(\GG)\rightarrow \M(\mrm{C}_0^u(\HH))$ satisfying $\Delta_{\HH}^u\theta_{\HH\subseteq \GG}=(\theta_{\HH\subseteq\GG}\otimes \theta_{\HH\subseteq \GG})\Delta_{\GG}^u$ (i.e.~$\theta_{\HH\subseteq\GG}$ is a \emph{strong quantum homomorphism}), such that its \emph{dual} strong quantum homomorphism $\wh{\theta}_{\HH\subseteq \GG}\colon \mrm{C}_0^u(\whH)\rightarrow\M(\mrm{C}_0^u(\whG))$, characterised by $(\theta_{\HH\subseteq\GG}\otimes \id)\WW^{\GG}=(\id\otimes \wh{\theta}_{\HH\subseteq\GG})\WW^{\HH}$,  satisfies $\pi_{\whG}\wh{\theta}_{\HH\subseteq\GG}=\gamma_{\HH\subseteq \GG}\pi_{\whH}$. In this situation, there is an action $\Delta_{\HH,\GG}\colon \HH\curvearrowright \LL^{\infty}(\GG)$ implemented by $(\id\otimes \pi_{\whG} \wh{\theta}_{\HH\subseteq\GG} )\wW^{\HH}$.

Assume now that $\HH$ is a quantum subgroup of $\GG$, and assume $\GG$ acts on a von Neumann algebra $\M$ via $\alpha^{\M}$. Then we can restrict $\alpha^{\M}\colon \GG\curvearrowright \M$ to an action $\alpha^{\M}\rest_{\HH}\colon \HH\curvearrowright \M$ (see \cite[Section 6.5]{DeCommerPhD}, but notice a difference in terminology). The restricted action is characterised by $(\Delta_{\HH,\GG}\otimes \id )\alpha^{\M}=(\id\otimes\alpha^{\M})(\alpha^{\M}\rest_{\HH})$. One can check that if $\M\subseteq \B(\msf{H}_{\M})$ and $U^{\M}=(\id\otimes \phi_{\M})\wW^{\GG}$ is a representation of $\GG$ on $\msf{H}_{\M}$ which implements $\alpha^{\M}$, then $U^{\M}\rest_{\HH}=(\id\otimes \phi_{\M}\wh{\theta}_{\HH\subseteq\GG}){{\wW}^{\HH}}\in \Rep(\HH)$ implements $\alpha^{\M}\rest_{\HH}$.

Next we introduce definitions of bicharacter and quasi-triangular locally compact quantum group (c.f.~\cite[Definition 2.1]{MeyerRoyWoronowiczII}).

\begin{definition}
Let $\GG,\HH$ be locally compact quantum groups.
\begin{enumerate}
\item A unitary $\wh{\mc{X}}\in \LL^{\infty}(\whH)\bar\otimes \LL^{\infty}(\whG)$ is a \emph{bicharacter} if
\begin{equation}\label{eq1}
(\Delta_{\whH}\otimes \id)\wh{\mc{X}}=
\wh{\mc{X}}_{[13]}\wh{\mc{X}}_{[23]},\quad
(\id\otimes \Delta_{\whG})\wh{\mc{X}}=
\wh{\mc{X}}_{[13]}\wh{\mc{X}}_{[12]}.
\end{equation}
\item A pair $(\GG,\wh{\oon{R}})$ is said to be a \emph{quasi-triangular} locally compact quantum group if $\wh{\oon{R}}\in \LL^{\infty}(\whG)\bar\otimes \LL^{\infty}(\whG)$ is a bicharacter satisfying
\begin{equation}\label{eq2}
\wh{\oon{R}} \Delta_{\whG}(x)\wh{\oon{R}}^*=
\Delta_{\whG^{op}}(x)\quad(x\in \LL^{\infty}(\whG)).
\end{equation}
In this case, the bicharacter $\wh{\oon{R}}$ is called an \emph{$\oon{R}$-matrix}.
\end{enumerate}
\end{definition}

\begin{remark}\label{remark2}\noindent
\begin{enumerate}
\item Our definition of an $\oon{R}$-matrix differs from the one in \cite{MeyerRoyWoronowiczII} because we use ``left'' conventions.
\item Every bicharacter $\wh{\mc{X}}$ belongs to the algebra $\M(\mrm{C}_0(\whH)\otimes \mrm{C}_0(\whG))$ (see e.g.\ the discussion after \cite[Proposition 4.1]{DawsCPMults}) and admits a universal lift, i.e.~there exists a unique unitary $\wh{\mc{X}}^u\in \M(\mrm{C}_0^u(\whH)\otimes \mrm{C}_0^u(\whG))$ satisfying $(\pi_{\whH}\otimes \pi_{\whG})\wh{\mc{X}}^u=\wh{\mc{X}}$ and (a universal version of) equations \eqref{eq1}. It will be convenient to use also the half-lifted versions $\wh{\mc{X}}^{r,u}=(\pi_{\whH}\otimes \id)\wh{\mc{X}}^u$ and $\wh{\mc{X}}^{u,r}=(\id\otimes \pi_{\whG})\wh{\mc{X}}^u$. If $\wh{\mc{X}}=\wh{\oon{R}}$ is an $\oon{R}$-matrix, then the lift $\wh{\oon{R}}^u$ also satisfies a universal version of \eqref{eq2} (\cite[Proposition 4.7]{Homomorphisms}, \cite[Proposition 2.4]{MeyerRoyWoronowiczII}), namely
\begin{equation}\label{eq61}
\wh{\oon{R}}^u \Delta_{\whG}^u(x)\wh{\oon{R}}^{u *}=
\Delta_{\whG^{op}}^u(x)\quad(x\in \mrm{C}_0^u(\whG)).
\end{equation}
\item The bicharacter property \eqref{eq1} implies that the universal lift $\wh{\mc{X}}^u$ satisfies $(\eps_{\whH}^{u}\otimes \id)\wh{\mc{X}}^u=\I, (\id\otimes \eps_{\whG}^u)\wh{\mc{X}}^u=\I$, see \cite[Lemma 2.15]{MeyerRoyWoronowiczII}.
\item The bicharacter $\wh{\mc{X}}$ can be seen as describing a \emph{morphism of quantum groups} $\whH\rightarrow \GG$, c.f.~\cite[Section 1.3]{DKSS_ClosedSub}.
\item Every bicharacter satisfies $(R_{\whH}\otimes R_{\whG})\wh{\mc{X}}=\wh{\mc{X}}$ and $(\tau^{\whH}_t\otimes \tau^{\whG}_t)\wh{\mc{X}}=\wh{\mc{X}}$, and similarly at the universal level (\cite[Proposition 3.10]{Homomorphisms}).
\item Above, we use the fact that $R_{\whH}\otimes R_{\whG}$ extends to a linear, $*$-preserving, antimultiplicative and strictly continuous bijection on $\M(\mrm{C}_0(\whH)\otimes \mrm{C}_0(\whG))$ -- this also applies to other maps which are linear, antimultiplicative, bijective and $*$-preserving. A similar property holds at the von Neumann algebraic level.
\end{enumerate}
\end{remark}

\begin{proposition}\label{prop1}
Let $(\GG,\wh{\oon{R}})$ be a quasi-triangular locally compact quantum group. Then $\wh{\oon{R}}$ and $\wh{\oon{R}}^u$ satisfy the following version of the Yang-Baxter equation:
\[
\wh{\oon{R}}_{[23]} \wh{\oon{R}}_{[13]} \wh{\oon{R}}_{[12]}=
\wh{\oon{R}}_{[12]}\wh{\oon{R}}_{[13]} \wh{\oon{R}}_{[23]},\quad
\wh{\oon{R}}_{[23]}^u \wh{\oon{R}}_{[13]}^u \wh{\oon{R}}_{[12]}^u=
\wh{\oon{R}}_{[12]}^u \wh{\oon{R}}_{[13]}^u \wh{\oon{R}}_{[23]}^u.
\]
\end{proposition}

This result can be proven by a straightforward calculation using \eqref{eq2}, see \cite[Equation (2.17)]{MeyerRoyWoronowiczII}.

\begin{remark}\label{remark4}
If $\wh{\mc{X}}\in \LL^{\infty}(\whH)\bar\otimes \LL^{\infty}(\whG)$ is a bicharacter, then so is $\wh{\mc{X}}^*_{[21]}\in \LL^{\infty}(\whG)\bar\otimes \LL^{\infty}(\whH)$. Its universal lift is $\wh{\mc{X}}^{u *}_{[21]}$. Furthermore, if $\wh{\oon{R}}$ is an $\oon{R}$-matrix, then $\wh{\oon{R}}^*_{[21]}$ is also an $\oon{R}$-matrix. This easy observation can be used sometimes to obtain quick proofs.
\end{remark}

In concrete calculations, it is very convenient to link the bicharacter $\wh{\mc{X}}$ with the Kac-Takesaki operator. More precisely, we have the following result (see \cite[Proposition 6.5]{Kustermans}, \cite[Proposition 4.2]{Homomorphisms}).

\begin{lemma}\label{lemma1}
Let $\GG,\HH$ be locally compact quantum groups and $\wh{\mc{X}}\in \M(\mrm{C}_0(\whH)\otimes \mrm{C}_0(\whG))$ a bicharacter with lift $\wh{\mc{X}}^u$. There is a unique non-degenerate $*$-homomorphism $\Phi\colon \mrm{C}_0^u(\GG)\rightarrow \M(\mrm{C}^u_0(\whH))$ such that
\[
\wh{\mc{X}}=(\pi_{\whH}\Phi\otimes \id)\Ww^{\GG},\quad
\wh{\mc{X}}^u=(\Phi\otimes \id)\WW^{\GG}\quad
\textnormal{and}\quad
\Delta_{\whH}^{u} \Phi=(\Phi\otimes \Phi)\Delta_{\GG}^{u}.
\]
\end{lemma}

\subsection{Drinfeld Double} 
One of the most important classes of examples of quasi-triangular quantum groups are the Drinfeld doubles.

\begin{definition}
Let $\GG$ be a locally compact quantum group. The \emph{Drinfeld double} of $\GG$ (\cite[Section 8]{BaajVaes}) is the locally compact quantum group defined via
\[
\LL^{\infty}(D(\GG))=\LL^{\infty}(\GG)\bar\otimes \LL^{\infty}(\whG),\quad 
\Delta_{D(\GG)}=\oon{Ad}(\ww^{\GG}_{[32]})\circ \Delta_{\GG\times \whG}.
\]
\end{definition}
One can show that $D(\GG)$ is unimodular and
\[
\ww^{D(\GG)}=\ww^{\GG}_{[13]}Z_{[34]}^*\ww^{\whG}_{[24]}
Z_{[34]}
\quad\textnormal{where}\quad
Z=\nu_{\GG}^{-\frac{i}{4}}\ww^{\GG}(J_{\GG}\otimes J_{\whG})
\ww^{\GG}(J_{\GG}\otimes J_{\whG})
\]
(see \cite[Definition 3.4, Proposition 8.1]{BaajVaes}). Both $\GG$ and $\whG$ are quantum subgroups of $D(\GG)$ (see e.g.~\cite[Lemma 7.13]{DKV_ApproxLCQG}), the associated maps $\gamma_{\GG\subseteq D(\GG)}, \gamma_{\whG\subseteq D(\GG)}$ are given by
\[
\gamma_{\GG\subseteq D(\GG)} (\wh{x}) = \wh{x}\otimes \I,\quad 
\gamma_{\whG\subseteq D(\GG)}(x)=Z^*(\I\otimes x)Z
\]
for $x\in \LL^{\infty}(\GG),\wh{x}\in\LL^{\infty}(\whG)$. In other words, we have
\begin{equation}\label{eq45}
\ww^{D(\GG)}=
(\id\otimes \gamma_{\GG\subseteq D(\GG)})(\ww^{\GG})_{[134]}
(\id\otimes \gamma_{\whG\subseteq D(\GG)})(\ww^{\whG})_{[234]}.
\end{equation}
An analogous formula holds at the half-universal level
\begin{equation}\label{eq46}
\wW^{D(\GG)}=
(\id\otimes \wh{\theta}_{\GG\subseteq D(\GG)})(\wW^{\GG})_{[134]}
(\id\otimes \wh{\theta}_{\whG\subseteq D(\GG)})(\wW^{\whG})_{[234]}
\end{equation}
(for the sake of consistency, we consider $\mrm{C}_0^u(\wh{D(\GG)})$ as having two legs, and to ease the notation we write simply $\wh{\theta}_{\GG\subseteq D(\GG)}, \wh{\theta}_{\whG\subseteq D(\GG)}$). Equality \eqref{eq46} can be proved by applying the half-lifted comultiplication to \eqref{eq45}.

The Drinfeld double is in a canonical way a quasi-triangular locally compact quantum group.

\begin{proposition}\label{prop14}
By means of the unitary
\[
\wh{\oon{R}}=(\gamma_{\whG\subseteq D(\GG)}\otimes \gamma_{\GG\subseteq D(\GG)})(\ww^{\GG}),
\]
the pair $(D(\GG),\wh{\oon{R}})$ is a quasi-triangular locally compact quantum group.
\end{proposition}

(See \cite[Lemma 6.11]{RoyDrinfeld} for a proof in a slightly different setting; note however a difference in terminology.) Another important feature of the Drinfeld double is a bijective correspondence between its actions and compatible pairs of actions of $\GG$ and $\whG$ (see \cite[Definition 3.1]{NestVoigt} and \cite[Proposition 3.2]{NestVoigt} in the \cst-algebraic framework).

\begin{definition}\label{def4}
A von Neumann algebra $\M$ is said to be a $\GG$-Yetter-Drinfeld ($\GG$-YD) von Neumann algebra if it is equipped with actions $\alpha^{\M}_{\GG}\colon \GG\curvearrowright \M$, $\alpha^{\M}_{\whG}\colon \whG\curvearrowright \M$ satisfying the compatibility condition
\[
(\chi \oon{Ad}(\ww^{\GG}) \otimes \id )(\id\otimes \alpha^{\M}_{\whG})\alpha^{\M}_{\GG}=
(\id\otimes \alpha^{\M}_{\GG})\alpha^{\M}_{\whG}.
\]
\end{definition}

\begin{proposition}\label{prop15}
There is a bijective correspondence between $D(\GG)$-von Neumann algebras and $\GG$-YD von Neumann algebras. More precisely, if $\alpha^{\M}$ is an action of $D(\GG)$ on $\M$, then the restricted actions $\alpha^{\M}\rest_{\GG},\alpha^{\M}\rest_{\whG}$ make $\M$ into a $\GG$-YD von Neumann algebra. Conversely, if $\M$ is a $\GG$-YD von Neumann algebra with respect to actions $\alpha^{\M}_{\GG}\colon \GG\curvearrowright \M$, $\alpha^{\M}_{\whG}\colon \whG\curvearrowright \M$, then $(\id\otimes\alpha^{\M}_{\whG})\alpha^{\M}_{\GG}$ is an action of $D(\GG)$. These constructions are inverse to each other.
\end{proposition}

We will need a result concerning implementations of a YD-action, which give rise to implementation of the corresponding action of $D(\GG)$. 

\begin{proposition}\label{prop16}
Let $\M\subseteq \B(\msf{H}_{\M})$ be a $\GG$-YD von Neumann algebra with respect to actions $\GG,\whG\curvearrowright \M$. Let $U^{\M,\GG},U^{\M,\whG}$ be implementations of these actions on $\msf{H}_{\M}$ with corresponding morphisms $\phi_{\M,\GG},\phi_{\M,\whG}$. Assume that
\begin{equation}\label{eq43}
\ww^{\GG}_{[12]}U^{\M,\GG}_{[13]} U^{\M,\whG}_{[23]}\ww^{\GG *}_{[12]}=
U^{\M, \whG}_{[23]} U^{\M, \GG}_{[13]}.
\end{equation}
Then there is an implementation $U^{\M}$ of $D(\GG)\curvearrowright \M$ with the corresponding morphism $\phi_{\M}$ such that 
\begin{equation}\label{eq44}
U^{\M}=
U^{\M,\GG}_{[13]} U^{\M,\whG}_{[23]},\quad
\phi_{\M,\GG}=\phi_{\M} \wh{\theta}_{\GG\subseteq D(\GG)},\quad
\phi_{\M,\whG}=\phi_{\M} \wh{\theta}_{\whG\subseteq D(\GG)}.
\end{equation}
Conversely, if $\phi_{\M}$ implements $D(\GG)\curvearrowright \M$, then $\phi_{\M} \wh{\theta}_{\GG\subseteq D(\GG)},$ $\phi_{\M} \wh{\theta}_{\whG\subseteq D(\GG)}$ implement the restricted actions $\GG,\whG\curvearrowright \M$. These implementations and the associated representations satisfy \eqref{eq43} and \eqref{eq44}.
\end{proposition}

Note that \eqref{eq43} resembles  the YD condition.

\begin{proof}
Assume first that $\M$ is a $\GG$-YD von Neumann algebra  with respect to actions implemented by $U^{\M,\GG},U^{\M,\whG}$ which satisfy \eqref{eq43}. Consider the unitary $U^{\M}= U^{\M, \GG}_{[13]}U^{\M, \whG}_{[23]}\in \LL^{\infty}(D(\GG))\bar\otimes \B(\msf{H}_{\M})$. It is a representation of $D(\GG)$:
\[\begin{split}
&\quad\;
(\Delta_{D(\GG)}\otimes\id)U^{\M}=
\ww^{\GG}_{[32]}
U^{\M,\GG}_{[15]}
U^{\M,\GG}_{[35]}
U^{\M,\whG}_{[25]}
U^{\M,\whG}_{[45]}
\ww^{\GG *}_{[32]}=
\bigl(
\ww^{\GG}_{[23]}
U^{\M,\GG}_{[15]}
U^{\M,\GG}_{[25]}
U^{\M,\whG}_{[35]}
U^{\M,\whG}_{[45]}
\ww^{\GG *}_{[23]}\bigr)_{[13245]}\\
&=
\bigl(
\ww^{\GG}_{[23]}
U^{\M,\GG}_{[15]}
\ww^{\GG *}_{[23]}
U^{\M,\whG}_{[35]}
U^{\M,\GG}_{[25]}
\ww^{\GG}_{[23]}
U^{\M,\whG}_{[45]}
\ww^{\GG *}_{[23]}\bigr)_{[13245]}=
\bigl(
U^{\M,\GG}_{[15]}
U^{\M,\whG}_{[35]}
U^{\M,\GG}_{[25]}
U^{\M,\whG}_{[45]}
\bigr)_{[13245]}\\
&=
U^{\M,\GG}_{[15]}
U^{\M,\whG}_{[25]}
U^{\M,\GG}_{[35]}
U^{\M,\whG}_{[45]}=
U^{\M}_{[125]}U^{\M}_{[345]},
\end{split}\]
and it is immediate that $U^{\M}$ implements the action $\alpha^{\M}=(\id\otimes\alpha^{\M}_{\whG})\alpha^{\M}_{\GG}\colon D(\GG)\curvearrowright \M$. There is a non-degenerate $*$-homomorphism $\phi_{\M}\colon \mrm{C}_0^u(\wh{D(\GG)})\rightarrow \B(\msf{H}_{\M})$ such that $U^{\M}=(\id\otimes \phi_{\M})(\wW^{D(\GG)})$. Using Equation \eqref{eq46} we have
\begin{equation}\begin{split}\label{eq59}
&\quad\;
(\id\otimes \phi_{\M}\wh{\theta}_{\GG\subseteq D(\GG)})(\wW^{\GG})_{[13]}
(\id\otimes \phi_{\M}\wh{\theta}_{\whG\subseteq D(\GG)})(\wW^{\whG})_{[23]}=(\id\otimes\phi_{\M})(\wW^{D(\GG)})=
U^{\M}\\
&=U^{\M,\GG}_{[13]}U^{\M,\whG}_{[23]}=
(\id\otimes \phi_{\M,\GG})(\wW^{\GG})_{[13]}
(\id\otimes \phi_{\M,\whG})(\wW^{\whG})_{[23]}.
\end{split}\end{equation}
Take $\omega\in \LL^1(\GG),\wh{\omega}\in\LL^1(\whG)$ and apply $\omega\otimes\wh{\omega}\otimes\id$ to \eqref{eq59}:
\begin{equation}\label{eq60}
\phi_{\M}\wh{\theta}_{\GG\subseteq D(\GG)}(
(\omega\otimes\id)\wW^{\GG})\,
\phi_{\M}\wh{\theta}_{\whG\subseteq D(\GG)}(
(\wh\omega\otimes\id)\wW^{\whG})=
\phi_{\M,\GG}((\omega\otimes\id)\wW^{\GG})\,
\phi_{\M,\whG}((\wh\omega\otimes\id)\wW^{\whG}).
\end{equation}
By choosing an appropriate net of functionals $\wh{\omega}_i\in\LL^1(\whG)$, we can assume that $(\wh{\omega}_i\otimes\id)\wW^{\whG}\xrightarrow[i\in I]{}\I$ strictly in $\M(\mrm{C}_0^u(\GG))$ (\cite[Proposition 4.2]{Kustermans}), hence \eqref{eq60} implies
\[
\phi_{\M}\wh{\theta}_{\GG\subseteq D(\GG)}(
(\omega\otimes\id)\wW^{\GG})
=
\phi_{\M,\GG}((\omega\otimes\id)\wW^{\GG})
\]
for an arbitrary $\omega\in \LL^1(\GG)$. By density we can conclude that $\phi_{\M}\wh{\theta}_{\GG\subseteq D(\GG)}=\phi_{\M,\GG}$. Equation $\phi_{\M}\wh{\theta}_{\whG\subseteq D(\GG)}=\phi_{\M,\whG}$ similarly follows from \eqref{eq60}. This proves the first part of the proposition.\\

Next we prove the converse: assume that action $D(\GG)\curvearrowright \M$ is implemented by $U^{\M}$ and $\phi_{\M}$. Define $\phi_{\M,\GG}=\phi_{\M} \wh{\theta}_{\GG\subseteq D(\GG)}$, $\phi_{\M,\whG}=\phi_{\M} \wh{\theta}_{\whG\subseteq D(\GG)}$ and associated representations $U^{\M,\GG},U^{\M,\whG}$. We need to show that these implement the restricted actions $\GG,\whG\curvearrowright \M$, and that they satisfy \eqref{eq44} and \eqref{eq43}.

First, applying $\id\otimes\id\otimes\phi_{\M}$ to \eqref{eq46} we see that $U^{\M}=U^{\M,\GG}_{[13]} U^{\M,\whG}_{[23]}$. By the definition of comultiplication on $D(\GG)$ we have
\[\begin{split}
&\quad\;
U^{\M,\GG}_{[15]}U^{\M,\whG}_{[25]}
U^{\M,\GG}_{[35]}U^{\M,\whG}_{[45]}=
U^{\M}_{[125]}U^{\M}_{[345]}=
(\Delta_{D(\GG)}\otimes\id)U^{\M}=
\ww^{\GG}_{[32]} 
(\Delta_{\GG\times \whG}\otimes\id)(U^{\M})\ww^{\GG *}_{[32]}\\
&=
\ww^{\GG}_{[32]} 
U^{\M,\GG}_{[15]}
U^{\M,\GG}_{[35]}
U^{\M,\whG}_{[25]}
U^{\M,\whG}_{[45]}
\ww^{\GG *}_{[32]}=
U^{\M,\GG}_{[15]}
\ww^{\GG}_{[32]} 
U^{\M,\GG}_{[35]}
U^{\M,\whG}_{[25]}
\ww^{\GG *}_{[32]}
U^{\M,\whG}_{[45]}.
\end{split}\]
Cancelling $U^{\M,\GG}_{[15]},U^{\M,\whG}_{[45]}$ and applying the flip on legs $2,3$ gives us \eqref{eq43}. Finally, we need to check that $\phi_{\M,\GG},\phi_{\M,\whG}$ indeed implement the restricted actions $\alpha^{\M}_{\GG}$, $\alpha^{\M}_{\whG}$. For $m\in\M$ we have
\[\begin{split}
&\quad\;
(\id\otimes\alpha^{\M})\alpha^{\M}_{\GG}(m)=
(\Delta_{\GG,D(\GG)}\otimes\id)\alpha^{\M}(m)\\
&=
(\pi_{\GG}\theta_{\GG\subseteq D(\GG)}\otimes\id)(\Ww^{D(\GG) *})_{[123]}
(\id\otimes \phi_{\M})(\wW^{D(\GG) *})_{[234]} m_{[4]}
(\id\otimes \phi_{\M})(\wW^{D(\GG) })_{[234]}
(\pi_{\GG}\theta_{\GG\subseteq D(\GG)}\otimes\id)(\Ww^{D(\GG) })_{[123]}\\
&=
(\id\otimes \phi_{\M})(\wW^{D(\GG) *})_{[234]}
(\pi_{\GG}\theta_{\GG\subseteq D(\GG)}\otimes \phi_{\M})(\WW^{D(\GG) *})_{[14]}
(\pi_{\GG}\theta_{\GG\subseteq D(\GG)}\otimes\id)(\Ww^{D(\GG) *})_{[123]}
 m_{[4]}\\
 &\quad\quad\quad
 \quad\quad\quad
 \quad\quad\quad
(\pi_{\GG}\theta_{\GG\subseteq D(\GG)}\otimes\id)(\Ww^{D(\GG) })_{[123]}
(\pi_{\GG}\theta_{\GG\subseteq D(\GG)}\otimes \phi_{\M})(\WW^{D(\GG) })_{[14]}
 (\id\otimes \phi_{\M})(\wW^{D(\GG) })_{[234]}\\
 &=
U^{\M *}_{[234]}
(\id \otimes \phi_{\M}\wh{\theta}_{\GG\subseteq D(\GG)})(\wW^{\GG *})_{[14]}
 m_{[4]}
(\id \otimes \phi_{\M}\wh{\theta}_{\GG\subseteq D(\GG)})(\wW^{\GG })_{[14]}
U^{\M}_{[234]},
\end{split}\]
which implies that action $\alpha^{\M}_{\GG}$ is implemented by $\phi_{\M}\wh{\theta}_{\GG\subseteq D(\GG)}$. This proves the claim for $\GG$. The reasoning for $\whG$ is analogous.

\end{proof}

\section{Braided flip operators}\label{sec:braidedflip}

In this section we introduce braided flip operators, and derive their basic properties.\\

Let $\GG,\HH$ be locally compact quantum groups with bicharacter $\wh{\mc{X}}\in \M(\mrm{C}_0(\whH)\otimes\mrm{C}_0(\whG))$. Let $U$ be a representation of $\GG$ on $\msf{H}_U$, with the corresponding morphisms $\phi_U$. Similarly let $V$ be a representation of $\HH$ on $\msf{H}_V$ with morphism $\phi_V$. We will be mostly interested in the situation where representations implement actions of $\GG$ and $\HH$, but the following construction works more generally.

\begin{definition}\label{def1}
We define the \emph{braided flip operator} as the following unitary
\[
\braid{U}{V}=
(\phi_{V}\otimes \phi_{U})(\wh{\mc{X}}^u)\Sigma\colon \msf{H}_{U}\otimes \msf{H}_{V}\rightarrow \msf{H}_{V}\otimes \msf{H}_{U}.
\]
\end{definition}

\begin{remark}
One can also introduce another version of braided flip operators, namely $\braidu{U}{V}=(\braid{V}{U})^*$. These unitaries correspond to bicharacter $\wh{\mc{X}}_{[21]}^*$, so whether we work with $\braid{U}{V}$ or $\braidu{U}{V}$ is a matter of choice. In this paper we will use only $\braid{U}{V}$.
\end{remark}

In the next section we will use operators $\braid{U}{V}$ to introduce the braided tensor product of von Neumann algebras -- this construction works well for a general bicharacter. If $\GG=\HH$ and $\wh{\mc{X}}=\wh{\oon{R}}$ is an $\oon{R}$-matrix, then braided flip operators have additional properties. In the next two results we will introduce the braided structure on the monoidal category $\Rep(\GG)$ of unitary representations of $\GG$.

\begin{lemma}\label{lemma2}
Assume that $\GG=\HH$ and $\wh{\mc{X}}=\wh{\oon{R}}$ is an $\oon{R}$-matrix. Then the unitary operator $\braid{U}{V}$ is a morphism, i.e.
\[
\braid{U}{V}_{[23]}(U\otop V)=
(V\otop U)\;\braid{U}{V}_{[23]}.
\]
\end{lemma}

\begin{proof}
For $\omega \in \LL^1(\GG)$ we have using equation \eqref{eq61}
\[\begin{split}
&\quad\;
( \omega\otimes \id\otimes \id)
(\braid{U}{V}_{[23]}(U\otop V))=
(\phi_{V}\otimes \phi_{U})(\wh{\oon{R}}^u)\Sigma
(\omega\otimes \id\otimes \id)(U_{[12]}V_{[13]})\\
&=
(\phi_{V}\otimes \phi_{U})(\wh{\oon{R}}^u)\Sigma
(\omega\otimes \id\otimes \id)
((\id\otimes \phi_{U})(\wW^{\GG})_{[12]} 
(\id\otimes \phi_{V})(\wW^{\GG})_{[13]} )\\
&=
(\phi_{V}\otimes \phi_{U})(\wh{\oon{R}}^u)\Sigma
(\phi_{U}\otimes \phi_{V})
(\omega\otimes \id\otimes \id)
(\wW^{\GG}_{[12]} \wW^{\GG}_{[13]} )\\
&=
(\phi_{V}\otimes \phi_{U})(\wh{\oon{R}}^u)
(\phi_{V}\otimes \phi_{U})
(\omega\otimes \id\otimes \id)
(\wW^{\GG}_{[13]} \wW^{\GG}_{[12]} )\Sigma\\
&=
(\phi_{V}\otimes \phi_{U})
(\omega\otimes \id\otimes \id)
(\wh{\oon{R}}^{u }_{[23]}(\id\otimes \Delta_{\whG}^{u})(\wW^{\GG}) )\Sigma\\
&=
(\phi_{V}\otimes \phi_{U})
(\omega\otimes \id\otimes \id)
((\id\otimes \Delta_{\whG^{op}}^{u})(\wW^{\GG}) \wh{\oon{R}}^{u }_{[23]})\Sigma\\
&=
(\phi_{V}\otimes \phi_{U})
(\omega\otimes \id\otimes \id)
(\wW^{\GG}_{[12]}\wW^{\GG}_{[13]} )
(\phi_{V}\otimes \phi_{U})(\wh{\oon{R}}^{u })
\Sigma\\
&=
(\omega\otimes \id\otimes\id)(
V_{[12]} U_{[13]} )
\braid{U}{V}=
(\omega\otimes \id\otimes\id)\bigl(
(V\otop U )
\braid{U}{V}_{[23]}\bigr),
\end{split}\]
which proves the claim.
\end{proof}

\begin{proposition}\label{prop2}
Assume that $\GG=\HH$ and $\wh{\mc{X}}=\wh{\oon{R}}$ is an $\oon{R}$-matrix. Then the family of morphisms $\braid{U}{V}\,(U,V\in \Rep(\GG))$ forms a braiding, i.e.
\[
\braid{U}{V\otop W}=
(\id\otimes \braid{U}{W})
(\braid{U}{V}\otimes \id) \colon \msf{H}_U\otimes \msf{H}_V\otimes \msf{H}_W\rightarrow \msf{H}_V\otimes \msf{H}_W\otimes \msf{H}_U
\]
and
\[
\braid{U\otop V}{ W}=
(\braid{U}{W}\otimes \id )
(\id\otimes \braid{V}{W}) \colon \msf{H}_U\otimes \msf{H}_V\otimes \msf{H}_W\rightarrow \msf{H}_W\otimes  \msf{H}_U\otimes \msf{H}_V
\]
for $U,V,W\in \Rep(\GG)$. Together with the usual (trivial) associators and unit $\I\otimes 1$, this turns $\Rep(\GG)$ into a braided monoidal category.
\end{proposition}

Note that the braiding on $\Rep(\GG)$ depends on the $\oon{R}$-matrix, thus more precisely we should write that $\Rep(\GG,\wh{\oon{R}})$ is a braided monoidal category.

\begin{proof}
Let us check the first relation:
\[\begin{split}
&\quad\;
(\id\otimes \braid{U}{W})
(\braid{U}{V}\otimes \id)=
(\id\otimes (\phi_{W}\otimes \phi_{U})(\wh{\oon{R}}^{u })\Sigma)
((\phi_{V}\otimes \phi_{U})(\wh{\oon{R}}^{ u})\Sigma\otimes \id)\\
&=
(\phi_{W}\otimes \phi_{U})(\wh{\oon{R}}^{u })_{[23]} 
(\phi_{V}\otimes \phi_{U})(\wh{\oon{R}}^{u })_{[13]} \Sigma_{[23]}\Sigma_{[12]}=
(\phi_{V}\otimes \phi_{W}\otimes \phi_{U})(\wh{\oon{R}}^{u }_{[23]}  \wh{\oon{R}}^{u }_{[13]}) \Sigma_{[23]}\Sigma_{[12]}\\
&=
(\phi_{V}\otimes \phi_{W}\otimes \phi_{U})(\Delta_{\whG^{op}}^{u}\otimes \id)(\wh{\oon{R}}^{u }) \Sigma_{[23]}\Sigma_{[12]}=
(\phi_{V\otop W}\otimes \phi_{U})(\wh{\oon{R}}^{u }) \Sigma_{[23]}\Sigma_{[12]}=\braid{U}{V\otop W},
\end{split}\]
where we have used $\phi_{V\otop W}=(\phi_{V}\otimes \phi_{W})\Delta_{\whG^{op}}^{u}$. We check the second equality in a similar manner:
\[\begin{split}
&\quad\;
( \braid{U}{W}\otimes \id )
(\id\otimes \braid{V}{W})=
( (\phi_{W}\otimes \phi_{U})(\wh{\oon{R}}^{u })\Sigma\otimes \id )
(\id\otimes (\phi_{W}\otimes \phi_{V})(\wh{\oon{R}}^{u })\Sigma)\\
&=
(\phi_{W}\otimes \phi_{U})(\wh{\oon{R}}^{u })_{[12]} 
(\phi_{W}\otimes \phi_{V})(\wh{\oon{R}}^{u })_{[13]} \Sigma_{[12]}\Sigma_{[23]}=
(\phi_{W}\otimes \phi_{U}\otimes \phi_{V})(\wh{\oon{R}}^{u }_{[12]} \wh{\oon{R}}^{u }_{[13]}) \Sigma_{[12]}\Sigma_{[23]}\\
&=
(\phi_{W}\otimes \phi_{U}\otimes \phi_{V})(\id\otimes \Delta_{\whG^{op}}^{u})(\wh{\oon{R}}^{u }) \Sigma_{[12]}\Sigma_{[23]}=
(\phi_{W}\otimes \phi_{ U\otop V})(\wh{\oon{R}}^{u }) \Sigma_{[12]}\Sigma_{[23]}=\braid{U\otop V}{W}.
\end{split}\]
We already know by Lemma \ref{lemma1} that $\braid{U}{V}$ are morphisms (i.e.\ intertwiners):
\[
\braid{U}{V}_{[23]}(U\otop V)=
(V\otop U)\;\braid{U}{V}_{[23]}.
\]
This implies 
\[
(\braid{U}{V}_{[23]})^*
(V\otop U)=
(U\otop V)
(\braid{U}{V}_{[23]})^*
\]
that is, $(\braid{U}{V})^*$ is an intertwiner $\msf{H}_V\otimes \msf{H}_U\rightarrow\msf{H}_U\otimes \msf{H}_V$. Consequently $\braid{U}{V}$ are isomorphisms in $\Rep(\GG)$. It is left to check that the family $\{\braid{U}{V}\}_{U,V}$ is natural: take equivariant maps $f\colon \msf{H}_U\rightarrow\msf{H}_{U'}, g\colon \msf{H}_V\rightarrow\msf{H}_{V'}$. Then
\[
(g\otimes f)\braid{U}{V}=
(g\otimes f) (\phi_{V}\otimes \phi_{U})(\wh{\oon{R}}^{u })\Sigma
=
 (\phi_{V'}\otimes \phi_{U'})(\wh{\oon{R}}^{u })\Sigma
 (f\otimes g)=
\braid{V'}{U'}(f\otimes g),
\]
which ends the proof.
\end{proof}

\section{Braided tensor product $\M\ov\boxtimes \N$}\label{sec:btp}

\subsection{General setup}

Let $\M\subseteq \B(\msf{H}_{\M}),\N\subseteq\B(\msf{H}_{\N})$ be von Neumann algebras, $\GG,\HH$ locally compact quantum groups and $\wh{\mc{X}}\in \M(\mrm{C}_0(\whH)\otimes \mrm{C}_0(\whG))$ a bicharacter. Assume that we have (left) actions $\alpha^{\M}\colon \HH\curvearrowright \M,\alpha^{\N}\colon \GG\curvearrowright \N$ implemented by $U^{\M}\in \M(\mrm{C}_0(\HH)\otimes \mc{K}(\msf{H}_{\M}))$ and $U^{\N}\in \M(\mrm{C}_0(\GG)\otimes \mc{K}(\msf{H}_{\N}))$. Let $\phi_{\M},\phi_{\N}$ be the associated $*$-homomorphisms, i.e.~$U^{\M}=(\id\otimes \phi_{\M})(\wW^{\HH})$, $U^{\N}=(\id\otimes\phi_{\N})(\wW^{\GG})$; see Section \ref{sec:preliminaries}.

Out of this data we will construct a new von Neumann algebra $\M\ov\boxtimes \N$ together with canonical embeddings $\iota_{\M},\iota_{\N}$ of $\M,\N$ (Theorem \ref{thm1}). This construction does not depend on the way we represent $\M,\N$ or implement the actions (Proposition \ref{prop3}). If $\wh{\mc{X}}$ is an $\oon{R}$-matrix, then $\M\ov\boxtimes\N$ carries a canonical action such that $\iota_{\M},\iota_{\N}$ are equivariant -- in fact, a more general result is true, see Proposition \ref{prop20}.\\

Recall that in the previous section we have introduced unitary maps $\braid{U^{\N}}{U^{\M}}\colon \msf{H}_{\N}\otimes \msf{H}_{\M}\rightarrow \msf{H}_{\M}\otimes \msf{H}_{\N}$. It will be convenient to change notation; understanding that we have fixed implementations of actions, we will write $\braid{\N}{\M}=\braid{U^{\N}}{U^{\M}}$.

\begin{definition}\label{def2}
Define normal, injective, unital $*$-homomorphisms
\[\begin{split}
\iota_{\M}\colon \M\ni m \mapsto 
\iota_{\M}(m)&=
\braid{\N}{\M}(\I\otimes m)(\braid{\N}{\M})^*\\
&=
(\phi_{\M}\otimes\phi_{\N})(\wh{\mc{X}}^u)
(m\otimes\I)
(\phi_{\M}\otimes\phi_{\N})(\wh{\mc{X}}^u)^*\in 
\B(\msf{H}_{\M}\otimes\msf{H}_{\N}),\\
\iota_{\N}\colon \N\ni n \mapsto 
\iota_{\N}(n)&=\I\otimes n \in \B(\msf{H}_{\M}\otimes \msf{H}_{\N}),
\end{split}\]
and $\swot$-closed subspace
\begin{equation}\label{eq3}
\M\ov\boxtimes\N
=
\ov{\lin}^{\,\swot}\{\iota_{\M}(m)\iota_{\N}(n)\mid
m\in\M,n\in\N\}\subseteq \B(\msf{H}_{\M}\otimes \msf{H}_{\N}).
\end{equation}
The space $\M\ov\boxtimes\N$ is called the \emph{braided tensor product} of $\M$ and $\N$.
\end{definition}

\begin{remark}\noindent
\begin{enumerate}
\item The space $\M\ov\boxtimes\N$ depends not only on algebras $\M,\N$, but also on the actions and the bicharacter $\wh{\mc{X}}$. Most of the time this data will be clear from the context, otherwise we will indicate it in the notation.
\item We obtain the same space $\M\ov\boxtimes\N$ if in \eqref{eq3} we take closure in different topologies; $\wot$ or $\ssots$ -- this follows from \cite[Theorem 2.6 (iv)]{TakesakiI} and the double commutant theorem together with Theorem \ref{thm1}.
\end{enumerate}
\end{remark}

The main result of this section is Theorem \ref{thm1}, which states that $\M\ov\boxtimes\N$ is a von Neumann algebra. While it is not too dificult to prove this statement for dual actions, it will take us some work to show in general. We first establish several preliminary results, some of them of independent interest. We begin with a useful embedding $\M\ov\boxtimes\N\hookrightarrow  \B(\LL^2(\HH))\bar\otimes \M\bar\otimes\B(\LL^2(\GG)) \bar\otimes \N$, and show that the braided tensor product does not depend on the choice of implementations.

\begin{proposition}\label{prop4}
In $\B(\LdH\otimes \msf{H}_{\M}\otimes  \LdG\otimes\msf{H}_{\N})$ we have
\[
 U^{\N *}_{[34]}\wh{\mc{X}}_{[13]}U^{\M * }_{[12]}
(\iota_{\M}(m)\iota_{\N}(n))_{[24]}
U^{\M }_{[12]} \wh{\mc{X}}_{[13]}^*U^{\N }_{[34]} =
\wh{\mc{X}}_{[13]} 
\alpha^{\M}(m)_{[12]}\wh{\mc{X}}_{[13]}^*
\alpha^{\N}(n)_{[34]}
\]
for any $m\in\M,n\in\N$.
\end{proposition}

\begin{proof}
Take $m\in\M,n\in\N$ and calculate
\begin{equation}\begin{split}\label{eq5}
&\quad\;
 U^{\N *}_{[34]} \wh{\mc{X}}_{[13]} U^{\M *}_{[12]}
(\iota_{\M}(m)\iota_{\N}(n))_{[24]}
U^{\M }_{[12]} \wh{\mc{X}}_{[13]}^* U^{\N }_{[34]}
\\
&=
U^{\N *}_{[34]}\wh{\mc{X}}_{[13]} U^{\M * }_{[12]}
(\phi_{\M}\otimes \phi_{\N})(\wh{\mc{X}}^u)_{[24]} m_{[2]} 
(\phi_{\M}\otimes \phi_{\N})(\wh{\mc{X}}^u)^*_{[24]}
n_{[4]} 
U^{\M }_{[12]}\wh{\mc{X}}_{[13]}^*U^{\N }_{[34]}\\
&=
U^{\N *}_{[34]}\wh{\mc{X}}_{[13]} U^{\M * }_{[12]}
(\phi_{\M}\otimes \phi_{\N})(\wh{\mc{X}}^u)_{[24]} m_{[2]} 
(\phi_{\M}\otimes \phi_{\N})(\wh{\mc{X}}^u)^*_{[24]}
U^{\M }_{[12]}\wh{\mc{X}}_{[13]}^*U^{\N }_{[34]}
\alpha^{\N}(n)_{[34]}.
\end{split}\end{equation}
Observe that
\begin{equation}\label{eq4}\begin{split}
(\Delta_{\whH}\otimes \id)\wh{\mc{X}}^{r,u}=
\wh{\mc{X}}^{r,u}_{[13]}\wh{\mc{X}}^{r,u}_{[23]}=
\ww^{\whH *}_{[12]} \wh{\mc{X}}^{r,u}_{[23]} \ww^{\whH}_{[12]}\quad&\Rightarrow\quad
\wh{\mc{X}}^{r,u}_{[23]}\wh{\mc{X}}^{r,u}_{[13]}=
\ww^{\HH }_{[12]} \wh{\mc{X}}_{[13]}^{r,u} \ww^{\HH *}_{[12]}\\\quad&\Rightarrow\quad
\ww^{\HH *}_{[12]}\wh{\mc{X}}^{r,u}_{[23]}=
 \wh{\mc{X}}_{[13]}^{r,u} \ww^{\HH *}_{[12]} \wh{\mc{X}}^{r,u *}_{[13]}.
\end{split}\end{equation}
Applying to this equation $\id\otimes \Delta_{\whH}^{r,u}\otimes \id $ we obtain
\[
\ww^{\HH *}_{[12]}\wW^{\HH *}_{[13]} \wh{\mc{X}}^{r,u}_{[24]}
\wh{\mc{X}}_{[34]}^u=
\wh{\mc{X}}^{r,u}_{[14]} \ww^{\HH *}_{[12]}
\wW^{\HH *}_{[13]} \wh{\mc{X}}^{r,u *}_{[14]}=
\wh{\mc{X}}^{r,u}_{[14]} \ww^{\HH *}_{[12]}
\wh{\mc{X}}^{r,u *}_{[14]}\wh{\mc{X}}^{r,u}_{[14]}
\wW^{\HH *}_{[13]} \wh{\mc{X}}^{r,u *}_{[14]}
\]
hence using again \eqref{eq4}
\[
\ww^{\HH *}_{[12]}\wW^{\HH *}_{[13]} \wh{\mc{X}}^{r,u}_{[24]}
\wh{\mc{X}}_{[34]}^u=
\ww^{\HH *}_{[12]} \wh{\mc{X}}^{r,u}_{[24]}
\wh{\mc{X}}^{r,u}_{[14]}
\wW^{\HH *}_{[13]} \wh{\mc{X}}^{r,u *}_{[14]}\quad\Rightarrow\quad
\wW^{\HH *}_{[13]}
\wh{\mc{X}}_{[34]}^u=
\wh{\mc{X}}^{r,u}_{[14]}
\wW^{\HH *}_{[13]} \wh{\mc{X}}^{r,u *}_{[14]}
\]
or $\wW^{\HH *}_{[12]} \wh{\mc{X}}_{[23]}^u=
\wh{\mc{X}}^{r,u}_{[13]} \wW^{\HH *}_{[12]} \wh{\mc{X}}^{r,u *}_{[13]}.
$
Using this we have 
\[\begin{split}
&\quad\;
U^{\M * }_{[12]} (\phi_{\M}\otimes \phi_{\N})(\wh{\mc{X}}^u)_{[24]}=
(\id\otimes \phi_{\M}\otimes \phi_{\N})(\wW^{\HH *}_{[12]} \wh{\mc{X}}^{u }_{[23]})_{[124]}=
(\id\otimes \phi_{\M}\otimes \phi_{\N})(
\wh{\mc{X}}^{r,u}_{[13]} \wW^{\HH *}_{[12]} \wh{\mc{X}}^{r,u *}_{[13]}
)_{[124]}\\
&=
(\id\otimes\phi_{\N})(\wh{\mc{X}}^{r,u})_{[14]}
U^{\M *}_{[12]}
(\id\otimes\phi_{\N})(\wh{\mc{X}}^{r,u})^*_{[14]}
\end{split}\]
Plugging this twice in \eqref{eq5} we obtain
\[\begin{split}
&\quad\;
U^{\N * }_{[34]}\wh{\mc{X}}_{[13]}
U^{\M * }_{[12]}
(\iota_{\M}(m)\iota_{\N}(n))_{[24]}
U^{\M}_{[12]}\wh{\mc{X}}_{[13]}^* U^{\N }_{[34]}\\
&=
U^{\N *}_{[34]}\wh{\mc{X}}_{[13]} 
(\id\otimes\phi_{\N})(\wh{\mc{X}}^{r,u})_{[14]}
U^{\M *}_{[12]}
(\id\otimes\phi_{\N})(\wh{\mc{X}}^{r,u})^*_{[14]}
 m_{[2]} 
(\id\otimes\phi_{\N})(\wh{\mc{X}}^{r,u})_{[14]}\\
&\quad\quad\quad\quad\quad\quad
\quad\quad\quad\quad\quad\quad\quad\quad\quad
\quad\quad\quad\quad\quad\quad
U^{\M }_{[12]}
(\id\otimes\phi_{\N})(\wh{\mc{X}}^{r,u})^*_{[14]}
\wh{\mc{X}}_{[13]}^*U^{\N }_{[34]}
\alpha^{\N}(n)_{[34]}\\
&=
U^{\N *}_{[34]}\wh{\mc{X}}_{[13]} 
(\id\otimes\phi_{\N})(\wh{\mc{X}}^{r,u})_{[14]}
\alpha^{\M}(m)_{[12]}
(\id\otimes\phi_{\N})(\wh{\mc{X}}^{r,u})^*_{[14]}
\wh{\mc{X}}_{[13]}^*U^{\N }_{[34]}
\alpha^{\N}(n)_{[34]}
\end{split}\]
Next observe that 
\[\begin{split}
&\quad\;
U^{\N*}_{[34]} \wh{\mc{X}}_{[13]}
\,
(\id\otimes \phi_{\N})(\wh{\mc{X}}^{r,u})_{[14]}=
(\id\otimes \id\otimes \phi_{\N})
(\wW^{\GG *}_{[23]} \wh{\mc{X}}_{[12]} \wh{\mc{X}}^{r,u}_{[13]})_{[134]}=
(\id\otimes\phi_{\N}\otimes\id)
(\wW^{\whG }_{[23]} \wh{\mc{X}}_{[13]} \wh{\mc{X}}^{u,r}_{[12]})_{[143]}\\
&=
(\id\otimes\phi_{\N}\otimes\id)
(\Ww^{\whG }_{[23]} \Ww^{\whG *}_{[23]} \wh{\mc{X}}_{[13]} \Ww^{\whG}_{[23]})_{[143]}=
\wh{\mc{X}}_{[13]} (\phi_{\N}\otimes \id)(\Ww^{\whG} )_{[43]}=
\wh{\mc{X}}_{[13]} U^{\N *}_{[34]},
\end{split}\]
hence
\[\begin{split}
&\quad\;
U^{\N * }_{[34]}\wh{\mc{X}}_{[13]}
U^{\M * }_{[12]}
(\iota_{\M}(m)\iota_{\N}(n))_{[24]}
U^{\M}_{[12]}\wh{\mc{X}}_{[13]}^* U^{\N }_{[34]}=
\wh{\mc{X}}_{[13]} U^{\N *}_{[34]}
\alpha^{\M}(m)_{[12]}
U^{\N }_{[34]}\wh{\mc{X}}_{[13]}^*
\alpha^{\N}(n)_{[34]}\\
&=
\wh{\mc{X}}_{[13]} 
\alpha^{\M}(m)_{[12]}\wh{\mc{X}}_{[13]}^*
\alpha^{\N}(n)_{[34]}
\end{split}\]
as claimed.
\end{proof}

\subsection{Independence of implementations}\label{sec:Independence}\noindent

In this subsection we show that $\M\ov\boxtimes\N$ is, up to an isomorphism, independent of all ``choices'' or implementations, including passing to isomorphic quantum groups. Let us introduce the setting.

Let $\GG_1,\HH_1$ be locally compact quantum groups and $\wh{\mc{X}}_1\in \LL^{\infty}(\whH_1)\bar\otimes\LL^{\infty}(\whG_1)$ a bicharacter. Let $\alpha^{\M_1}\colon \HH_1\curvearrowright \M_1,\alpha^{\N_1}\colon \GG_1\curvearrowright \N_1$ be left actions on von Neumann algebras $\M_1\subseteq\B(\msf{H}_{\M_1}),\N_1\subseteq\B(\msf{H}_{\N_1})$ implemented by $U^{\M_1}=(\id\otimes\phi_{\M_1})\wW^{\HH_1},U^{\N_1}=(\id\otimes\phi_{\N_1})\wW^{\GG_1}$. Next, let $\GG_2,\HH_2$ be locally compact quantum groups which are isomorphic with $\GG_1,\HH_1$. More precisely, assume that there are $*$-isomorphisms\footnote{We write $\theta$ instead of $\pi$ to avoid confusion with the reducing map.} $\theta_{\GG}\colon \mrm{C}_0^u(\GG_1)\rightarrow \mrm{C}_0^u(\GG_2)$, $\theta_{\HH}\colon \mrm{C}_0^u(\HH_1)\rightarrow \mrm{C}_0^u(\HH_2)$ which respect comultiplications. These isomorphism have reduced versions; $*$-isomorphisms $\theta_{\GG,r}\colon \mrm{C}_0(\GG_1)\rightarrow \mrm{C}_0(\GG_2)$, $\theta_{\HH,r}\colon \mrm{C}_0(\HH_1)\rightarrow \mrm{C}_0(\HH_2)$ satisfying $\pi_{\GG_2}\theta_{\GG}=\theta_{\GG,r}\pi_{\GG_1}$, $\pi_{\HH_2}\theta_{\HH}=\theta_{\HH,r}\pi_{\HH_1}$. We also have von Neumann algebraic versions: normal $*$-isomorphisms $\gamma_{\GG}\colon \LL^{\infty}(\GG_1)\rightarrow \LL^{\infty}(\GG_2)$, $\gamma_{\HH}\colon \LL^{\infty}(\HH_1)\rightarrow \LL^{\infty}(\HH_2)$ which agree with $\theta_{\GG,r},\theta_{\HH,r}$ on $\mrm{C}_0(\GG_1),\mrm{C}_0(\HH_1)$. See \cite[Theorem 1.10, Theorem 3.3]{DKSS_ClosedSub}. The isomorphism $\gamma_{\GG}$ is implemented by the standard unitary $v_{\GG}\colon \LL^2(\GG_1)\rightarrow \LL^2(\GG_2)$, i.e.~the unique unitary implementing $\gamma_{\GG}$: $\gamma_{\GG}(x)=v_{\GG} x v_{\GG}^*\,(x\in \LL^{\infty}(\GG_1))$, respecting modular conjugation $v_{\GG}J_{\GG_1}=J_{\GG_2}v_{\GG}$ and standard positive cone $v_{\GG}\mf{P}_{\vp_{\GG_1}}=\mf{P}_{\vp_{\GG_2}}$ -- see \cite[Theorem 2.3]{Haagerup} and Appendix \ref{sec:AppendixImplementation}. Similarly $\gamma_{\HH}$ is implemented by the standard unitary $v_{\HH}$. The corresponding dual maps will be decorated with hats, e.g.~$\wh{\theta}_{\GG,r}\colon \mrm{C}_0(\whG_2)\rightarrow \mrm{C}_0(\whG_1)$. We have $\wh{v}_{\GG}=v_{\GG}^*, \wh{v}_{\HH}=v_{\HH}^*$ (Proposition \ref{prop5}). Let $\M_2\subseteq \B(\msf{H}_{\M_2}),\N_2\subseteq\B(\msf{H}_{\N_2})$ be von Neumann algebras isomorphic with $\M_1,\N_1$ via $\theta_{\M}\colon \M_1\rightarrow\M_2, \theta_{\N}\colon\N_1\rightarrow\N_2$.

Assume that there is a bicharacter $\wh{\mc{X}}_2\in \LL^{\infty}(\whH_2)\bar\otimes\LL^{\infty}(\whG_2)$ and actions $\alpha^{\M_2}\colon \HH_2\curvearrowright \M_2,\alpha^{\N_2}\colon \GG_2\curvearrowright \N_2$, and that this structure is respected by the above isomorphisms: $(\wh{\theta}_{\HH}\otimes \wh{\theta}_{\GG})\wh{\mc{X}}_2^u=\wh{\mc{X}}_1^u$ and 
\[
\alpha^{\M_2} \theta_{\M}=
(\gamma_{\GG}\otimes \theta_{\M})\alpha^{\M_1},\quad
\alpha^{\N_2} \theta_{\N}=
(\gamma_{\HH}\otimes \theta_{\N})\alpha^{\N_1}.
\]
In this situation we can form two braided tensor products, $\M_1\ov\boxtimes\N_1\subseteq \B(\msf{H}_{\M_1}\otimes\msf{H}_{\N_1})$ and $\M_2\ov\boxtimes\N_2\subseteq \B(\msf{H}_{\M_2}\otimes\msf{H}_{\N_2})$. 

\begin{proposition}\label{prop3}
\noindent
\begin{enumerate}
\item There is a unique completely isometric isomorphism $\Upsilon\colon \M_1\ov\boxtimes\N_1\rightarrow \M_2\ov\boxtimes\N_2$ which is a {\swot} homeomorphism and satisfies $\Upsilon(\iota_{\M_1}(m)\iota_{\N_1}(n))=\iota_{\M_2}(\theta_{\M}(m))\iota_{\N_2}(\theta_{\N}(n))$ for $m\in\M_1,n\in\N_1$.
\item The space $\M_1\ov\boxtimes\N_1$ is closed under multiplication if and only if it is closed under adjoints. The same is true for $\M_2\ov\boxtimes \N_2$.
\item If one of the spaces $\M_1\ov\boxtimes\N_1,\M_2\ov\boxtimes \N_2$ is closed under multiplication, then so is the other. In this case $\Upsilon$ is a $*$-isomorphism.
\end{enumerate}
\end{proposition}

 The main usefulness of this result comes from the fact that quite often it is not difficult to find ``an'' implementation, but it's more work to show that it is the standard one. This proposition is also an important ingredient in the proof of Theorem \ref{thm1}.
 
 \begin{proof}
For any $m\in\M_1,n\in\N_1$, we have by Proposition \ref{prop4}
\begin{equation}\label{eq6}
U^{\N_1 *}_{[34]}\wh{\mc{X}}_{1 [13]} U^{\M_1 *}_{[12]}
(\iota_{\M_1}(m)\iota_{\N_1}(n))_{[24]}
 U^{\M_1 }_{[12]}\wh{\mc{X}}_{1 [13]}^*
U^{\N_1 }_{[34]}=
\wh{\mc{X}}_{1 [13]} 
\alpha^{\M_1 }(m)_{[12]}\wh{\mc{X}}_{1 [13]}^*
\alpha^{\N_1 }(n)_{[34]}
\end{equation}
in $\B( \LL^2(\HH_1)\otimes \msf{H}_{\M_1 }\otimes\LL^2(\GG_1)\otimes  \msf{H}_{\N_1 })$ and
\begin{equation}\label{eq7}
U^{\N_2 *}_{[34]}\wh{\mc{X}}_{2 [13]} U^{\M_2 *}_{[12]}
(\iota_{\M_2}(\theta_{\M}(m))\iota_{\N_2}(\theta_{\N}(n)))_{[24]}
 U^{\M_2 }_{[12]}\wh{\mc{X}}_{2 [13]}^*
U^{\N_2 }_{[34]}=
\wh{\mc{X}}_{2 [13]} 
\alpha^{\M_2 }(\theta_{\M}(m))_{[12]}\wh{\mc{X}}_{2 [13]}^*
\alpha^{\N_2 }(\theta_{\N}(n))_{[34]}
\end{equation}
in $\B(\LL^2(\HH_2)\otimes \msf{H}_{\M_2}\otimes \LL^2(\GG_2)\otimes \msf{H}_{\N_2})$.

Consider three mappings 
\[
\Upsilon_1\colon \M_1\ov{\boxtimes}\N_1 \ni \mathbf{X}\mapsto
U^{\N_1 *}_{[34]}\wh{\mc{X}}_{1 [13]} U^{\M_1 *}_{[12]}
\mathbf{X}_{[24]}
 U^{\M_1 }_{[12]}\wh{\mc{X}}_{1 [13]}^*
U^{\N_1 }_{[34]}
\in \B( \LL^2(\HH_1)\otimes \msf{H}_{\M_1 }\otimes\LL^2(\GG_1)\otimes  \msf{H}_{\N_1 }),
\]
next
\[
\oon{Ad}(v_{\HH}) \otimes\theta_{\M}\otimes   \Ad(v_{\GG})\otimes \theta_{\N}\colon \B(\LL^2(\HH_1))\bar\otimes \M_1\bar\otimes \B(\LL^2(\GG_1))\bar\otimes\N_1 \rightarrow 
\B(\LL^2(\HH_2))\bar\otimes \M_2\bar\otimes \B(\LL^2(\GG_2))\bar\otimes\N_2
\]
and finally
\[
\Upsilon_{2}=\oon{Ad}\bigl( 
U^{\M_2}_{[12]}
 \wh{\mc{X}}_{2 [13]}^*
U^{\N_2 }_{[34]}
 \bigr)\in \Aut\bigl( 
\B(\LL^2(\HH_2)\otimes \msf{H}_{\M_2} \otimes \LL^2(\GG_2)\otimes \msf{H}_{\N_2 }) \bigr).
\]
Clearly all these maps are completely isometric and $\swot$-continuous. Equation \eqref{eq6} shows that image of $\Upsilon_1$ lies in $\B(\LL^2(\HH_1))\bar\otimes \M_1\bar\otimes \B(\LL^2(\GG_1))\bar\otimes \N_1$ -- hence we can consider $\Upsilon_1$ as a map
\[
\Upsilon_1\colon \M_1\ov\boxtimes\N_1\rightarrow
\B(\LL^2(\HH_1))\bar\otimes \M_1\bar\otimes \B(\LL^2(\GG_1))\bar\otimes \N_1.
\]
Before we go further we need to make another observation. By our assumption we have $(\wh{\theta}_{\HH}\otimes \wh{\theta}_{\GG})\wh{\mc{X}}_2^u=\wh{\mc{X}}_1^u$. Applying reducing morphisms gives
\begin{equation}\label{eq8}
\wh{\mc{X}}_1= (\wh{\theta}_{\HH, r}\otimes \wh{\theta}_{\GG,r})\wh{\mc{X}}_2=(\wh{\gamma}_{\HH}\otimes\wh{\gamma}_{\GG})\wh{\mc{X}}_2=
(\wh{v}_{\HH}\otimes \wh{v}_{\GG})\wh{\mc{X}}_2 (\wh{v}_{\HH}^*\otimes \wh{v}_{\GG}^*)=
(v^*_{\HH}\otimes v^*_{\GG})\wh{\mc{X}}_2 (v_{\HH}\otimes v_{\GG}).
\end{equation}
Consider the composition
\[
\Upsilon_3= \Upsilon_2
( \Ad(v_{\HH})\otimes \theta_{\M}\otimes \Ad(v_{\GG})\otimes \theta_{\N})\Upsilon_1\colon \M_1\ov\boxtimes\N_1\rightarrow 
\B(\LL^2(\HH_2)\otimes \msf{H}_{\M_2} \otimes \LL^2(\GG_2)\otimes \msf{H}_{\N_2 }).
\]
It is completely isometric and $\swot$-continuous, as a composition of such maps. Furthermore using \eqref{eq6}, \eqref{eq8} and \eqref{eq7} we have
\begin{equation}\begin{split}\label{eq9}
&\quad\;
\Upsilon_3(\iota_{\M_1}(m)\iota_{\N_1}(n))=
\Upsilon_2 (\Ad(v_{\HH})\otimes \theta_{\M}\otimes \Ad(v_{\GG})\otimes \theta_{\N})
\bigl( 
\wh{\mc{X}}_{1 [13]} 
\alpha^{\M_1 }(m)_{[12]}\wh{\mc{X}}_{1 [13]}^*
\alpha^{\N_1 }(n)_{[34]}
\bigr)\\
&=
\Upsilon_2\bigl(
\wh{\mc{X}}_{2 [13]} 
(\gamma_{\HH}\otimes \theta_{\M})\alpha^{\M_1 }(m)_{[12]}
\wh{\mc{X}}_{2 [13]}^*
(\gamma_{\GG}\otimes \theta_{\N})\alpha^{\N_1 }(n)_{[34]}
\bigr)\\
&=
\Upsilon_2\bigl(
\wh{\mc{X}}_{2 [13]} 
\alpha^{\M_2 }(\theta_{\M}(m))_{[12]}
\wh{\mc{X}}_{2 [13]}^*
\alpha^{\N_2 }(\theta_{\N}(n))_{[34]}
\bigr)=
(\iota_{\M_1}(\theta_{\M}(m))\iota_{\N_2}(\theta_{\N}(n)))_{[24]}.
\end{split}\end{equation}
Thus the image of $\Upsilon_3$ lies in $\CC\bar\otimes \B(\msf{H}_{\M_2})\bar\otimes \CC\bar\otimes \B(\msf{H}_{\N_2})$. Let $\Upsilon_4\colon \M_1\ov\boxtimes\N_1\rightarrow \B(\msf{H}_{\M_2}\otimes \msf{H}_{\N_2})$ be the map $\Upsilon_3$ composed with the $*$-isomorphism $\CC\bar\otimes \B(\msf{H}_{\M_2})\bar\otimes \CC\bar\otimes \B(\msf{H}_{\N_2})\simeq \B(\msf{H}_{\M_2})\bar\otimes \B(\msf{H}_{\N_2})$. Then \eqref{eq9} gives
\[
\Upsilon_4(\iota_{\M_1}(m)\iota_{\N_1}(n))=
\iota_{\M_2}(\theta_{\M}(m))\iota_{\N_2}(\theta_{\N}(n)).
\]
In particular (by \swot-continuity and closedness) we can consider $\Upsilon_4$ as a map
\[
\Upsilon_4\colon \M_1\ov\boxtimes\N_1\rightarrow \M_2\ov\boxtimes\N_2
\]
which is \swot-continuous and completely isometric.

Reversing this construction, i.e.~starting with $\M_2\ov\boxtimes\N_2$, we obtain a completely isometric, \swot-continuous map $\tilde{\Upsilon}_4\colon \M_2\ov\boxtimes\N_2\rightarrow \M_1\ov\boxtimes\N_1$ which is given by $
\tilde{\Upsilon}_4(\iota_{\M_2}(\theta_{\M}(m))\iota_{\N_2}(\theta_{\N}(n)))=
\iota_{\M_1}(m)\iota_{N_1}(n)$. Clearly $\Upsilon_4$ and $\tilde{\Upsilon}_4$ are each other inverses. Uniqueness is also clear. This proves the first claim.\\

Since $(\iota_{\M_1}(m) \iota_{\N_1}(n))^*=\iota_{\N_1}(n^*)\iota_{\M_1}(m^*)$, we have that if $\M_1\ov\boxtimes\N_1$ is closed under multiplication, it is also closed under adjoints. Conversely, assume $\M_1\ov\boxtimes\N_1$ is closed under adjoints. Since multiplication in $\B(\msf{H}_{\M}\otimes \msf{H}_{\N}
)$ is separately \swot-continuous, we will have that $\M_1\ov\boxtimes\N_1$ is closed under multiplication if we can show that it contains all products $(\iota_{\M_1}(m) \iota_{\N_1}(n))(\iota_{\M_1}(\wt{m}) \iota_{\N_1}(\wt{n}))$ with $m,\wt{m}\in \M_1,n,\wt{n}\in\N_1$. But
\[
(\iota_{\M_1}(m) \iota_{\N_1}(n))(\iota_{\M_1}(\wt{m}) \iota_{\N_1}(\wt{n}))=
\iota_{\M_1}(m) (\iota_{\M_1}(\wt{m}^* )\iota_{\N_1}(n^*))^* \iota_{\N_1}(\wt{n}),
\]
by assumption $(\iota_{\M_1}(\wt{m}^* )\iota_{\N_1}(n^*))^*\in \M_1\ov\boxtimes\N_1$ and clearly $\M_1\ov\boxtimes\N_1$ is closed under left-multiplication by $\iota_{\M_1}(\M_1)$ and right multiplication by $\iota_{\N_1}(\N_1)$. The proof for $\M_2\ov\boxtimes\N_2$ is analogous. This proves the second point.\\

Finally, assume that $\M_1\ov\boxtimes\N_1$ is a von Neumann algebra. By construction, $\Upsilon_4\colon \M_1\ov\boxtimes\N_1\rightarrow \B(\msf{H}_{\M_2}\otimes \msf{H}_{\N_2})$ is $*$-preserving (it is a composition of maps which are clearly $*$-preserving). Take $x=\Upsilon_4(y)\in\M_2\ov\boxtimes\N_2$. Then we have
\[
x^*=\Upsilon_4(y)^*=\Upsilon_4(y^*)\in \Upsilon_4(\M_1\ov\boxtimes\N_1)=\M_2\ov\boxtimes\N_2.
\]
This shows that $\M_2\ov\boxtimes \N_2$ is closed under adjoint. The second point shows that it is also closed under multiplication, hence it is a von Neumann algebra. An analogous argument (using $\tilde{\Upsilon}_4$) proves the converse. In this situation, $\Upsilon_4$ is a $*$-homomorphism by construction.
 \end{proof}

\subsection{Universal lift of an action $\GG\curvearrowright \M$}\label{sec:VNUniversal}

In the proof of Theorem \ref{thm1} we will need an auxilliary construction, which can be thought of as a lift of an action to the universal level.

In what follows, we will use the notion of enveloping von Neumann algebra. Recall that if $\mc{A}$ is a $\cst$-algebra, then the second dual $\mc{A}^{**}$ has the structure of a von Neumann algebra, $\mc{A}\subseteq \mc{A}^{**}$ is a weak$^*$-dense $*$-subalgebra and $\mc{A}^{**}$ has the following universal property: for every von Neumann algebra $\oon{P}\subseteq \B(\msf{H}_{\oon{P}})$ and a non-degenerate $*$-homomorphism $\pi\colon \mc{A}\rightarrow \B(\msf{H}_{\oon{P}})$ with $\pi(\mc{A})\subseteq \oon{P}$, there exists a unique extension of $\pi$ to a normal, unital $*$-homomorphism $\pi^{\vN}\colon \mc{A}^{**}\rightarrow \oon{P}$ (\cite[Theorem 3.7.7, Proposition 3.7.8]{Pedersen}). Note also that $\M(\mc{A})$ can be identified with the two sided multipliers of the image of $\mc{A}$ in any faithful, non-degenerate representation (\cite[Proposition 3.12.3]{Pedersen}), in particular $\M(\mc{A})\subseteq \mc{A}^{**}$.
\begin{remark}\label{remark1}
Let us note that if $\mc{A}\rightarrow \M(\mc{B})$ is a non-degenerate $*$-homomorphism, then  the normal extension $\mc{A}^{**}\rightarrow \mc{B}^{**}$ and the strict extension $\M(\mc{A})\rightarrow \M(\mc{B})\subseteq \mc{B}^{**}$ agree on $\M(\mc{A})\subseteq \mc{A}^{**}$.
\end{remark}

Let $\GG$ be a locally compact quantum group. We will use a von-Neumann-algebraic version of half-lifted comultiplications. Recall that we have a non-degenerate $*$-homomorphism
\[
\Delta_{\GG}^{u,r} \colon \mrm{C}_0(\GG)\ni x \mapsto 
\Ww^{\GG *}(\I\otimes x) \Ww^{\GG}\in \M(\mrm{C}_0^u(\GG)\otimes \mrm{C}_0(\GG)).
\]
Since $\Ww^{\GG}\in\M(\mrm{C}_0^u(\GG)\otimes\mrm{C}_0(\whG))$ can be seen as an element of $\mrm{C}_0^{u}(\GG)^{**}\bar\otimes \LL^{\infty}(\whG)$, we have a natural extension of $\Delta_{\GG}^{u,r}$ to a normal, unital $*$-homomorphism
\[
\Delta_{\GG}^{u,r,\vN} \colon \LL^{\infty}(\GG)\ni x \mapsto 
\Ww^{\GG *}(\I\otimes x) \Ww^{\GG}\in \mrm{C}_0^{u}(\GG)^{**}\bar\otimes \LL^{\infty}(\GG)
\]
(it follows from $w^*$-density of $\mrm{C}_0(\GG)$ in $\LL^{\infty}(\GG)$ that the co-domain is correct). Similarly, $\Delta_{\GG}^{r,u}$ admits an extension to a normal, unital $*$-homomorphism
\[
\Delta_{\GG}^{r,u,\vN} \colon \LL^{\infty}(\GG)\ni x \mapsto 
\vV^{\GG }(x\otimes \I) \vV^{\GG *}\in \LL^{\infty}(\GG)\bar\otimes \mrm{C}_0^{u}(\GG)^{**}.
\]

Next assume that $\GG$ acts on a von Neumann algebra $\M\subseteq \B(\msf{H}_{\M})$ and $U^{\M}=(\id\otimes \phi_{\M})\wW^{\GG}\in \M(\mrm{C}_0(\GG)\otimes \mc{K}(\msf{H}_{\M}))$ is a representation implementing the action $\alpha^{\M}$. As mentioned above, we can extend uniquely the (non-degenerate) $*$-homomorphism $\phi_{\M}\colon \mrm{C}_0^u(\whG)\rightarrow \B(\msf{H}_{\M})$ to a normal, unital $*$-homomorphism $\mrm{C}_0^u(\whG)^{**}\rightarrow \B(\msf{H}_{\M})$. We denote this extension by $\phi_{\M}^{\vN}$. We will use the same notation also for other normal maps given by the universal property of the enveloping von Neumann algebra.\\

Define the normal, unital, injective $*$-homomorphism\footnote{Note that $(\id\otimes\phi_{\M})(\WW^{\GG})$ is an element of $\M(\mrm{C}_0^u(\GG)\otimes \mc{K}(\msf{H}_{\M}))\subseteq \mrm{C}_0^u(\GG)^{**}\bar\otimes \B(\msf{H}_{\M})$, so the multiplication \eqref{eq10} makes sense in $\mrm{C}_0^u(\GG)^{**}\bar\otimes \B(\msf{H}_{\M})$. We can also see $\WW^{\GG}$ as an element of $\mrm{C}_0^u(\GG)^{**}\bar\otimes \mrm{C}_0^u(\whG)^{**}$ but we have $(\id\otimes\phi_{\M})(\WW^{\GG})=(\id\otimes\phi_{\M}^{\vN})(\WW^{\GG})$, by Remark \ref{remark1}.}
\begin{equation}\label{eq10}
\alpha^{\M,u}\colon \M\ni m\mapsto 
(\id\otimes\phi_{\M})(\WW^{\GG})^*
(\I\otimes m)
(\id\otimes\phi_{\M})(\WW^{\GG })
\in \mrm{C}_0^u(\GG)^{**}\bar\otimes \M.
\end{equation}
We can think of $\alpha^{\M,u}$ as a universal lift of the action $\alpha^{\M}\colon \GG\curvearrowright \M$. Note that a priori, the above map lands in $\mrm{C}_0^u(\GG)^{**}\bar\otimes \B(\msf{H}_{\M})$, but in the next two lemmas we show that $\alpha^{\M,u}$ indeed has the above codomain, does not depend on the choice of implementation and satisfies a version of the action condition.

\begin{lemma}\label{lemma3}
For $\omega\in\LL^1(\GG),m\in\M$ we have
\begin{equation}\label{eq11}
(\id\otimes \phi_{\M})(\WW^{\GG})^*
(\I\otimes (\omega\otimes\id)\alpha^{\M}(m))
(\id\otimes \phi_{\M})(\WW^{\GG })=
\bigl( (\omega\otimes \id )\Delta^{r,u,\vN}_{\GG}\otimes \id \bigr)
\alpha^{\M}(m)
\end{equation}
in $\mrm{C}^{u}_0(\GG)^{**}\bar\otimes\B(\msf{H}_{\M})$. Furthermore
\[
\alpha^{\M,u}(m)=(\id\otimes \phi_{\M})(\WW^{\GG})^*
(\I\otimes m)
(\id\otimes \phi_{\M})(\WW^{\GG })\in \mrm{C}_0^u(\GG)^{**}\bar\otimes \M.
\]
\end{lemma}

\begin{proof}
The desired equation follows from a straightforward calculation:
\[\begin{split}
&\quad\;
(\id\otimes \phi_{\M})(\WW^{\GG *})
(\I\otimes (\omega\otimes\id)\alpha^{\M}(m))
(\id\otimes \phi_{\M})(\WW^{\GG })\\
&=
(\id\otimes \omega\otimes \id)\bigl(
(\id\otimes \phi_{\M} )(\WW^{\GG *})_{[13]}
\alpha^{\M}(m)_{[23]}
(\id\otimes \phi_{\M})(\WW^{\GG })_{[13]}
\bigr)\\
&=
(\id\otimes \omega\otimes \id)\bigl(
(\id\otimes \phi_{\M} )(\WW^{\GG *})_{[13]}
(\id\otimes \phi_{\M})(\wW^{\GG *})_{[23]}
m_{[3]}
(\id\otimes \phi_{\M})(\wW^{\GG })_{[23]}
(\id\otimes \phi_{\M})(\WW^{\GG })_{[13]}
\bigr)\\
&=
(\id\otimes \omega\otimes \id)\bigl(
(\Delta_{\GG}^{r,u}\otimes \phi_{\M} )(\wW^{\GG *})_{[213]}
m_{[3]}
(\Delta_{\GG}^{r,u}\otimes \phi_{\M} )(\wW^{\GG })_{[213]}\bigr)\\
&=
(\omega\otimes \id \otimes \id)\bigl(
(\Delta_{\GG}^{r,u}\otimes \phi_{\M} )(\wW^{\GG *})
m_{[3]}
(\Delta_{\GG}^{r,u}\otimes \phi_{\M} )(\wW^{\GG })\bigr)\\
&=
(\omega\otimes \id \otimes \id)
(\Delta_{\GG}^{r,u,\vN}\otimes \id)
\bigl(
(\id\otimes \phi_{\M} )(\wW^{\GG *})
m_{[2]}
(\id\otimes \phi_{\M} )(\wW^{\GG })\bigr)\\
&=
((\omega\otimes \id)\Delta_{\GG}^{r,u,\vN}\otimes \id)
\alpha^{\M}(m).
\end{split}\]
Recall (\cite[Proposition 2.9]{QGProjection}) that in the von Neumann algebraic context the Podleś condition
\[
\overline{\lin}^{\,\swot} \alpha^{\M}(\M)(\LL^{\infty}(\GG)\otimes \I) =\LL^{\infty}(\GG)\bar\otimes \M
\]
is always satisfied. Consequently, $\M=\ov{\lin}^{\,\swot}\{(\rho\otimes \id)\alpha^{\M}(m)\,|\, \rho\in \LL^1(\GG),m\in\M\}$. The last claim easily follows. 
\end{proof}

\begin{lemma}\label{lemma4}
The map $\alpha^{\M,u}$ satisfies
\begin{enumerate}
\item $(\id\otimes \alpha^{\M})\alpha^{\M,u}=(\Delta_{\GG}^{u,r,\vN}\otimes \id )\alpha^{\M}$,
\item $(\id\otimes\alpha^{\M,u})\alpha^{\M}=(\Delta^{r,u,\vN}_{\GG}\otimes \id)\alpha^{\M}$.
\end{enumerate}
Moreover, $\alpha^{\M,u}$ does not depend on the choice of implementation $U^{\M}$.
\end{lemma}

\begin{proof}
Take $m\in\M$ and calculate
\[\begin{split}
&\quad\;
(\id\otimes \alpha^{\M})\alpha^{\M,u}(m)=
(\id\otimes \phi_{\M})(\wW^{\GG *})_{[23]} 
(\id\otimes \phi_{\M})(\WW^{\GG *})_{[13]} m_{[3]} 
(\id\otimes \phi_{\M})(\WW^{\GG})_{[13]}
(\id\otimes \phi_{\M})(\wW^{\GG})_{[23]} \\
&=
( \Delta_{\GG}^{u,r}\otimes \phi_{\M})(\wW^{\GG *})
m_{[3]}
( \Delta_{\GG}^{u,r}\otimes \phi_{\M})(\wW^{\GG })=
(\Delta_{\GG}^{u,r,\vN}\otimes \id )\alpha^{\M}(m)
\end{split}\]
which proves the first claim. Since $\id\otimes \alpha^{\M}$ is injective, the last claim follows. Finally, we have
\[\begin{split}
&\quad\;
(\id\otimes \alpha^{\M,u})\alpha^{\M}(m)=
(\id\otimes \phi_{\M})(\WW^{\GG *})_{[23]}
(\id\otimes \phi_{\M})(\wW^{\GG *})_{[13]}
m_{[3]}
(\id\otimes \phi_{\M})(\wW^{\GG })_{[13]}
(\id\otimes \phi_{\M})(\WW^{\GG })_{[23]}\\
&=
( \Delta_{\GG}^{r,u}\otimes \phi_{\M})(\wW^{\GG *})
m_{[3]}
( \Delta_{\GG}^{r,u}\otimes \phi_{\M})(\wW^{\GG })=
(\Delta^{r,u,\vN}_{\GG}\otimes \id)\alpha^{\M}(m).
\end{split}\]
\end{proof}

We will also need a version of the Podleś condition for the universal lift $\alpha^{\M,u}$.

\begin{proposition}\label{prop7}
For any action $\alpha^{\M}\colon \GG\curvearrowright \M$ we have $\ov{\lin}^{\,w^*} (\mrm{C}^{u}_0(\GG)^{**}\otimes\I)\alpha^{\M,u}(\M)=\mrm{C}_0^u(\GG)^{**}\bar\otimes \M$.
\end{proposition}

The proof of this proposition is a technical modification of the proof of \cite[Proposition 2.9]{QGProjection}.

\begin{proof}
First we show the equality (c.f.~\cite[Equality (2.9)]{QGProjection} and \cite[Proposition 3.6]{BaajSkandalis})
\begin{equation}\label{eq28}
\ov{\lin}^{\,w^*} \{(x\otimes \I) \Ww^{\GG}(\I\otimes b)\mid 
x\in \mrm{C}_0^u(\GG)^{**},b\in \B(\LL^2(\GG))\}=
\mrm{C}_0^u(\GG)^{**}\bar\otimes \B(\LL^2(\GG)).
\end{equation}
We calculate the left hand side as follows, using in the 3rd equality that $\ww^{\whG }$ is a unitary in $\B(\LL^2(\GG))\bar\otimes \B(\LL^2(\GG))$:
\begin{equation}\begin{split}\label{eq30}
&\quad\;
\ov{\lin}^{\,w^*} \{(x\otimes \I) \Ww^{\GG}(\I\otimes b)\mid 
x\in \mrm{C}_0^u(\GG)^{**},b\in \B(\LL^2(\GG))\}\\
&=
\ov{\lin}^{\,w^*} \{
(\id\otimes \id\otimes \omega ) (
\Ww^{\GG}_{[13]}  \Ww^{\GG}_{[12]} b_{[2]})\mid 
\omega\in \B(\LL^2(\GG))_*,b\in \B(\LL^2(\GG))\}\\
&=
\ov{\lin}^{\,w^*} \{
(\id\otimes \id\otimes \omega ) (
\ww^{\whG *}_{[23]}\Ww^{\GG}_{[13]}\ww^{\whG}_{[23]}
 (b\otimes a)_{[23]})\mid 
\omega\in \B(\LL^2(\GG))_*,a,b\in \B(\LL^2(\GG))\}\\
&=
\ov{\lin}^{\,w^*} \{
(\id\otimes \id\otimes \omega ) (
\ww^{\whG *}_{[23]}\Ww^{\GG}_{[13]}
 (b\otimes a)_{[23]})\mid 
\omega\in \B(\LL^2(\GG))_*,a,b\in \B(\LL^2(\GG))\}\\
&=
\ov{\lin}^{\,w^*} \{
(\id\otimes \id\otimes \omega ) (
a_{[3]} \ww^{\whG *}_{[23]}
b_{[2]}\Ww^{\GG}_{[13]})
\mid 
\omega\in \B(\LL^2(\GG))_*,a,b\in \B(\LL^2(\GG))\}\\
&=
\ov{\lin}^{\,w^*} \{
(\id\otimes  \omega\otimes \id ) (
a_{[2]} \ww^{\GG }_{[23]}
b_{[3]}\Ww^{\GG}_{[12]})
\mid 
\omega\in \B(\LL^2(\GG))_*,a,b\in \B(\LL^2(\GG))\}.
\end{split}\end{equation}
Now, recall that regularity at the von Neumann algebraic level is always satisfied:
\[
\ov{\lin}^{\,\swot} \,\{
(\id\otimes\omega)(\Sigma \ww^{\GG})\mid \omega\in\B(\LdG)_*\}=
\B(\LdG),
\]
where $\Sigma$ is the flip map (see the last equation in the proof of \cite[Proposition 2.9]{BaajSkandalisVaes} and \cite[Proposition 2.5]{VaesVanDaele}). Consequently
\[\begin{split}
&\quad\;
\ov{\lin}^{\,\swot}\,
\{(a\otimes \I)\ww^{\GG} (\I\otimes b)\mid a,b\in \B(\LdG)\}\\
&=
\ov{\lin}^{\,\swot}\,
\{(|\xi\ra\la\eta |\otimes \I)\ww^{\GG} (\I\otimes |\zeta\ra\la \varkappa| )\mid \xi,\eta,\zeta,\varkappa\in \LdG\}\\
&=
\ov{\lin}^{\,\swot}\,
\{\Sigma (\I\otimes |\xi\ra\la\eta |)(\Sigma \ww^{\GG}) (\I\otimes |\zeta\ra\la \varkappa| )\mid \xi,\eta,\zeta,\varkappa\in \LdG\}\\
&=
\ov{\lin}^{\,\swot}\,
\{\Sigma ((\id\otimes \omega )(\Sigma \ww^{\GG})\otimes a)
\mid \omega\in\B(\LdG)_*,a\in \B(\LdG)\}\\
&=
\Sigma (\B(\LdG)\bar\otimes\B(\LdG))=
\B(\LdG)\bar\otimes\B(\LdG).
\end{split}\]
With this we continue \eqref{eq30}:
\[\begin{split}
&\quad\;
\ov{\lin}^{\,w^*} \{(x\otimes \I) \Ww^{\GG}(\I\otimes b)\mid 
x\in \mrm{C}_0^u(\GG)^{**},b\in \B(\LL^2(\GG))\}\\
&=
\ov{\lin}^{\,w^*} \{
(\id\otimes  \omega\otimes \id ) (
(a\otimes b)_{[23]}\Ww^{\GG}_{[12]})
\mid 
\omega\in \B(\LL^2(\GG))_*,a,b\in \B(\LL^2(\GG))\}\\
&=
\ov{\lin}^{\,w^*} \{
(\id\otimes  \omega)(\Ww^{\GG})\otimes b
\mid 
\omega\in \B(\LL^2(\GG))_*,b\in \B(\LL^2(\GG))\}=
\mrm{C}_0^u(\GG)^{**}\bar\otimes \B(\LdG)
\end{split}\]
which proves \eqref{eq28}. Consider the map $R_{\GG}^{u}$ acting on $\mrm{C}_0^u(\GG)$ and its double dual, denoted $R_{\GG}^{u,\vN}$. It is the unique normal extension of $R_{\GG}^u$ to $\mrm{C}_0^u(\GG)^{**}\rightarrow \mrm{C}_0^u(\GG)^{**}$ and it is bijective, linear, $*$-preserving and antimultiplicative. Applying (see Remark \ref{remark2}) $R_{\GG}^{u,\vN}\otimes J_{\whG}(\cdot)^* J_{\whG}$ and the flip to \eqref{eq28} we obtain
\[
\ov{\lin}^{\,w^*} \{ (b\otimes \I) (R_{\GG}^u \otimes j_{\whG})(\Ww^{\GG})_{[21]}(\I\otimes x )\mid 
x\in \mrm{C}_0^u(\GG)^{**},b\in \B(\LL^2(\GG))\}=
\B(\LL^2(\GG))\bar\otimes \mrm{C}_0^u(\GG)^{**}.
\]
Since $(R_{\GG}^u\otimes j_{\whG})(\Ww^{\GG})_{[21]}=\vV^{\GG}$, applying the adjoint gives 
\begin{equation}\label{eq29}
\ov{\lin}^{\,w^*} \{ (\I\otimes x )\vV^{\GG *}(b\otimes \I) \mid 
x\in \mrm{C}_0^u(\GG)^{**},b\in \B(\LL^2(\GG))\}=
\B(\LL^2(\GG))\bar\otimes \mrm{C}_0^u(\GG)^{**}.
\end{equation}
Using the weak Podleś condition for $\alpha^{\M}$ (\cite[Corollary 2.7]{QGProjection}), Lemma \ref{lemma4} and equation \eqref{eq29}, we can prove the Podleś condition for $\alpha^{\M,u}$:
\[\begin{split}
&\quad\;
\ov{\lin}^{\, w^*} (\mrm{C}_0^u(\GG)^{**}\otimes \I)
\alpha^{\M,u}(\M)\\
&=
\ov{\lin}^{\, w^*} \{
(x\otimes \I)\alpha^{\M,u}( (\omega\otimes\id)\alpha^{\M}(m))\mid
x\in \mrm{C}_0^u(\GG)^{**},\omega\in \B(\LL^2(\GG))_*, m\in\M\}\\
&=
\ov{\lin}^{\, w^*} \{
(\omega\otimes\id\otimes \id)(
x_{[2]}
(\Delta^{r,u,\vN}_{\GG}\otimes \id)\alpha^{\M}(m))\mid
x\in \mrm{C}_0^u(\GG)^{**},\omega\in \B(\LL^2(\GG))_*, m\in\M\}\\
&=
\ov{\lin}^{\, w^*} \{
(\omega\otimes\id\otimes \id)(
(b\otimes x)_{[12]}\vV^{\GG}_{[12]} \alpha^{\M}(m)_{[13]}
\vV^{\GG *}_{[12]})
\mid
b\in\B(\LL^2(\GG)),x\in \mrm{C}_0^u(\GG)^{**},\omega\in \B(\LL^2(\GG))_*, m\in\M\}\\
&=
\ov{\lin}^{\, w^*} \{
(\omega\otimes\id\otimes \id)(
 \alpha^{\M}(m)_{[13]}
x_{[2]}
\vV^{\GG *}_{[12]}b_{[1]})
\mid
b\in\B(\LL^2(\GG)),x\in \mrm{C}_0^u(\GG)^{**},\omega\in \B(\LL^2(\GG))_*, m\in\M\}\\
&=
\ov{\lin}^{\, w^*} \{
(\omega\otimes\id\otimes \id)(
 \alpha^{\M}(m)_{[13]}
(b\otimes x)_{[12]}
\mid
b\in\B(\LL^2(\GG)),x\in \mrm{C}_0^u(\GG)^{**},\omega\in \B(\LL^2(\GG))_*, m\in\M\}\\
&=
\ov{\lin}^{\, w^*} \{
x\otimes (\omega\otimes\id)
 \alpha^{\M}(m)
\mid
x\in \mrm{C}_0^u(\GG)^{**},\omega\in \B(\LL^2(\GG))_*, m\in\M\}=
\mrm{C}_0^u(\GG)^{**}\bar\otimes \M.
\end{split}\]
\end{proof}

\subsection{$\M\ov\boxtimes\N$ is a von Neumann algebra}

As before, let $\GG,\HH$ be locally compact quantum groups with bicharacter $\wh{\mc{X}}\in \LL^{\infty}(\whH)\bar\otimes \LL^{\infty}(\whG)$ and actions on von Neumann algebras $\alpha^{\M}\colon \HH\curvearrowright \M,\alpha^{\N}\colon\GG \curvearrowright \N$. Represent $\M\subseteq \B(\msf{H}_{\M}), \N\subseteq \B(\msf{H}_{\N})$ and let the actions be implemented by $U^{\M}=(\id\otimes\phi_{\M})\wW^{\HH},U^{\N}=(\id\otimes\phi_{\N})\wW^{\GG}$. Lemma \ref{lemma1} gives us a non-degenerate $*$-homomorphism $\Phi\colon \mrm{C}_0^u(\GG)\rightarrow \M(\mrm{C}_0^u(\whH))$ which commutes with comultiplications, such that $\wh{\mc{X}}^u=(\Phi\otimes \id)\WW^{\GG}$. The main result of this section is the following theorem.

\begin{theorem}\label{thm1}
$\M\ov\boxtimes\N\subseteq \B(\msf{H}_{\M}\otimes \msf{H}_{\N})$ is a von Neumann algebra.
\end{theorem}

We note that Proposition \ref{prop3} shows that $\M\ov\boxtimes\N$ does not depend on the choices of implementations (only on the actions $\HH\curvearrowright\M, \GG\curvearrowright\N$ and the bicharacter $\wh{\mc{X}}$). Before we prove Theorem \ref{thm1}, we need to establish several auxilliary results. First, using the bicharacter $\mc{X}$ we introduce an action $\whG\curvearrowright \M$.

Consider the dual morphism to $\Phi$, which is a non-degenerate $*$-homomorphism $\wh\Phi\colon \mrm{C}_0^u(\HH)\rightarrow \M(\mrm{C}_0^u(\whG))\subseteq \mrm{C}_0^u(\whG)^{**}$, uniquely characterised by $(\Phi\otimes \id)\WW^{\GG}=(\id\otimes\wh{\Phi})\WW^{\whH}$ (see \cite[Corollary 4.3]{Homomorphisms}). By the universal property of the double dual (see Section \ref{sec:VNUniversal}), we can extend $\wh\Phi$ to a unital normal $*$-homomorphism $\wh\Phi^{\vN}\colon \mrm{C}_0^u(\HH)^{**}\rightarrow \mrm{C}_0^u(\whG)^{**}$. Similarly, we extend the reducing map to a unital normal $*$-homomorphism $\pi_{\whG}^{\vN}\colon \mrm{C}_0^u(\whG)^{**}\rightarrow \LL^{\infty}(\whG)$. Next define
\[
\alpha^{\M}_{\whG}=( \pi_{\whG}^{\vN} \wh\Phi^{\vN}\otimes \id)\alpha^{\M,u}\colon \M\rightarrow \LL^{\infty}(\whG)\bar\otimes \M.
\]
By construction, $\alpha^{\M}_{\whG}$ is a unital normal $*$-homomorphism. 

\begin{proposition}\label{prop6}
We have that $\alpha^{\M}_{\whG}$ defines an action $\whG\curvearrowright \M $ and 
\[
\alpha^{\M}_{\whG}(m)=(\pi_{\whG} \wh{\Phi}\otimes \phi_{\M})(\WW^{\HH })^*(\I\otimes m)
(\pi_{\whG} \wh{\Phi}\otimes \phi_{\M})(\WW^{\HH })
\]
for $m\in\M$.
\end{proposition}

\begin{proof}
Seeing $(\id\otimes \phi_{\M})(\WW^{\HH})$ as an element of $ \mrm{C}_0^u(\HH)^{**}\bar\otimes \B(\msf{H}_{\M})$ or $\M( \mrm{C}_0^u(\HH)\otimes \mc{K}(\msf{H}_{\M}))$, we have by Remark \ref{remark1}
\[
( \pi_{\whG}^{\vN}\wh{\Phi}^{\vN}\otimes \id )
((\id\otimes \phi_{\M})(\WW^{\HH}))=( \pi_{\whG}\wh{\Phi}\otimes \phi_{\M})(\WW^{\HH})
\]
(to see this, apply an arbitrary normal functional to the right leg) and
\[\begin{split}
\alpha^{\M}_{\whG}(m)&=
( \pi_{\whG}^{\vN} \wh{\Phi}^{\vN}\otimes \id )\bigl(
(\id\otimes \phi_{\M})(\WW^{\HH})^* (\I\otimes m)
(\id\otimes \phi_{\M})(\WW^{\HH})\bigr)\\
&=
(\pi_{\whG} \wh{\Phi}\otimes \phi_{\M})(\WW^{\HH})^*
(\I\otimes m)
(\pi_{\whG} \wh{\Phi}\otimes \phi_{\M})(\WW^{\HH})
\end{split}\]
for $m\in\M$. This also shows that $\alpha^{\M}_{\whG}$ is injective. Next we check the action condition: for an arbitary $m\in\M$ we calculate
\[\begin{split}
&\quad\;
(\Delta_{\whG}\otimes \id )\alpha^{\M}_{\whG}(m)\\
&=
( \Delta_{\whG}\otimes \id )
\bigl(
( \pi_{\whG} \wh{\Phi}\otimes \phi_{\M})(\WW^{\HH})^* (\I\otimes m)
( \pi_{\whG} \wh{\Phi}\otimes \phi_{\M})(\WW^{\HH})
\bigr)\\
&=
(\pi_{\whG}\wh{\Phi}\otimes \pi_{\whG}\wh{\Phi}\otimes \phi_{\M})(\WW^{\HH}_{[13]}\WW^{\HH}_{[23]})^*
(\I\otimes \I\otimes m)
(\pi_{\whG}\wh{\Phi}\otimes \pi_{\whG}\wh{\Phi}\otimes \phi_{\M})(\WW^{\HH}_{[13]}\WW^{\HH}_{[23]})\\
&=
(\pi_{\whG}\wh{\Phi}\otimes \phi_{\M})(\WW^{\HH})^*_{[23]}
(\pi_{\whG}\wh{\Phi}\otimes \phi_{\M})(\WW^{\HH})^*_{[13]}
(\I\otimes \I\otimes m)
(\pi_{\whG}\wh{\Phi}\otimes \phi_{\M})(\WW^{\HH})_{[13]}
(\pi_{\whG}\wh{\Phi}\otimes \phi_{\M})(\WW^{\HH})_{[23]}\\
&=
(\id\otimes \alpha^{\M}_{\whG})\alpha^{\M}_{\whG}(m).
\end{split}\]
\end{proof}

Next we record a general lemma, which we will use in the case of action $\alpha^{\M}_{\whG}$. It is well known in the community, but we record a short proof for the convenience of the reader (c.f.~\cite[Equation (2.3)]{QGProjection}).

\begin{lemma}\label{lemma7}
Let $\KK$ be a locally compact quantum group acting on a von Neumann algebra $\oon{P}$ via $\alpha^{\oon{P}}\colon \KK\curvearrowright \oon{P}$. We have
\[
\B(\LL^2(\KK))\bar\otimes \oon{P}=
\ov{\lin}^{\,\swot}\{(T\otimes \I)\alpha^{\oon{P}}(p)\mid p\in\oon{P},T\in \B(\LL^2(\KK))\}.
\]
\end{lemma}

\begin{proof}
Clearly we have $\supseteq$, we need to show the converse inclusion $\subseteq$. The Podleś condition for $\alpha^{\oon{P}}$ (\cite[Proposition 2.9]{QGProjection}) reads
\begin{equation}\label{eq13}
\LL^{\infty}(\KK)\bar\otimes \oon{P} =
\ov{\lin}^{\,\swot}\{
(x\otimes \I)\alpha^{\oon{P}}(p)\mid x\in \LL^{\infty}(\KK),m\in\oon{P}\}.
\end{equation}
Take $T\in \B(\LL^2(\KK)),p\in\oon{P}$. By \cite[Proposition 2.5]{VaesVanDaele} we can choose elements $x_{i,k}\in\LL^{\infty}(\KK),\wh{x}_{i,k}\in \LL^{\infty}(\wh\KK)\,(i\in I,1\le k\le K_i)$ such that $\sum_{k=1}^{K_i}\wh{x}_{i,k} x_{i,k}\xrightarrow[i\in  I]{}T $, with convergence in \swot. Next, for $i\in I,1\le k\le K_i$  we can by \eqref{eq13} find $x_{i,k,j,l}\in \LL^{\infty}(\KK), p_{i,k,j,l}\in \oon{P}\,(j\in J_{i,k},1\le l\le L_{i,k,j})$ so that $\sum_{l=1}^{L_{i,k,j}} (x_{i,k,j,l}\otimes \I)\alpha^{\oon{P}}(p_{i,k,j,l})\xrightarrow[j\in J_{i,k}]{}x_{i,k}\otimes p$ again in {\swot}. We obtain
\[\begin{split}
T\otimes p&=
\underset{i\in I}{\swot \textsc{-}\!\lim} \sum_{k=1}^{K_i} 
(\wh{x}_{i,k}\otimes \I)(x_{i,k}\otimes p)=
\underset{i\in I}{\swot\textsc{-}\! \lim}
\sum_{k=1}^{K_i} 
\underset{j\in J_{i,k}}{\swot\textsc{-}\!\lim}
\sum_{l=1}^{L_{i,k,j}}
(\wh{x}_{i,k}x_{i,k,j,l}\otimes \I)\alpha^{\oon{P}}(p_{i,k,j,l}),
\end{split}\]
which ends the proof.
\end{proof}

Our proof of Theorem \ref{thm1} will rely on the biduality theorem \cite[Theorem 2.6]{VaesUnitary}, which we will use to reduce the general situation to the case of dual actions. Let us recall its content in the case of the action $\alpha^{\N}\colon \GG\curvearrowright \N$.

Consider the crossed product von Neumann algebra, defined as
\[
\N^{\ltimes}=\GG\ltimes \N=\ov{\lin}^{\,\swot}\{(\wh x\otimes \I)\alpha^{\N}(n)\mid \wh x\in\LL^{\infty}(\whG),n\in\N\}\subseteq \B(\LL^2(\GG))\bar \otimes \N
\]
(see \cite[Definition 2.1, Lemma 3.3]{VaesUnitary}). On this von Neumann algebra we have an action $\alpha^{\N^{\ltimes}}\colon \whG^{op}\curvearrowright \N^{\ltimes}$ uniquely determined by 
\[
\alpha^{\N^{\ltimes}}( \alpha^{\N}(n))=\I\otimes \alpha^{\N}(n),\quad 
\alpha^{\N^{\ltimes}}(\wh x\otimes \I)=
\Delta_{\whG^{op}}(\wh x)\otimes \I,\quad
(n\in\N, \wh{x}\in \Linfd).
\]
 It is implemented by
 \begin{equation}\label{eq14}
 \ww^{\whG^{op}}\otimes \I=(\id\otimes (\pi_{\GG'}\otimes \I))\wW^{\whG^{op}}=
 \LL^{\infty}(\whG^{op})\bar\otimes \B(\LL^2(\GG)\otimes \msf{H}_{\N})
\end{equation}
 (\cite[Proposition 2.2]{VaesUnitary}, see also \cite[Section 4]{KustermansVaesVN}). The action $\alpha^{\N^{\ltimes}}$ is sometimes called the \emph{dual action}. We can consider the bidual crossed product $\N^{\ltimes \ltimes}=\whG^{op}\ltimes (\GG\ltimes \N)\subseteq \B(\LL^2(\GG)\otimes\LL^2(\GG))\bar\otimes \N$, which is equipped with the dual action $\alpha^{\N^{\ltimes \ltimes}}\colon (\whG^{op})^{\wedge op}\curvearrowright \N^{\ltimes\ltimes}$. Observe that on $\B(\LL^2(\GG))\bar\otimes \N$ we have the obvious $\GG$-action $\alpha^{\B(\LL^2(\GG))\bar\otimes \N}=(\chi\otimes \id)(\id\otimes \alpha^{\N})$ (where $\chi$ is the flip map). The biduality theorem states that
\[
\Psi\colon \B(\LL^2(\GG))\bar\otimes \N\ni 
z \mapsto 
(\ww^{\GG} \otimes \I)(\id\otimes\alpha^{\N})(z)
(\ww^{\GG *}\otimes \I)
\in 
\N^{\ltimes \ltimes}
\]
is a well defined isomorphism. Furthermore, $\vv^{\GG *}_{[21]}\otimes \I$ is a $\alpha^{\B(\LL^2(\GG))\bar\otimes \N}$-cocycle, hence we can define the action
\[
\alpha^{\B(\LL^2(\GG)\bar\otimes \N}_{+}=
\Ad(\vv^{\GG *}_{[21]}\otimes \I)\alpha^{\B(\LL^2(\GG)\bar\otimes \N}\colon \GG\curvearrowright \B(\LL^2(\GG))\bar\otimes \N
\]
which is isomorphic to the bidual action $\alpha^{\N^{\ltimes\ltimes}}$ in the sense that
\begin{equation}\label{eq18}
\alpha^{\N^{\ltimes \ltimes}}(\Psi(z))=
(j_{\GG}R_{\GG}\otimes \Psi)(\alpha^{\B(\LL^2(\GG))\bar\otimes \N}_+(z))
\quad(z\in \B(\LL^2(\GG))\bar\otimes \N).
\end{equation}
Note that $(\whG^{op})^{\wedge op}={\GG^{op}} '$ and $\oon{Ad}(u_{\GG})=j_{\GG} R_{\GG} \colon \LL^{\infty}(\GG)\rightarrow\LL^{\infty}({\GG^{op}} ')$ is the von Neumann algebraic version of isomorphism $\GG\simeq {\GG^{op}} '$ (see \cite[Section 4]{KustermansVaesVN} and Section \ref{sec:preliminaries}).

In order to work with the actions $\alpha^{\B(\LL^2(\GG))\bar\otimes\N}, \alpha^{\B(\LL^2(\GG))\bar\otimes\N}_+$ we need their concrete implementations.

\begin{lemma}\label{lemma5}\noindent
\begin{enumerate}
\item The representation $U^{\N}_{[13]}\in \LL^{\infty}(\GG)\bar\otimes \B(\LL^2(\GG)\otimes \msf{H}_{\N})$ implements $\alpha^{\B(\LL^2(\GG))\bar\otimes \N}$ and the corresponding map is $\phi_{\B(\LL^2(\GG))\bar\otimes \N}=\I\otimes \phi_{\N}$.
\item 
The representation $U^{\N }_{[13]}\vv^{\GG }_{[21]}\in \Linf\bar\otimes \B(\LL^2(\GG)\otimes \msf{H}_{\N})$ implements $\alpha^{\B(\LL^2(\GG))\bar\otimes \N}_+$ and the corresponding map is $\phi_{\B(\LL^2(\GG))\bar\otimes \N,+}=(j_{\whG} R_{\whG}\pi_{\whG}\otimes \phi_{\N})\Delta_{\whG}^u$.
\end{enumerate}
\end{lemma}

\begin{proof}
Take $x\in \B(\LdG)\bar\otimes\N$. We have
\[
\alpha^{\B(\LL^2(\GG))\bar\otimes \N}(x)=
(\id\otimes \alpha^{\N})(x)_{[213]}=
U^{\N *}_{[13]} x_{[23]}U^{\N}_{[13]},
\]
hence $U^{\N}_{[13]}=(\id\otimes (\I\otimes \phi_{\N}))\wW^{\GG}$ implements $\alpha^{\B(\LL^2(\GG))\bar\otimes \N}$. Moreover,
\[
\alpha^{\B(\LdG)\bar\otimes \N}_+(x)=
\vv^{\GG *}_{[21]}
U^{\N *}_{[13]} x_{[23]} U^{\N}_{[13]}\vv^{\GG }_{[21]}
\]
so the unitary $U^{\N }_{[13]}\vv^{\GG }_{[21]}$ implements the action $\alpha^{\B(\LdG)\bar\otimes \N}_+$. We need to check that it is equal to $(\phi_{\M\bar\otimes \B(\LdG),+}\otimes \id)(\Vv^{\GG})$ (which will also show that it is a representation). This follows from a straightforward calculation:
\[\begin{split}
&\quad\;
U^{\N}_{[13]} \vv^{\GG}_{[21]}=
U^{\N}_{[13]} (R_{\GG}\otimes j_{\whG})(\ww^{\GG})_{[12]}=
(\id\otimes\phi_{\N})(\wW^{\GG})_{[13]} (\id \otimes j_{\whG}R_{\whG}\pi_{\whG})(\wW^{\GG})_{[12]}\\
&=
(\id\otimes (j_{\whG}R_{\whG}\pi_{\whG}\otimes\phi_{\N})\Delta_{\whG}^u)\wW^{\GG}=
(\id\otimes \phi_{\B(\LdG)\bar\otimes\N,+})\wW^{\GG}.
\end{split}\]
\end{proof}

\begin{proof}[Proof of Theorem \ref{thm1}]
We divide the proof into several claims. First we prove the theorem under the assumption that the action on $\N$ is the dual action.

\emph{Claim 1.} If $(\N,\alpha^{\N})=(\whG'\ltimes \wt{\N},\alpha^{\wt{\N}^{\ltimes }})$ is the dual action for some left action $\alpha^{\wt \N}\colon \whG'\curvearrowright \wt{\N}$, then $\M\ov\boxtimes\N$ is a von Neumann algebra.

\emph{Proof of Claim 1}. By Proposition \ref{prop3} we can freely use the implementation of the action $\alpha^{\N}$ on $\LL^2(\GG)\otimes \msf{H}_{\wt\N}$ as in \eqref{eq14}, hence 
\[\begin{split}
\M\ov\boxtimes\N &=
\ov{\lin}^{\,\swot}\iota_{\M}(\M)\iota_{\N}(\N)\\
&=
\ov{\lin}^{\,\swot}
\{
(\phi_{\M}\otimes \id)(\wh{\mc{X}}^{u,r})_{[12]} ((m\otimes \I)\otimes \I)
(\phi_{\M}\otimes \id)(\wh{\mc{X}}^{u,r})^*_{[12]}
(\I\otimes n)\mid m\in\M,n\in\N\}\\
&\subseteq \B(\LL^2(\M)\otimes \LdG\otimes \LL^2(\wt\N)).
\end{split}\]
 Take $m\in\M,\eta,\zeta\in\LdG$, let $x=(\id\otimes \omega_{\eta,\zeta})(\ww^{\GG *})$ be the associated element of $\LL^{\infty}(\GG)=\LL^{\infty}(\GG^{op})$ and let $\{\xi_\lambda\}_{\lambda\in \Lambda}$ be an orthonormal basis in $\LdG$. Observe that
 \[
\wh{\mc{X}}^{u,r *}_{[12]} \ww^{\GG *}_{[23]}=
(\wh{\mc{X}}^{u,r *}_{[13]} \ww^{\whG}_{[23]})_{[132]}=
( \ww^{\whG}_{[23]} \wh{\mc{X}}^{u,r *}_{[12]}\wh{\mc{X}}^{u,r *}_{[13]} )_{[132]}=
\ww^{\GG *}_{[23]} \wh{\mc{X}}^{u,r *}_{[13]}
\wh{\mc{X}}^{u,r *}_{[12]}.
\]
Using this twice, we can write $\iota_{\M}(m)\iota_{\N}(x\otimes \I)$ in a different way:
\begin{equation}\begin{split}\label{eq15}
&\quad\;
\iota_{\M}(m)\iota_{\N}(x\otimes \I)=
(\phi_{\M}\otimes \id)(\wh{\mc{X}}^{u,r})_{[12]} (m\otimes \I\otimes \I)
(\phi_{\M}\otimes \id)(\wh{\mc{X}}^{u,r})^*_{[12]}
(\id\otimes \omega_{\eta,\zeta})(\ww^{\GG *})_{[2]}\\
&=
(\phi_{\M}\otimes \id)(\wh{\mc{X}}^{u,r})_{[12]} m_{[1]}
(\phi_{\M}\otimes \id\otimes \omega_{\eta,\zeta})
(\wh{\mc{X}}^{u,r *}_{[12]} \ww^{\GG *}_{[23]})_{[12]}
\\
&=
(\phi_{\M}\otimes \id)(\wh{\mc{X}}^{u,r})_{[12]} m_{[1]}
(\phi_{\M}\otimes \id\otimes \omega_{\eta,\zeta})
(\ww^{\GG *}_{[23]}\wh{\mc{X}}^{u,r *}_{[13]} \wh{\mc{X}}^{u,r *}_{[12]} )_{[12]}\\
&=
(\id\otimes \id\otimes \omega_{\eta,\zeta})\bigl(
(\phi_{\M}\otimes \id)(\wh{\mc{X}}^{u,r})_{[12]}
\ww^{\GG *}_{[23]}
m_{[1]}
(\phi_{\M}\otimes \id)(\wh{\mc{X}}^{u,r *})_{[13]}\bigr)_{[12]}
(\phi_{\M}\otimes \id)( \wh{\mc{X}}^{u,r })_{[12]}^*
\\
&=
(\id\otimes \id\otimes \omega_{\eta,\zeta})\bigl(
\ww^{\GG *}_{[23]}
(\phi_{\M}\otimes \id\otimes\id )(\wh{\mc{X}}^{u,r}_{[12]} 
\wh{\mc{X}}^{u,r}_{[13]})
m_{[1]}
(\phi_{\M}\otimes \id)(\wh{\mc{X}}^{u,r *})_{[13]}\bigr)_{[12]}
(\phi_{\M}\otimes \id)( \wh{\mc{X}}^{u,r })_{[12]}^*\\
&=
\sum_{\lambda\in \Lambda}
(\id\otimes\omega_{\eta,\xi_\lambda})(\ww^{\GG *})_{[2]}
(\phi_{\M}\otimes \id)(\wh{\mc{X}}^{u,r})_{[12]}
(\id\otimes \omega_{\xi_\lambda,\zeta})
\\
&\hspace{3.2cm}
\bigl(
(\phi_{\M}\otimes \id)(\wh{\mc{X}}^{u,r}) m_{[1]}
(\phi_{\M}\otimes \id)(\wh{\mc{X}}^{u,r})^*\bigr)_{[1]}(\phi_{\M}\otimes\id)(\wh{\mc{X}}^{u,r})^*_{[12]}
\end{split}\end{equation}
(the above series converges in \swot). Observe that
\begin{equation}\begin{split}\label{eq23}
&\quad\;
(\phi_{\M}\otimes \id)(\wh{\mc{X}}^{u,r}) m_{[1]}
(\phi_{\M}\otimes \id)(\wh{\mc{X}}^{u,r})^*\\
&=
(\phi_{\M}\Phi\otimes \id)(\Ww^{\GG})(m\otimes \I)
(\phi_{\M}\Phi\otimes \id )(\Ww^{\GG})^*\\
&=
(\phi_{\M}\otimes \pi_{\whG}\wh{\Phi})(\WW^{\whH})(m\otimes \I)
(\phi_{\M}\otimes \pi_{\whG}\wh{\Phi})(\WW^{\whH})^*\\
&=
\bigl(
(\pi_{\whG}\wh{\Phi}\otimes \phi_{\M})(\WW^{\HH})^* (\I\otimes m)
(\pi_{\whG}\wh{\Phi}\otimes \phi_{\M})(\WW^{\HH})\bigr)_{[21]}\\
&=
\alpha^{\M}_{\whG}(m)_{[21]},
\end{split}\end{equation}
which belongs to $\M\bar\otimes \LL^{\infty}(\whG)$ (Proposition \ref{prop6}). Hence
\[
(\id\otimes \omega_{\xi_\lambda,\zeta})\bigl(
(\phi_{\M}\otimes \id)(\wh{\mc{X}}^{u,r}) m_{[1]}
(\phi_{\M}\otimes \id)(\wh{\mc{X}}^{u,r})^*\bigr)
=
(\omega_{\xi_\lambda,\zeta}\otimes \id)\alpha^{\M}_{\whG}(m)\in \M.
\]
We can thus continue \eqref{eq15} as follows:
\begin{equation}\begin{split}\label{eq16}
&\quad\;
\iota_{\M}(m)\iota_{\N}(x\otimes \I)\\
&=
\sum_{\lambda\in \Lambda}
(\id\otimes\omega_{\eta,\xi_\lambda})(\ww^{\GG *})_{[2]}
(\phi_{\M}\otimes \id)(\wh{\mc{X}}^{u,r})_{[12]}
(\omega_{\xi_\lambda,\zeta}\otimes\id)\alpha^{\M}_{\whG}(m)_{[1]}
(\phi_{\M}\otimes\id)(\wh{\mc{X}}^{u,r})^*_{[12]}\\
&=
\sum_{\lambda\in\Lambda}
\iota_{\N}\bigl(
 (\id\otimes\omega_{\eta,\xi_\lambda})(\ww^{\GG *})\otimes \I\bigr)
\iota_{\M}\bigl(
(\omega_{\xi_\lambda,\zeta}\otimes \id)\alpha^{\M}_{\whG}(m)
\bigr).
\end{split}\end{equation}

Next we consider elements of $\N=\whG'\ltimes\wt\N$ of the form $\alpha^{\wt\N}(\wt n)$ for $\wt n\in\wt{\N}$. Using $\alpha^{\wt \N}(\wt n)\in \LL^{\infty}(\whG')\bar\otimes \wt{\N}$ and $(\phi_{\M}\otimes\id)(\wh{\mc{X}}^{u,r})\in \B(\msf{H}_{\M})\bar\otimes \LL^{\infty}(\whG)$ we have
\begin{equation}\begin{split}\label{eq17}
&\quad\;
\iota_{\M}(m)\iota_{\N}(\alpha^{\wt \N}(\wt n))=
(\phi_{\M}\otimes\id)(\wh{\mc{X}}^{u,r})_{[12]}
(m\otimes \I\otimes \I)
(\phi_{\M}\otimes\id)(\wh{\mc{X}}^{u,r})_{[12]}^*
\alpha^{\wt\N}(\wt n)_{[23]}\\
&=
\alpha^{\wt\N}(\wt n)_{[23]}
(\phi_{\M}\otimes\id)(\wh{\mc{X}}^{u,r})_{[12]}
(m\otimes \I\otimes \I)
(\phi_{\M}\otimes\id)(\wh{\mc{X}}^{u,r})_{[12]}^*
=
\iota_{\N}(\alpha^{\wt \N}(\wt n))\iota_{\M}(m).
\end{split}\end{equation}
Since $\N=\whG'\ltimes \wt{\N}=\ov{\lin}^{\,\swot}\{(x\otimes \I)\alpha^{\wt{\N}}(\wt n)\mid \wt{n}\in\wt{\N},x\in\LL^{\infty}(\GG)\}$, any element of $\N$ can be approximated (in \swot) by sums of elements of the form $((\id\otimes \omega_{\eta,\zeta})(\ww^{\GG *})\otimes \I)\alpha^{\wt\N}(\wt n)$. Using equations \eqref{eq16}, \eqref{eq17} we have
\[\begin{split}
&\quad\;
\bigl( \iota_{\M}(m)\,\iota_{\N}\bigl(( (\id\otimes\omega_{\eta,\zeta})(\ww^{\GG *})\otimes \I)\alpha^{\wt \N}(\wt n) \bigr)\bigr)^*\\
&=
\bigl(
\sum_{\lambda\in\Lambda} \iota_{\N}\bigl( (\id\otimes\omega_{\eta,\xi_\lambda} )(\ww^{\GG *})\otimes \I \bigr)
\,\iota_{\M}\bigl((\omega_{\xi_\lambda,\zeta}\otimes\id)
\alpha^{\M}_{\whG}(m)\bigr)\,
\iota_{\N}( \alpha^{\wt\N}(\wt n))
\bigr)^*\\
&=
\bigl(
\sum_{\lambda\in\Lambda} \iota_{\N}\bigl( (\id\otimes\omega_{\eta,\xi_\lambda} )(\ww^{\GG *})\otimes \I \bigr)
\,
\iota_{\N}( \alpha^{\wt\N}(\wt n))
\,
\iota_{\M}\bigl((\omega_{\xi_\lambda,\zeta}\otimes\id)
\alpha^{\M}_{\whG}(m)\bigr)\,
\bigr)^*\\
&=
\sum_{\lambda\in\Lambda} \,
\iota_{\M}\bigl((\omega_{\zeta,\xi_\lambda}\otimes\id)
\alpha^{\M}_{\whG}(m^*)\bigr)\,
\iota_{\N}\bigl( 
\alpha^{\wt\N}(\wt{n}^*)\,
((\id\otimes\omega_{\xi_\lambda,\eta} )(\ww^{\GG })\otimes \I) \bigr)\in \M\ov\boxtimes\N.
\end{split}\]
Hence we can conclude that $\M\ov\boxtimes\N$ is closed under taking adjoints. In Proposition \ref{prop3} we have checked that this is enough to conclude that $\M\ov\boxtimes\N$ is a von Neumann algebra. This proves Claim 1. \qed\\

Now we go back to the situation where $\alpha^{\N},\alpha^{\M}$ are arbitrary left actions. Recall that quantum groups $\GG$ and ${\GG'}^{ op}={\GG^{op}}'$ are isomorphic; at the von Neumann algebraic level via $\oon{Ad}(u_{\GG})=j_{\GG} R_{\GG}\colon\LL^{\infty}(\GG)\rightarrow \LL^{\infty}({\GG'}^{op})$. Observe that $({\GG'}^{op})^{\wedge}={(\whG)'}^{op}$. Out of bicharacter $\wh{\mc{X}}\in \LL^{\infty}(\whH)\bar\otimes\LL^{\infty}(\whG)$ we construct new bicharacter
\[
(\id\otimes j_{\whG}R_{\whG})\wh{\mc{X}}\in \LL^{\infty}(\whH)\bar\otimes \LL^{\infty}( ({\GG '}^{op})^{\wedge}).
\]

\emph{Claim 2.} The braided tensor product $(\M,\alpha^{\M})\ov\boxtimes (\N^{\ltimes\ltimes},\alpha^{\N^{\ltimes \ltimes}})$, constructed with the bicharacter $(\id\otimes j_{\whG}R_{\whG})\wh{\mc{X}}$, is a von Neumann algebra.\\

Claim 2 follows immediately from Claim 1, since $(\N^{\ltimes \ltimes },\alpha^{\N^{\ltimes \ltimes }})$ is a dual action. 

In the rest of the proof we will work only with the actions of $\HH,\GG$ and the bicharacter $\wh{\mc{X}}$.\\

\emph{Claim 3.} $(\M,\alpha^{\M})\ov\boxtimes (\B(\LdG)\bar\otimes\N, \alpha_{+}^{\B(\LdG)\bar\otimes \N})$ is a von Neumann algebra.\\

\emph{Proof of Claim 3.} The biduality theorem tells us that the bidual action $\alpha^{\N^{\ltimes\ltimes}}\colon{\GG'}^{op}\curvearrowright \N^{\ltimes\ltimes}$ and the action $\alpha^{\B(\LdG)\bar\otimes\N}_+\colon$ $\GG\curvearrowright \B(\LdG)\bar\otimes \N$ are isomorphic (using the canonical isomorphism $\GG\simeq {\GG'}^{op}$), see \eqref{eq18}. Consequently Claim 3 follows from Claim 2 and Proposition \ref{prop3}. Indeed, we keep the same quantum group $\HH$, but pass to the isomorphic quantum group $\GG\simeq {\GG'}^{op}$ and von Neumann algebra $\B(\LdG)\bar\otimes\N\simeq \N^{\ltimes\ltimes}$ using isomorphisms which map between the bicharacters $\wh{\mc{X}}$ and $(\id\otimes j_{\whG}R_{\whG})\wh{\mc{X}}$.\qed\\

\emph{Claim 4.} $(\M,\alpha^{\M})\ov\boxtimes (\B(\LdG)\bar\otimes\N, \alpha^{\B(\LdG)\bar\otimes \N})$ is a von Neumann algebra.\\

\emph{Proof of Claim  4.} We will show
\begin{equation}\label{eq19}
(\M,\alpha^{\M})\ov\boxtimes (\B(\LdG)\bar\otimes\N, \alpha^{\B(\LdG)\bar\otimes \N})=
(\M,\alpha^{\M})\ov\boxtimes (\B(\LdG)\bar\otimes\N, \alpha_{+}^{\B(\LdG)\bar\otimes \N})
\end{equation}
(as an equality of subspaces in $\B(\msf{H}_{\M}\otimes\LdG\otimes \msf{H}_{\N})$) which reduces Claim 4 to Claim 3. Using Lemma \ref{lemma5} we have a concrete description of these braided tensor products as
\begin{equation}\begin{split}\label{eq20}
&\quad\;
(\M,\alpha^{\M})\ov\boxtimes (\B(\LdG)\bar\otimes\N, \alpha^{\B(\LdG)\bar\otimes \N})\\
&=\ov{\lin}^{\,\swot}
\{ 
(\phi_{\M}\otimes\phi_{\N})(\wh{\mc{X}}^u)_{[13]}
m_{[1]}
(\phi_{\M}\otimes\phi_{\N})(\wh{\mc{X}}^u)_{[13]}^*
(T\otimes n)_{[23]}
\mid 
m\in\M,T\in \B(\LdG),n\in\N\}
\end{split}\end{equation}
and
\begin{equation}\begin{split}\label{eq21}
&\quad\;
(\M,\alpha^{\M})\ov\boxtimes (\B(\LdG)\bar\otimes\N, \alpha_{+}^{\B(\LdG)\bar\otimes \N})\\
&=\ov{\lin}^{\,\swot}
\{ 
(\phi_{\M}\otimes(j_{\whG}R_{\whG}\pi_{\whG}\otimes \phi_{\N})\Delta_{\whG}^u)(\wh{\mc{X}}^u)
m_{[1]}
(\phi_{\M}\otimes(j_{\whG}R_{\whG}\pi_{\whG}\otimes \phi_{\N})\Delta_{\whG}^u)(\wh{\mc{X}}^u)^*
(T\otimes n)_{[23]}\\
&\quad\quad\quad\quad\quad\quad
\quad\quad\quad\quad\quad\quad
\quad\quad\quad\quad\quad\quad
\quad\quad\quad\quad\quad\quad
\quad\quad\quad\quad\quad\quad
\mid 
m\in\M,T\in \B(\LdG),n\in\N\}
\end{split}\end{equation}

Take a generating element in \eqref{eq21} for $m\in\M,T\in \B(\LdG),n\in\N$ and calculate
\begin{equation}\begin{split}\label{eq22}
&\quad\;
(\phi_{\M}\otimes(j_{\whG}R_{\whG}\pi_{\whG}\otimes \phi_{\N})\Delta_{\whG}^u)(\wh{\mc{X}}^u)
m_{[1]}
(\phi_{\M}\otimes(j_{\whG}R_{\whG}\pi_{\whG}\otimes \phi_{\N})\Delta_{\whG}^u)(\wh{\mc{X}}^u)^*
(T\otimes n)_{[23]}\\
&=
(\phi_{\M}\otimes\phi_{\N})(\wh{\mc{X}}^u)_{[13]}
(\phi_{\M}\otimes j_{\whG}R_{\whG}\pi_{\whG})(\wh{\mc{X}}^u)_{[12]}
m_{[1]}
(\phi_{\M}\otimes j_{\whG}R_{\whG}\pi_{\whG})(\wh{\mc{X}}^u)_{[12]}^*
(\phi_{\M}\otimes\phi_{\N})(\wh{\mc{X}}^u)_{[13]}^*
(T\otimes n)_{[23]}.
\end{split}\end{equation}
Using \eqref{eq23} we have
\begin{equation}\begin{split}\label{eq24}
&\quad\;
(\phi_{\M}\otimes j_{\whG}R_{\whG}\pi_{\whG})(\wh{\mc{X}}^u)_{[12]}
m_{[1]}
(\phi_{\M}\otimes j_{\whG}R_{\whG}\pi_{\whG})(\wh{\mc{X}}^u)_{[12]}^*\\
&=
(\id\otimes j_{\whG}R_{\whG})\bigl( 
(\phi_{\M}\otimes \id)(\wh{\mc{X}}^{u,r})_{[12]}
m_{[1]}
(\phi_{\M}\otimes \id)(\wh{\mc{X}}^{u,r})_{[12]}^*\bigr)\\
&=
(j_{\whG}R_{\whG}\otimes \id)\alpha^{\M}_{\whG}(m)_{[21]}\in 
\M\bar\otimes \LL^{\infty}(\whG)'.
\end{split}\end{equation}
The last containment allows us to write $(j_{\whG}R_{\whG}\otimes \id)\alpha^{\M}_{\whG}(m)_{[21]}=\sum_{i\in I} m_i\otimes \wh{x}'_i$ for some $m_i\in\M,\wh{x}'_i\in\LL^{\infty}(\whG)'$ (convergence in \swot) and consequently continue \eqref{eq22} as
\[\begin{split}
&\quad\;
(\phi_{\M}\otimes(j_{\whG}R_{\whG}\pi_{\whG}\otimes \phi_{\N})\Delta_{\whG}^u)(\wh{\mc{X}}^u)
m_{[1]}
(\phi_{\M}\otimes(j_{\whG}R_{\whG}\pi_{\whG}\otimes \phi_{\N})\Delta_{\whG}^u)(\wh{\mc{X}}^u)^*
(T\otimes n)_{[23]}\\
&=
\sum_{i\in I}
(\phi_{\M}\otimes\phi_{\N})(\wh{\mc{X}}^u)_{[13]}
(m_i\otimes \wh{x}'_i)_{[12]}
(\phi_{\M}\otimes\phi_{\N})(\wh{\mc{X}}^u)_{[13]}^*
(T\otimes n)_{[23]}\\
&=
\sum_{i\in I}
(\phi_{\M}\otimes\phi_{\N})(\wh{\mc{X}}^u)_{[13]}
m_{i [1]}
(\phi_{\M}\otimes\phi_{\N})(\wh{\mc{X}}^u)_{[13]}^*
(\wh{x}'_iT\otimes n)_{[23]}.
\end{split}\]
This shows the inclusion $\supseteq$ in \eqref{eq19}.

Conversely, take as before $m\in\M,T\in\B(\LdG),n\in\N$ and consider a generating element of \eqref{eq20},
\begin{equation}\label{eq25}
(\phi_{\M}\otimes\phi_{\N})(\wh{\mc{X}}^u)_{[13]}
m_{[1]}
(\phi_{\M}\otimes\phi_{\N})(\wh{\mc{X}}^u)_{[13]}^*
(T\otimes n)_{[23]}
=
(\phi_{\M}\otimes\phi_{\N})(\wh{\mc{X}}^u)_{[13]}
(m\otimes T)_{[12]}
(\phi_{\M}\otimes\phi_{\N})(\wh{\mc{X}}^u)_{[13]}^*
 n_{[3]}.
\end{equation}
Thanks to Lemma \ref{lemma7} for the action $\alpha^{\M}_{\whG}$, upon applying adjoint, flip and automorphism $\Ad(u_{\GG})$, we can write $m\otimes T$ as
\[
m\otimes T=
\sum_{j\in J}
(j_{\whG}R_{\whG}\otimes \id)\alpha^{\M}_{\whG}(m_j)_{[21]}
(\I\otimes T_j)
\]
(convergence in \swot) for some $T_j\in\B(\LdG),m_j\in\M$. Consequently, using again \eqref{eq24} we continue \eqref{eq25}:
\[\begin{split}
&\quad\;
(\phi_{\M}\otimes\phi_{\N})(\wh{\mc{X}}^u)_{[13]}
m_{[1]}
(\phi_{\M}\otimes\phi_{\N})(\wh{\mc{X}}^u)_{[13]}^*
(T\otimes n)_{[23]}
\\
&=
\sum_{j\in J}
(\phi_{\M}\otimes\phi_{\N})(\wh{\mc{X}}^u)_{[13]}
(j_{\whG}R_{\whG}\otimes \id)\alpha^{\M}_{\whG}(m_j)_{[21]}
T_{j [2]}
(\phi_{\M}\otimes\phi_{\N})(\wh{\mc{X}}^u)_{[13]}^*
n_{[3]}
\\
&=
\sum_{j\in J}
(\phi_{\M}\otimes\phi_{\N})(\wh{\mc{X}}^u)_{[13]}
(\phi_{\M}\otimes j_{\whG}R_{\whG}\pi_{\whG})(\wh{\mc{X}}^u)_{[12]}
m_{j [1]}
(\phi_{\M}\otimes j_{\whG}R_{\whG}\pi_{\whG})(\wh{\mc{X}}^u)_{[12]}^*
(\phi_{\M}\otimes\phi_{\N})(\wh{\mc{X}}^u)_{[13]}^*
 (T_j\otimes n)_{[23]}
 \\
&=
\sum_{j\in J}
(\phi_{\M}\otimes (j_{\whG}R_{\whG}\pi_{\whG}\otimes \phi_{\N})\Delta_{\whG}^u )(\wh{\mc{X}}^u)
m_{j [1]}
(\phi_{\M}\otimes (j_{\whG}R_{\whG}\pi_{\whG}\otimes \phi_{\N})\Delta_{\whG}^u )(\wh{\mc{X}}^u)^*
 (T_j\otimes n)_{[23]}
\end{split}\]
which clearly belongs to \eqref{eq21}. We conclude that the two vector spaces \eqref{eq20} and \eqref{eq21} are equal.\qed
\\

\emph{Claim 5.} $\M\ov\boxtimes \N$ is a von Neumann algebra.\\

\emph{Proof of Claim 5.} We have established in Lemma \ref{lemma5} that the morphism associated with action $\alpha^{\B(\LdG)\bar\otimes \N}$ is $\I\otimes \phi_{\N}$. Consequently, Claim 5 tells us that the subspace
\[\begin{split}
&\quad\;
(\M,\alpha^{\M})\ov\boxtimes (\B(\LdG)\bar\otimes \N,\alpha^{\B(\LdG)\bar\otimes\N})\\
&=
\ov{\lin}^{\,\swot}\{
(\phi_{\M}\otimes\phi_{\N})(\wh{\mc{X}}^u)_{[13]}
m_{[1]}
(\phi_{\M}\otimes\phi_{\N})(\wh{\mc{X}}^u)_{[13]}^*
(T\otimes n)_{[23]}
\mid
m\in\M,T\in\B(\LdG),n\in\N\}
\end{split}\]
is a von Neumann algebra in $ \B(\msf{H}_{\M}\otimes \LdG\otimes \msf{H}_{\N})$. Upon applying the flip map to the first two
legs we see that 
\[
\ov{\lin}^{\,\swot}\{
(\phi_{\M}\otimes\phi_{\N})(\wh{\mc{X}}^u)_{[23]}
m_{[2]}
(\phi_{\M}\otimes\phi_{\N})(\wh{\mc{X}}^u)_{[23]}^*
(T\otimes n)_{[13]}
\mid
m\in\M,T\in\B(\LdG),n\in\N\}
\]
is a von Neumann algebra in $\B(\LdG\otimes \msf{H}_{\M}\otimes\msf{H}_{\N})$. This space is clearly equal to
\[
\ov{\lin}^{\,\swot}\{ T\otimes x \mid T\in \B(\LdG),x\in \M\ov\boxtimes \N\}
\]
and Claim 5 follows.
\qed
\end{proof}

At the end of this section, let us record two easy results, the first concerned with quantum subgroups and the second with the associated bicharacter.

\subsection{Braided tensor product and quantum subgroups}
Assume that $\HH_1\subseteq\HH,\GG_1\subseteq\GG$ are locally compact quantum groups with quantum subgroups, and that $\wh{\mc{X}}\in \LL^{\infty}(\whH)\bar\otimes \LL^{\infty}(\whG)$, $\wh{\mc{X}}_1\in \LL^{\infty}(\whH_1)\bar\otimes \LL^{\infty}(\whG_1)$ are bicharacters such that $\wh{\mc{X}}=(\gamma_{\HH_1\subseteq \HH}\otimes \gamma_{\GG_1\subseteq \GG})(\wh{\mc{X}}_1)$. Assume that we have actions $\alpha^{\M}_{\HH}\colon \HH\curvearrowright \M, \alpha^{\N}_{\GG}\colon \GG\curvearrowright \N$ and corresponding restricted actions $\alpha^{\M}_{\HH_1}=(\alpha^{\M}_{\HH})\rest_{\HH_1}$, $\alpha^{\N}_{\GG_1}=(\alpha^{\N}_{\GG})\rest_{\GG_1}$. Next, represent $\M,\N$ on some Hilbert spaces and let $\phi_{\M},\phi_{\N}$ be implementations of the actions of $\HH,\GG$. Recall that $\phi_{\M}\wh{\theta}_{\HH_1\subseteq \HH}$ and $\phi_{\N}\wh{\theta}_{\GG_1\subseteq\GG}$ implement actions of $\HH_1,\GG_1$ (see Section \ref{sec:preliminaries}). In this situation, we can form two braided tensor products: $\M\ov{\boxtimes}_{\wh{\mc{X}}} \N$ with respect to the actions of $\GG,\HH$ and bicharacter $\wh{\mc{X}}$, and $\M\ov{\boxtimes}_{\wh{\mc{X}}_1}\N$ with respect to $\GG_1,\HH_1,\wh{\mc{X}}_1$.

\begin{proposition}\label{prop19}
We have $\M\ov{\boxtimes}_{\wh{\mc{X}}}\N=\M\ov{\boxtimes}_{\wh{\mc{X}}_1}\N$ with an equality of canonical embeddings of $\M,\N$.
\end{proposition}

\begin{proof}
Observe first that the universal lifts $\wh{\mc{X}}^u,\wh{\mc{X}}^u_1$ of bicharacters satisfy
\begin{equation}\label{eq49}
\wh{\mc{X}}^u=
(\wh{\theta}_{\HH_1\subseteq \HH}\otimes \wh{\theta}_{\GG_1\subseteq \GG})(\wh{\mc{X}}_1^u).
\end{equation}
Indeed, both of the above unitary operators are bicharacters and are lifts of $\wh{\mc{X}}$, hence the claim follows from the uniqueness part of \cite[Proposition 4.7]{Homomorphisms}. Now the equality $\M\ov{\boxtimes}_{\wh{\mc{X}}} \N=\M\ov{\boxtimes}_{\wh{\mc{X}}_1}\N$ is an immediate consequence of definitions.
\end{proof}

\subsection{Concerning the bicharacter $\wh{\mc{X}}_{[21]}^*$}
Let $\GG,\HH$ be locally compact quantum groups, $\wh{\mc{X}}\in \LL^{\infty}(\whH)\bar\otimes\LL^{\infty}(\whG)$ a bicharacter and $\M,\N$ von Neumann algebras equipped with actions $\HH\curvearrowright \M,\GG\curvearrowright \N$. Then $\wh{\mc{X}}_{[21]}^*\in \LL^\infty(\whG)\bar\otimes\LL^{\infty}(\whH)$ is a bicharacter (Remark \ref{remark4}) and we can consider two braided tensor products, $\M\ov{\boxtimes}_{\wh{\mc{X}}}\N$ and $\N\ov{\boxtimes}_{\wh{\mc{X}}_{[21]}^*}\M$. To avoid ambiguity, let us denote the corresponding embeddings by $\iota_{\M}^{\wh{\mc{X}}},\iota_{\N}^{\wh{\mc{X}}}$ and $\iota_{\N}^{\wh{\mc{X}}_{[21]}^*}, \iota_{\M}^{\wh{\mc{X}}_{[21]}^*}$.

\begin{proposition}\label{prop22}
The map
\[
\M\ov{\boxtimes}_{\wh{\mc{X}}} \N\ni \iota_{\M}^{\wh{\mc{X}}}(m) \iota_{\N}^{\wh{\mc{X}}}(n)\mapsto 
 \iota_{\M}^{\wh{\mc{X}}_{[21]}^*}(m)
\iota_{\N}^{\wh{\mc{X}}_{[21]}^*}(n) 
 \in 
\N\ov{\boxtimes}_{\wh{\mc{X}}_{[21]}^*}\M.
\]
extends to a well defined $*$-isomorphism.
\end{proposition}

\begin{proof}
Choose arbitary representations of $\M,\N$ and implementations $\phi_{\M},\phi_{\N}$ of the actions. By definition, for $m\in\M,n\in\N$ we have
\[
\iota_{\M}^{\wh{\mc{X}}}(m) \iota_{\N}^{\wh{\mc{X}}}(n)=
(\phi_{\M}\otimes \phi_{\N})(\wh{\mc{X}}^u) (m\otimes \I)
(\phi_{\M}\otimes \phi_{\N})(\wh{\mc{X}}^u)^*(\I\otimes n)
\in
\M\ov\boxtimes_{\wh{\mc{X}}}\N.
\]
Applying $\oon{Ad}\bigl( (\phi_{\M}\otimes \phi_{\N})(\wh{\mc{X}}^u)^*\bigr)$ and the tensor flip to this element gives
\[
(\I\otimes m)
(\phi_{\N}\otimes \phi_{\M})(\wh{\mc{X}}^{u *}_{[21]})(n\otimes \I)
(\phi_{\N}\otimes \phi_{\M})(\wh{\mc{X}}^u_{[21]})=
 \iota_{\M}^{\wh{\mc{X}}_{[21]}^*}(m)
\iota_{\N}^{\wh{\mc{X}}_{[21]}^*}(n) 
 \in \N\ov{\boxtimes}_{\wh{\mc{X}}^*_{[21]}} \M,
\]
which proves the claim.
\end{proof}

\section{Maps on $\M\ov\boxtimes \N$}\label{sec:maps}

Let $\HH,\GG$ be locally compact quantum groups acting on von Neumann algebras $\HH\curvearrowright \M,\wt\M, \GG\curvearrowright \N,\wt\N$, and $\wh{\mc{X}}\in \LL^{\infty}(\whH)\bar\otimes\LL^{\infty}(\whG)$ a bicharacter. In this section we prove that we can take the ``braided tensor product'' of equivariant maps, and deduce from this certain permanence of approximation properties. Proposition \ref{prop13} shows that the braided tensor product of non-equivariant maps might not exist.

\begin{proposition}\label{prop8}
Let $\vartheta_1\colon \M\rightarrow \wt{\M}$, $\vartheta_2\colon \N\rightarrow \wt{\N}$ be maps which are normal, completely bounded and equivariant. There exists a unique normal linear map
\[
\vartheta_1\boxtimes\vartheta_2\colon 
\M\ov\boxtimes\N \ni \iota_{\M}(m)\iota_{\N}(n)\mapsto
\iota_{\wt{\M}}(\vartheta_1(m))\iota_{\wt{\N}}(\vartheta_2(n))
\in \wt{\M}\ov\boxtimes\wt{\N}.
\]
It is completely bounded with $\|\vartheta_1\boxtimes\vartheta_2\|_{cb}\le \|\vartheta_1\|_{cb}\|\vartheta_2\|_{cb}$. If $\vartheta_1,\vartheta_2$ are unital / completely positive / $*$-homomorphisms / completely isometric, then so is $\vartheta_1\boxtimes\vartheta_2$.
\end{proposition}

In the next section (Corollary \ref{cor3}) we prove a result concerning equivariance of $\vartheta_1\boxtimes \vartheta_2$.

\begin{proof}
Since elements of the form $\iota_{\M}(m)\iota_{\N}(n)$ span a $w^*$-dense linear subspace in $\M\ov\boxtimes\N$, uniqueness of $\vartheta_1\boxtimes\vartheta_2$ is clear. Choose arbitrary implementations $U^{\M},U^{\wt\M},U^{\N},U^{\wt\N}$ of the actions. By Proposition \ref{prop4} we have
\begin{equation}\label{eq26}
U^{\N *}_{[34]} \wh{\mc{X}}_{[13]} U^{\M *}_{[12]}
\;
(\iota_{\M}(m)\iota_{\N}(n))_{[24]}\;
U^{\M }_{[12]}\wh{\mc{X}}_{[13]}^* U^{\N}_{[34]}=
\wh{\mc{X}}_{[13]} \alpha^{\M}(m)_{[12]} 
\wh{\mc{X}}_{[13]}^* \alpha^{\N}(n)_{[34]}
\end{equation}
for any $m\in\M,n\in\N$. Define
\[
\Upsilon_1\colon \M\ov\boxtimes \N\ni x \mapsto x_{[24]}\in \B(\LL^{2}(\HH)\otimes \msf{H}_{\M}\otimes \LL^2(\GG)\otimes \msf{H}_{\N}),
\]
which is a normal, injective $*$-homomorphism. Consider the composition
\[
\oon{Ad}(U^{\N *}_{[34]} \wh{\mc{X}}_{[13]} U^{\M *}_{[12]}) \Upsilon_1\colon 
\M\ov\boxtimes \N\rightarrow 
\B(\LL^2(\HH))\bar\otimes \M\bar\otimes \B(\LL^2(\GG))\bar\otimes \N,
\]
which is again a normal, injective $*$-homomorphism. Equation \eqref{eq26} shows that it has the correct codomain. Next we compose with $\id\otimes \vartheta_1\otimes \id\otimes \vartheta_2$, which is a well defined normal CB map. Indeed, it can be defined as the dual to the CB map $\id\otimes (\vartheta_1)_*\otimes\id\otimes(\vartheta_2)_*$ on $\B(\LL^2(\HH))_*\wh\otimes \wt{\M}_*\wh\otimes\B(\LL^2(\GG))_*\wh\otimes\wt{\N}_*$ (\cite[Corollary 7.1.3]{EffrosRuan}). On generators we have
\[\begin{split}
&\quad\;
(\id\otimes \vartheta_1\otimes\id\otimes\vartheta_2)
\oon{Ad}(U^{\N *}_{[34]} \wh{\mc{X}}_{[13]} U^{\M *}_{[12]}) \Upsilon_1\bigl(\iota_{\M}(m)\iota_{\N}(n))\\
&=
\wh{\mc{X}}_{[13]}
(\id\otimes\vartheta_1)\alpha^{\M}(m)_{[12]}
\wh{\mc{X}}^*_{[13]}
(\id\otimes\vartheta_2)\alpha^{\N}(n)_{[34]}\\
&=
\wh{\mc{X}}_{[13]}
\alpha^{\wt\M}(\vartheta_1(m))_{[12]}
\wh{\mc{X}}^*_{[13]}
\alpha^{\wt\N}(\vartheta_2(n))_{[34]}
\in
\B(\LL^{2}(\HH))\bar\otimes \wt{\M}\bar\otimes \B(\LL^2(\GG))
\bar\otimes \wt{\N}
\end{split}\]
as $\theta_1,\theta_2$ are assumed to be equivariant. Composing with $\oon{Ad}( U^{\wt{\M} }_{[12]}\wh{\mc{X}}_{[13]}^* U^{\wt{\N} }_{[34]})$ gives that 
\[
\oon{Ad}( U^{\wt{\M} }_{[12]}\wh{\mc{X}}_{[13]}^* U^{\wt{\N} }_{[34]})
(\id\otimes \vartheta_1\otimes\id\otimes\vartheta_2)
\oon{Ad}(U^{\N *}_{[34]} \wh{\mc{X}}_{[13]} U^{\M *}_{[12]}) \Upsilon_1
\]
defines a map 
\begin{equation}\label{eq27}
\M\ov\boxtimes\N\rightarrow 
\B(\LL^2(\HH)\otimes \msf{H}_{\wt{\M}}\otimes\LL^2(\GG)\otimes
\msf{H}_{\wt{\N}})\colon 
\iota_{\M}(m)\iota_{\N}(n)\mapsto
(\iota_{\wt{\M}}(\vartheta_1(m))\iota_{\wt{\N}}(\vartheta_2(n)))_{[24]}.
\end{equation}
Thus we see that in fact we can take $(\wt{\M}\,\ov\boxtimes\,\wt{\N})_{[24]}$ as the codomain of the above map. We define $\vartheta_1\boxtimes\vartheta_2$ by composition of \eqref{eq27} with the canonical isomorphism $(\wt{\M}\,\ov\boxtimes\,\wt{\N})_{[24]}\rightarrow \wt{\M}\,\ov\boxtimes\,\wt{\N}$. In this way we have defined a normal, linear map $\vartheta_1\boxtimes\vartheta_2$ with CB norm bounded above by $ \|\vartheta_1\|_{cb}\|\vartheta_2\|_{cb}$ (as this property holds for tensor products). 

Clearly if $\vartheta_1,\vartheta_2$ are unital, then so is $\vartheta_1\boxtimes\vartheta_2$. If $\vartheta_1,\vartheta_2$ are CP then so is $\vartheta_1\boxtimes\vartheta_2$ as all the involved maps are CP. An analogous argument works for $*$-homomorphisms and also for completely isometric maps. Indeed,  by \cite[Corollary 4.1.9]{EffrosRuan}, $\vartheta_1,\vartheta_2$ are completely isometric if, and only if $(\vartheta_1)_*,(\vartheta_2)_*$ are complete quotient mappings. By \cite[Proposition 7.1.7]{EffrosRuan} $\id\otimes (\vartheta_1)_*\otimes\id\otimes (\vartheta_2)_*$ also is a complete quotient mapping, hence $\id\otimes \vartheta_1\otimes \id\otimes\vartheta_2$ is completely isometric. 
\end{proof}

\begin{remark}
If $\omega_{\M}\in\M_*,\omega_{\N}\in\N_*$ are faithful, invariant states, then one can show that $\omega_{\M}\boxtimes \omega_{\N}$ is a faithful normal state on $\M\ov\boxtimes\N$. We prove this in our next work concerning standard structure for $\M\ov\boxtimes\N$ (\cite{DeCommerKrajczokToAppear}).
\end{remark}

Let us recall definitions of approximation properties for a von Neumann algebra $\oon{P}$:
\begin{itemize}
\item $\oon{P}$ has \emph{$w^*$ CPAP} (or is \emph{semidiscrete}) if there exists a net $(\vartheta_i)_{i\in I}$ of unital, normal, finite rank CP maps $\oon{P}\rightarrow \oon{P}$ such that $\vartheta_i\xrightarrow[i\in I]{}\id$ in the point-$w^*$-topology (see \cite[Definition 2.3.3]{BrownOzawa} and \cite[Lemma 2.1]{ChoiEffros}).
\item $\oon{P}$ has \emph{$w^*$ CBAP} if there exists a net $(\vartheta_i)_{i\in I}$ of finite rank, normal CB maps $\oon{P}\rightarrow \oon{P}$ such that $\vartheta_i\xrightarrow[i\in I]{}\id$ in the point-$w^*$-topology and $\sup_{i\in I} \|\vartheta_i\|_{cb}<+\infty$. The smallest bound $\Lambda_{cb}(\oon{P})$ for this value is called the Cowling-Haagerup constant of $\oon{P}$ (\cite[Definition 12.3.9]{BrownOzawa}).
\item $\oon{P}$ has \emph{$w^*$ OAP} if there exists a net $(\vartheta_i)_{i\in I}$ of finite rank, normal CB maps $\oon{P}\rightarrow \oon{P}$ such that $(\vartheta_i\otimes \id)x\xrightarrow[i\in I]{}x$ weak$^*$ for all Hilbert spaces $\msf{H}$ and $x\in \oon{P}\bar\otimes \B(\msf{H})$. In other words, $\vartheta_i\xrightarrow[i\in I]{}\id$ in the stable point-$w^*$-topology (see \cite[section 1,2, Proposition 1.7]{HaagerupKraus}).
\end{itemize}

We can use Proposition \ref{prop8} to show that the braided tensor product preserves the above approximation properties, provided that the maps $\vartheta_i$ are equivariant (in Corollary \ref{cor2} we show that we cannot completely abandon the equivariance condition). Let $\HH\curvearrowright\M, \GG\curvearrowright\N,\wh{\mc{X}}$ be as above.

\begin{proposition}\label{prop9}
Assume that $\M$ and $\N$ both simultaneously have one of the approximation properties $w^*$ CPAP, $w^*$ CBAP, $w^*$ OAP and we can choose the implementing maps $\vartheta_i\in\CB^\sigma(\M)(i\in I),\psi_j\in\CB^\sigma(\N)(j\in J)$ to be equivariant. Then $\M\ov\boxtimes\N$ also has the relevant approximation property.
\end{proposition}

\begin{proof}
Consider normal, CB maps $\vartheta_i\boxtimes\psi_j\, (i\in I,j\in J)$ on $\M\ov\boxtimes\N$, defined in Proposition \ref{prop8}. They are finite rank. Indeed, for any $m\in\M,n\in\N$ we have
\[
(\vartheta_i\boxtimes\psi_j)(\iota_{\M}(m)\iota_{\N}(n))=
\iota_{\M}(\vartheta_i(m))\iota_{\N}(\psi_j(n))
\]
hence the claim follows from the definition of $\M\ov\boxtimes\N$ and the fact that finite dimensional subspaces are weak$^*$-closed (\cite[Theorem 1.21]{RudinFA}). Taking into consideration properties of $\vartheta_i\boxtimes\psi_j$ proved in Proposition \ref{prop8}, it is enough to check that the identity $\id\in \CB^\sigma(\M\ov\boxtimes\N)$ belongs to the stable point-$w^*$-closure of $\{\vartheta_i\boxtimes\psi_j\mid i\in I, j\in J\}$.  Fix a Hilbert space $\msf{H}$, $x\in (\M\ov\boxtimes \N)\bar\otimes \B(\msf{H})$ and $\omega \in ((\M\ov\boxtimes\N)\bar\otimes \B(\msf{H}))_*$.

Let $\Upsilon\colon\M\ov\boxtimes \N\rightarrow \B(\LL^2(\HH))\bar\otimes \M\bar\otimes \B(\LL^2(\GG))\bar\otimes \N$ be the injective, normal $*$-homomorphism given on generators by $\Upsilon(\iota_{\M}(m)\iota_{\N}(n))=\wh{\mc{X}}_{[13]} \alpha^{\M}(m)_{[12]}\wh{\mc{X}}^*_{[13]} \alpha^{\N}(n)_{[34]}$ (Proposition \ref{prop4}). Note that for $m\in\M,n\in\N$, $T\in\B(\msf{H})$
\[\begin{split}
&\quad\;
(\Upsilon\otimes \id)  ((\vartheta_i\boxtimes\psi_j)\otimes \id)(\iota_{\M}(m)\iota_{\N}(n)\otimes T)\\
&=
\wh{\mc{X}}_{[13]} \alpha^{\M}(\vartheta_i(m))_{[12]}\wh{\mc{X}}^*_{[13]} \alpha^{\N}(\psi_j(n))_{[34]}\otimes T\\
&=
\wh{\mc{X}}_{[13]} (\id\otimes \vartheta_i)\alpha^{\M}(m)_{[12]}\wh{\mc{X}}^*_{[13]} (\id\otimes\psi_j)\alpha^{\N}(n)_{[34]}\otimes T\\
&=
(\id\otimes\vartheta_i\otimes \id\otimes \psi_j\otimes \id )\bigl( 
\wh{\mc{X}}_{[13]} \alpha^{\M}(m)_{[12]}\wh{\mc{X}}^*_{[13]} \alpha^{\N}(n)_{[34]}\otimes T\bigr)\\
&=
(\id\otimes\vartheta_i\otimes \id\otimes \psi_j\otimes \id )(\Upsilon\otimes\id)\bigl(
\iota_{\M}(m)\iota_{\N}(n)\otimes T\bigr),
\end{split}\]
hence by continuity
\begin{equation}\label{eq31}
(\Upsilon\otimes \id)  ((\vartheta_i\boxtimes\psi_j)\otimes \id)(y)=
(\id\otimes\vartheta_i\otimes \id\otimes \psi_j\otimes \id )(\Upsilon\otimes \id)(y)
\end{equation}
for all $y\in (\M\ov\boxtimes\N)\bar\otimes\B(\msf{H})$. Observe that for this calculation we needed the assumption that $\vartheta_i,\psi_j$ are equivariant. An analogous reasoning shows 
\begin{equation}\label{eq32}
(\Upsilon\otimes \id)  ((\vartheta_i\boxtimes\id)\otimes \id)(y)=
(\id\otimes\vartheta_i\otimes \id\otimes \id\otimes \id )(\Upsilon\otimes \id)(y).
\end{equation}
Next, note that $\Upsilon\otimes \id$ is an embedding of $(\M\ov\boxtimes\N)\bar\otimes \B(\msf{H})$ into $\B(\LL^2(\HH))\bar\otimes\M\bar\otimes\B(\LdG)\bar\otimes\N\bar\otimes\B(\msf{H})$, hence we can find a normal functional $\wt\omega$ on $\B(\LL^2(\HH))\bar\otimes\M\bar\otimes\B(\LdG)\bar\otimes\N\bar\otimes\B(\msf{H})$ such that $\omega(y)=\wt{\omega}(\Upsilon\otimes\id)(y)$. Fix $i\in I$ and calculate
\[\begin{split}
&\quad\;
\omega\bigl(
 ((\vartheta_i\boxtimes\psi_j)\otimes\id)x-((\vartheta_i\boxtimes\id)\otimes\id)x
\bigr)=
\wt{\omega}(\Upsilon\otimes\id)\bigl(
 ((\vartheta_i\boxtimes\psi_j)\otimes\id)x-
  ((\vartheta_i\boxtimes\id)\otimes\id)x
\bigr)\\
&=
\wt{\omega}\bigl(
(\id\otimes \id\otimes\id\otimes \psi_j\otimes \id)
(\, (\id\otimes \vartheta_i\otimes\id\otimes \id\otimes \id) (\Upsilon\otimes\id)x \, )-
(\id\otimes \vartheta_i\otimes\id\otimes \id\otimes \id)
(\Upsilon\otimes\id)x 
\bigr)\xrightarrow[j\in J]{}0,
\end{split}\]
since $(\id\otimes \id\otimes\id\otimes \psi_j\otimes \id)_{j\in J}$ converges to the identity in the point-$w^*$-topology (\cite[Proposition 1.7]{HaagerupKraus}). This shows that $\vartheta_i\boxtimes\psi_j\xrightarrow[j\in J]{}\vartheta_i\boxtimes\id$ in the stable point-$w^*$ topology. Similarly we check that $\vartheta_i\boxtimes\id\xrightarrow[i\in I]{}\id$ in the stable point-$w^*$-topology, hence $\id\in \CB^\sigma(\M\ov\boxtimes\N)$ belongs to the stable point-$w^*$-closure of $\{\vartheta_i\boxtimes \psi_j\mid i\in I,j\in J\}$ as claimed (c.f.~ the proof of \cite[Theorem 7.16]{DKV_ApproxLCQG}).
\end{proof}

\section{Action on braided tensor product and canonical quantum subgroups of $\GG$}\label{sec:action}

Let $(\GG,\wh{\oon{R}})$ be a quasi-triangular locally compact quantum group, i.e.~$\GG$ is a locally compact quantum group and $\wh{\oon{R}}\in \LL^{\infty}(\whG)\bar\otimes\Linfd$ is an $\oon{R}$-matrix. Define two weak$^*$-closed subspaces
\[\begin{split}
\oon{L}_1&=\ov{\lin}^{\,\swot}\{(\id\otimes\omega)\wh{\oon{R}}\mid \omega\in \LL^1(\whG)\}\subseteq \Linfd,\\
\oon{L}_2&=\ov{\lin}^{\,\swot}\{(\omega\otimes\id)\wh{\oon{R}}\mid \omega\in \LL^1(\whG)\}\subseteq\Linfd.
\end{split}\]
It is not difficult to see that $\oon{L}_1,\oon{L}_2$ are in fact unital $*$-subalgebras, i.e.~von Neumann subalgebras of $\Linfd$. Then using the bicommutant theorem we see that $\wh{\oon{R}}\in\oon{L}_1\bar\otimes \oon{L}_2$. In fact, $\oon{L}_1,\LL_2$ are globally invariant under the unitary antipode and  scaling group, and satisfy
\[
\Delta_{\whG}(\oon{L}_1)\subseteq \oon{L}_1\bar\otimes \oon{L}_1,\quad
\Delta_{\whG}(\oon{L}_2)\subseteq \oon{L}_2\bar\otimes \oon{L}_2
\]
(see \cite[Proposition 4.1]{KasprzakKhosraviSoltan}). This means that $\oon{L}_1,\oon{L}_2$ are \emph{Baaj-Vaes subalgebras} and by \cite[Proposition A.5]{BaajVaes} there exist locally compact quantum groups $\GG_1,\GG_2$ such that
\[
\LL^{\infty}(\wh{\GG}_1)=\oon{L}_1,\quad
\LL^{\infty}(\wh{\GG}_2)=\oon{L}_2,
\]
with comultiplications of $\wh{\GG}_1,\wh{\GG}_2$ restrictions of the comultiplication of $\whG$. Denote by
\[
\gamma_{\GG_1\subseteq \GG}\colon \LL^{\infty}(\wh{\GG}_1)\rightarrow\LL^{\infty}(\whG),\quad
\gamma_{\GG_2\subseteq \GG}\colon \LL^{\infty}(\wh{\GG}_2)\rightarrow\LL^{\infty}(\whG)
\]
the formal inclusion maps; we see that they make $\GG_1,\GG_2$ into quantum subgroups of $\GG$. In particular, any action of $\GG$ can be restricted to $\GG_1$ or $\GG_2$. As usual, we will denote by $\theta_{\GG_1\subseteq \GG},\theta_{\GG_2\subseteq\GG}$ the associated strong quantum homomorphisms.

\begin{remark}
In the special case when $\GG=D(\HH)$ is the Drinfeld double of a locally compact quantum group $\HH$, we have $\GG_1=\HH$ and $\GG_2=\whH$. Indeed, this is an immediate consequence of the definition of $\oon{R}$-matrix: $\wh{\oon{R}}=(\gamma_{\whH\subseteq D(\HH)}\otimes \gamma_{\HH\subseteq D(\HH)})(\ww^{\HH})$ (Proposition \ref{prop14}).
\end{remark}

We can see the unitary $\wh{\oon{R}}$ in four different ways: as an element of $\LL^{\infty}(\whG_1\!)\bar\otimes \LL^{\infty}(\whG_2\!)$, $\LL^{\infty}(\whG_1\!)\bar\otimes \LL^{\infty}(\whG)$,

$\LL^{\infty}(\whG)\bar\otimes \LL^{\infty}(\whG_2\!)$ or $\LL^{\infty}(\whG)\bar\otimes \LL^{\infty}(\whG)$. Let us denote by $\wh{\oon{R}}_{12},\wh{\oon{R}}_{1},\wh{\oon{R}}_{2},\wh{\oon{R}}$ this operator, considered as an element of the respective algebra. By the way we define comultiplication on $\whG_1,\whG_2$, each of these unitaries is then a bicharacter. We note that there are expected relations between their lifts:
\begin{equation}\begin{split}\label{eq53}
(\wh{\theta}_{\GG_1\subseteq \GG}\otimes \id)(\wh{\oon{R}}^u_{12})&=
\wh{\oon{R}}_2^u,\\
(\id\otimes \wh{\theta}_{\GG_2\subseteq \GG})(\wh{\oon{R}}^u_{12})&=\wh{\oon{R}}^u_1\\
(\wh{\theta}_{\GG_1\subseteq \GG}\otimes \id )(\wh{\oon{R}}^u_{1})
&=
\wh{\oon{R}}^u=
(\id\otimes \wh{\theta}_{\GG_2\subseteq \GG})(\wh{\oon{R}}^u_{2}).
\end{split}\end{equation}
These properties follow from uniqueness of lifts, \cite[Proposition 4.7]{Homomorphisms}. The unitaries $\wh{\oon{R}}_1,\wh{\oon{R}}_{2}$ retain a version of the $\oon{R}$-matrix condition.

\begin{lemma}\label{lemma11}
For $\wh{x}\in \mrm{C}_0^u(\whG_1),\wh{y}\in \mrm{C}_0^u(\whG_2)$ we have
\[\begin{split}
(\id\otimes \wh{\theta}_{\GG_1\subseteq\GG})\Delta_{\whG_1}^{u, op} (\wh{x})&=
\wh{\oon{R}}_1^u 
(\id\otimes \wh{\theta}_{\GG_1\subseteq\GG})\Delta_{\whG_1}^u(\wh{x})
\wh{\oon{R}}_1^{u *},\\
(\wh{\theta}_{\GG_2\subseteq\GG}\otimes \id )\Delta_{\whG_2}^{u, op} (\wh{y})&=
\wh{\oon{R}}_2^u 
( \wh{\theta}_{\GG_2\subseteq\GG}\otimes \id)\Delta_{\whG_2}^u(\wh{y})
\wh{\oon{R}}_2^{u *}.
\end{split}\]
\end{lemma}

\begin{proof}
Since $\wh{\oon{R}}$ is an $\oon{R}$-matrix and $\gamma_{\GG_1\subseteq\GG}$ is the inclusion map which respects coproducts, we have at the reduced level
\[
(\id\otimes \id\otimes \gamma_{\GG_1\subseteq \GG})(\id\otimes \Delta_{\whG_1}^{op})(\ww^{\GG_1})=
\wh{\oon{R}}_{1[23]}
(\id\otimes (\id\otimes\gamma_{\GG_1\subseteq\GG})\Delta_{\wh{\GG}_1})(\ww^{\GG_1})
\wh{\oon{R}}_{1[23]}^*
\]
or equivalently
\begin{equation}\label{eq51}
\ww^{\GG_1}_{[12]}(\id\otimes \gamma_{\GG_1\subseteq \GG})(\ww^{\GG_1})_{[13]}=
\wh{\oon{R}}_{1[23]}
(\id\otimes \gamma_{\GG_1\subseteq \GG})(\ww^{\GG_1})_{[13]}\ww^{\GG_1}_{[12]}
\wh{\oon{R}}_{1[23]}^*.
\end{equation}
Applying $\id\otimes\id\otimes \Delta_{\whG}^{r,u}$ gives
\[\begin{split}
&\quad\;
\ww^{\GG_1}_{[12]}
(\id\otimes \wh{\theta}_{\GG_1\subseteq \GG})(\wW^{\GG_1})_{[14]}
(\id\otimes \gamma_{\GG_1\subseteq \GG})(\ww^{\GG_1})_{[13]}
\\
&=\wh{\oon{R}}_{1[24]}^{r,u}
\wh{\oon{R}}_{1[23]}
(\id\otimes \wh{\theta}_{\GG_1\subseteq \GG})(\wW^{\GG_1})_{[14]}
(\id\otimes \gamma_{\GG_1\subseteq \GG})(\ww^{\GG_1})_{[13]}
\ww^{\GG_1}_{[12]}
\wh{\oon{R}}_{1[23]}^*
\wh{\oon{R}}_{1[24]}^{r,u *}.
\end{split}\]
Combining this with \eqref{eq51} we get
\[\begin{split}
&\quad\;
\ww^{\GG_1}_{[12]}
(\id\otimes \wh{\theta}_{\GG_1\subseteq \GG})(\wW^{\GG_1})_{[14]}
(\id\otimes \gamma_{\GG_1\subseteq \GG})(\ww^{\GG_1})_{[13]}\\
&=
\wh{\oon{R}}_{1[24]}^{r,u}
\wh{\oon{R}}_{1[23]}
(\id\otimes \wh{\theta}_{\GG_1\subseteq \GG})(\wW^{\GG_1})_{[14]}
\wh{\oon{R}}_{1 [23]}^*
\ww^{\GG_1}_{[12]}
(\id\otimes\gamma_{\GG_1\subseteq\GG})(\ww^{\GG_1})_{[13]}
\wh{\oon{R}}_{1[24]}^{r,u *}\\
&=
\wh{\oon{R}}_{1[24]}^{r,u}
(\id\otimes \wh{\theta}_{\GG_1\subseteq \GG})(\wW^{\GG_1})_{[14]}
\ww^{\GG_1}_{[12]}
\wh{\oon{R}}_{1[24]}^{r,u *}
(\id\otimes\gamma_{\GG_1\subseteq\GG})(\ww^{\GG_1})_{[13]},
\end{split}\]
hence after cancelling of $(\id\otimes\gamma_{\GG_1\subseteq\GG})(\ww^{\GG_1})_{[13]}$ and cutting of the third leg we get
\begin{equation}\label{eq52}
\ww^{\GG_1}_{[12]}(\id\otimes \wh{\theta}_{\GG_1\subseteq \GG})(\wW^{\GG_1})_{[13]}=
\wh{\oon{R}}_{1[23]}^{r,u}
(\id\otimes \wh{\theta}_{\GG_1\subseteq \GG})(\wW^{\GG_1})_{[13]}\ww^{\GG_1}_{[12]}
\wh{\oon{R}}_{1[23]}^{r,u *}.
\end{equation}
Applying $\id\otimes \Delta_{\whG_1}^{r,u}\otimes \id$ we have
\[
\wW^{\GG_1}_{[13]}\ww^{\GG_1}_{[12]}
(\id\otimes \wh{\theta}_{\GG_1\subseteq \GG})(\wW^{\GG_1})_{[14]}=
\wh{\oon{R}}_{1[24]}^{r,u}
\wh{\oon{R}}_{1[34]}^{u}
(\id\otimes \wh{\theta}_{\GG_1\subseteq \GG})(\wW^{\GG_1})_{[14]}\wW^{\GG_1}_{[13]}\ww^{\GG_1}_{[12]}
\wh{\oon{R}}_{1[34]}^{u *}
\wh{\oon{R}}_{1[24]}^{r,u *}
\]
and similarly combining this with \eqref{eq52} we arrive at
\[
\wW^{\GG_1}_{[12]}(\id\otimes \wh{\theta}_{\GG_1\subseteq \GG})(\wW^{\GG_1})_{[13]}=
\wh{\oon{R}}_{1[23]}^{u}
(\id\otimes \wh{\theta}_{\GG_1\subseteq \GG})(\wW^{\GG_1})_{[13]}\wW^{\GG_1}_{[12]}
\wh{\oon{R}}_{1[23]}^{u *}
\]
or equivalently
\[
(\id\otimes (\id\otimes\wh{\theta}_{\GG_1\subseteq\GG})\Delta_{\whG_1}^{u,op})(\wW^{\GG_1})=
\wh{\oon{R}}_{1[23]}^{u}
(\id\otimes (\id\otimes\wh{\theta}_{\GG_1\subseteq\GG})\Delta_{\whG_1}^u)(\wW^{\GG_1})
\wh{\oon{R}}_{1[23]}^{u *}.
\]
Slicing off the first leg, gives the first equation of the claim. The second one can be proved in an analogous way.
\end{proof}

If $\M,\N$ are von Neumann algebras which carry (an appropriate) action, we can consider their braided tensor product defined using one of the bicharacters $\wh{\oon{R}}_{12},\wh{\oon{R}}_{1},\wh{\oon{R}}_{2},\wh{\oon{R}}$. Proposition \ref{prop19} tells us that we do not risk running into an ambiguity by writing simply $\M\ov\boxtimes\N$. For example, assume that $\M$ carries an action of $\GG_1$ and $\N$ an action of $\GG$, then we can define the braided tensor product using the bicharacter $\wh{\oon{R}}_{1}$. We could also restrict the  action $\GG\curvearrowright \N$ to $\GG_2$ and define $\M\ov\boxtimes \N$ using $\wh{\oon{R}}_{12}$ -- both constructions give us isomorphic von Neumann algebras (and actually \emph{equal} if we choose coherent implementations). In the next proposition we introduce a canonical action on the braided tensor product (c.f.~\cite[Proposition 4.2]{MeyerRoyWoronowiczII}). Recall the notation $U^{\N}\rest_{\HH}$ introduced in Section \ref{sec:preliminaries}.

\begin{proposition}\label{prop20}
Let $\M,\N$ be von Neumann algebras.
\begin{enumerate}
\item If $\GG\curvearrowright \M,\N$, then there is a unique action $\GG\curvearrowright \M\ov\boxtimes \N$ such that the embeddings of $\M,\N$ are $\GG$-equivariant. If the representations $U^{\M},U^{\N}$ implement the actions on $\M,\N$, then $U^{\M}\otop U^{\N}$ implements the action on $\M\ov\boxtimes\N$.
\item If $\GG_1\curvearrowright \M,\GG\curvearrowright\N$, then there is a unique action $\GG_1\curvearrowright \M\ov\boxtimes \N$ such that the embeddings of $\M,\N$ are $\GG_1$-equivariant (considering the restricted action $\GG_1\curvearrowright \N$). If the representations $U^{\M},U^{\N}$ implement the actions on $\M,\N$, then $U^{\M}\otop (U^{\N}\rest_{\GG_1})$ implements the action on $\M\ov\boxtimes\N$.
\item If $\GG\curvearrowright \M,\GG_2\curvearrowright\N$, then there is a unique action $\GG_2\curvearrowright \M\ov\boxtimes \N$ such that the embeddings of $\M,\N$ are $\GG_2$-equivariant (considering the restricted action $\GG_2\curvearrowright \M$). If the representations $U^{\M},U^{\N}$ implement the actions on $\M,\N$, then $(U^{\M}\rest_{\GG_2})\otop U^{\N}$ implements the action on $\M\ov\boxtimes\N$.
\end{enumerate}
\end{proposition}

Note that in cases $(1),(2),(3)$ braided tensor product is constructed using respective bicharacter $\wh{\oon{R}},\wh{\oon{R}}_1,\wh{\oon{R}}_2$.

\begin{proof}
Notice first that if there is an action on $\M\ov\boxtimes\N$ for which the embeddings of $\M,\N$ are equivariant, then it has to be unique. Consequently, it is enough to show existence of such an action, and that it is implemented by representations as in the claim. Let $U^{\M},U^{\N}$ be implementations of actions on $\M,\N$.

(Case $(1)$). We claim that the map
\begin{equation}\label{eq33}
\M\ov\boxtimes \N\ni z \mapsto 
(U^{\M}\otop U^{\N})^* (\I\otimes z)
(U^{\M}\otop U^{\N})
\in \LL^{\infty}(\GG)\bar\otimes (\M\ov\boxtimes\N)
\end{equation}
is well defined, and that the embeddings of $\M,\N$ are equivariant; then \eqref{eq33} is the desired action of $\GG$. Take $m\in\M,n\in\N$. We have
\begin{equation}\begin{split}\label{eq34}
&\quad\;
(U^{\M}\otop U^{\N})^* (\I\otimes \iota_{\N}(n))
(U^{\M}\otop U^{\N})=
 U^{\N *}_{[13]} U^{\M *}_{[12]}(\I\otimes \I\otimes n)
U^{\M }_{[12]} U^{\N}_{[13]}\\
&=
 U^{\N *}_{[13]} (\I\otimes \I\otimes n)U^{\N}_{[13]}=
(\id\otimes \iota_{\N})\alpha^{\N}(n)\in \LL^{\infty}(\GG)\bar\otimes (\M\ov\boxtimes\N).
\end{split}\end{equation}
Next, using the universal version of \eqref{eq2},
\begin{equation}\begin{split}\label{eq35}
&\quad\;
(U^{\M}\otop U^{\N})^* (\I\otimes \iota_{\M}(m))
(U^{\M}\otop U^{\N})=
 U^{\N  *}_{[13]} U^{\M *}_{[12]}
(\phi_{\M}\otimes \phi_{\N})(\wh{\oon{R}}^u)_{[23]}
m_{[2]}
(\phi_{\M}\otimes \phi_{\N})(\wh{\oon{R}}^u)_{[23]}^*
U^{\M }_{[12]} U^{\N}_{[13]}\\
&=
(\id\otimes \phi_{\M}\otimes \phi_{\N})(
\wW^{\GG *}_{[13]}\wW^{\GG *}_{[12]} \wh{\oon{R}}^u_{[23]})
m_{[2]}
(\id\otimes \phi_{\M}\otimes \phi_{\N})(
\wh{\oon{R}}^{u *}_{[23]}\wW^{\GG }_{[12]} 
\wW^{\GG }_{[13]})\\
&=
(\id\otimes \phi_{\M}\otimes \phi_{\N})(
(\id\otimes\Delta_{\whG}^{op})(\wW^{\GG *})\wh{\oon{R}}^u_{[23]})
m_{[2]}
(\id\otimes \phi_{\M}\otimes \phi_{\N})(
\wh{\oon{R}}^{u *}_{[23]}
(\id\otimes\Delta_{\whG}^{op})(\wW^{\GG })
)\\
&=
(\id\otimes \phi_{\M}\otimes \phi_{\N})(
\wh{\oon{R}}^u_{[23]}
(\id\otimes\Delta_{\whG})(\wW^{\GG *}))
m_{[2]}
(\id\otimes \phi_{\M}\otimes \phi_{\N})(
(\id\otimes\Delta_{\whG})(\wW^{\GG })
\wh{\oon{R}}^{u *}_{[23]}
)\\
&=
(\phi_{\M}\otimes \phi_{\N})(
\wh{\oon{R}}^u)_{[23]}
U^{\M *}_{[12]} U^{\N *}_{[13]}
m_{[2]}
U^{\N}_{[13]} U^{\M}_{[12]}
(\phi_{\M}\otimes \phi_{\N})(
\wh{\oon{R}}^u)_{[23]}^*\\
&=
(\id\otimes \iota_{\M})\alpha^{\M}(m)\in \LL^{\infty}(\GG)\bar\otimes (\M\ov\boxtimes \N).
\end{split}\end{equation}
Equations \eqref{eq34}, \eqref{eq35} prove that \eqref{eq33} is well defined and that the embeddings $\iota_{\M},\iota_{\N}$ are equivariant.

(Case $(2)$). Define the map
\begin{equation}\label{eq50}
\M\ov\boxtimes \N\ni z \mapsto 
(U^{\M}\otop (U^{\N}\rest_{\GG_1}))^* (\I\otimes z)
(U^{\M}\otop (U^{\N}\rest_{\GG_1}))
\in \LL^{\infty}(\GG_1)\bar\otimes (\M\ov\boxtimes\N).
\end{equation}
As in case $(1)$, we need to show that this map is well defined and that the embeddings of $\M,\N$ are equivariant. Recall that the restricted action $\GG_1\curvearrowright \N$ is implemented by $U^{\N}\rest_{\GG_1}$ and $\phi_{\N} \wh{\theta}_{\GG_1\subseteq \GG}$ (Section \ref{sec:preliminaries}). A calculation analogous to \eqref{eq34} shows that $\iota_{\N}$ is equivariant. Next, using Lemma \ref{lemma11},
\[\begin{split}
&\quad\;
(U^{\M}\otop (U^{\N}\rest_{\GG_1}))^* (\I\otimes \iota_{\M}(m))
(U^{\M}\otop (U^{\N}\rest_{\GG_1}))\\
&=
( U^{\N}\rest_{\GG_1})^{  *}_{[13]} U^{\M *}_{[12]}
(\phi_{\M}\otimes \phi_{\N})(\wh{\oon{R}}^u_1)_{[23]}
m_{[2]}
(\phi_{\M}\otimes \phi_{\N})(\wh{\oon{R}}^u_1)_{[23]}^*
U^{\M }_{[12]} (U^{\N}\rest_{\GG_1})_{[13]}\\
&=
(\id\otimes \phi_{\M}\otimes \phi_{\N})
((\id\otimes (\id\otimes \wh{\theta}_{\GG_1\subseteq\GG})\Delta_{\wh{\GG}_1}^{op})(
\wW^{\GG_1 *})\,  \wh{\oon{R}}^u_{1 [23]})
m_{[2]}
(\id\otimes \phi_{\M}\otimes \phi_{\N})( \wh{\oon{R}}_{1 [23]}^{u *}\, (\id\otimes (\id\otimes \wh{\theta}_{\GG_1\subseteq\GG})\Delta_{\wh{\GG}_1}^{op})(
\wW^{\GG_1 }))\\
&=
(\phi_{\M}\otimes \phi_{\N})(
\wh{\oon{R}}_{1 }^u)_{[23]}
(\id\otimes\phi_{\M}\otimes \phi_{\N}\wh{\theta}_{\GG_1\subseteq\GG})
(\wW^{\GG_1 *}_{[12]}\wW^{\GG_1 *}_{[13]})
m_{[2]}
(\id\otimes\phi_{\M}\otimes \phi_{\N}\wh{\theta}_{\GG_1\subseteq\GG})
(\wW^{\GG_1 }_{[13]}\wW^{\GG_1 }_{[12]})
(\phi_{\M}\otimes \phi_{\N})(
\wh{\oon{R}}_{1 }^u)_{[23]}^*\\
&=
(\phi_{\M}\otimes \phi_{\N})(
\wh{\oon{R}}_{1 }^u)_{[23]}
U^{\M *}_{[12]} m_{[2]}
U^{\M}_{[12]}
(\phi_{\M}\otimes \phi_{\N})(
\wh{\oon{R}}_{1 }^u)_{[23]}^*=
(\id\otimes \iota_{\M})\alpha^{\M}(m)\in \LL^{\infty}(\GG_1)
\bar\otimes (\M\ov\boxtimes\N),
\end{split}\]
which proves the claim. Case $(3)$ can be proven in an analogous way.
\end{proof}

\begin{corollary}\label{cor3}
Let $\M,\wt{\M},\N,\wt{\N}$ be von Neumann algebras and $\vartheta_1\colon \M\rightarrow \wt{\M},\vartheta_2\colon \N\rightarrow\wt{\N}$ normal, completely bounded maps.
\begin{enumerate}
\item If $\GG\curvearrowright \M,\wt{\M},\N,\wt{\N}$ and $\vartheta_1,\vartheta_2$ are $\GG$-equivariant, then $\vartheta_1\boxtimes\vartheta_2$ is $\GG$-equivariant.
\item If $\GG_1\curvearrowright \M,\wt{\M},\GG\curvearrowright\N,\wt{\N}$, $\vartheta_1$ is $\GG_1$-equivariant and $\vartheta_2$ is $\GG$-equivariant, then $\vartheta_1\boxtimes\vartheta_2$ is $\GG_1$-equivariant.
\item If $\GG\curvearrowright \M,\wt{\M},\GG_2\curvearrowright\N,\wt{\N}$, $\vartheta_1$ is $\GG$-equivariant and $\vartheta_2$ is $\GG_2$-equivariant, then $\vartheta_1\boxtimes\vartheta_2$ is $\GG$-equivariant.
\end{enumerate}
\end{corollary}

\begin{proof}
We prove $(1)$, the proof of the remaining assertions being analogous. Take $x\in \M,y\in \N$. Denote by $\iota_{\M},\iota_{\N}$ the embeddings into $\M\ov\boxtimes\N$ and by $\iota_{\wt\M},\iota_{\wt\N}$ the embeddings into $\wt{\M}\ov\boxtimes\wt{\N}$. The claim follows from the following calculation on generators:
\[\begin{split}
&\quad\;
\alpha^{\wt{\M}\ov\boxtimes\wt{\N}} (\vartheta_1\boxtimes \vartheta_2)(\iota_{\M}(x)\iota_{\N}(y))=
\alpha^{\wt{\M}\ov\boxtimes\wt{\N}} 
(\iota_{\wt\M}(\vartheta_1(x))\iota_{\wt\N}(\vartheta_2(y)))=
(\id\otimes\iota_{\wt\M})\alpha^{\wt\M}(\vartheta_1(x))
(\id\otimes\iota_{\wt\N})\alpha^{\wt\N}(\vartheta_2(y))\\
&=
(\id\otimes\iota_{\wt\M}\vartheta_1)\alpha^{\M}(x)
(\id\otimes\iota_{\wt\N}\vartheta_2)\alpha^{\N}(y)=
(\id\otimes (\vartheta_1\boxtimes\vartheta_2))
\bigl((\id\otimes\iota_{\M})\alpha^{\M}(x)\,
(\id\otimes\iota_{\N})\alpha^{\N}(y)\bigr)\\
&=
(\id\otimes (\vartheta_1\boxtimes\vartheta_2))
\alpha^{\M\ov\boxtimes\N}(
\iota_{\M}(x) \iota_{\N}(y)).
\end{split}\]
\end{proof}

The coherence property proved in Proposition \ref{prop19}, and recalled in special case before Proposition \ref{prop20}, has its dynamical counterpart. More precisely, assume that the quasi-triangular locally compact quantum group $\GG$ acts on $\M,\N$. Then part $(1)$ of Proposition \ref{prop20} gives us a canonical action $\GG\curvearrowright \M\ov\boxtimes\N$, which can be restricted to $\GG_1$ or $\GG_2$. We can also consider restricted actions $\GG_1\curvearrowright \M$ or $\GG_2\curvearrowright \N$, then parts $(2)$ and $(3)$ of Proposition \ref{prop20} give us respectively: an action $\GG_1\curvearrowright \M\ov\boxtimes\N$ on the braided tensor product constructed using $\wh{\oon{R}}_1$, and an action $\GG_2\curvearrowright \M\ov\boxtimes\N$ on the braided tensor product constructed using $\wh{\oon{R}}_2$.

\begin{proposition}\label{prop21}
In the above situation, we have the following properties:
\begin{enumerate}
\item $\GG_1\curvearrowright \M\ov\boxtimes\N$ is the restriction of $\GG\curvearrowright \M\ov\boxtimes\N$ to $\GG_1$,
\item $\GG_2\curvearrowright \M\ov\boxtimes\N$ is the restriction of $\GG\curvearrowright \M\ov\boxtimes\N$ to $\GG_2$.
\end{enumerate}
\end{proposition}

\begin{proof}
Choose $U^{\M}, U^{\N}$, implementations of actions $\GG\curvearrowright \M,\N$ on some Hilbert spaces $\msf{H}_{\M},\msf{H}_{\N}$. According to Proposition \ref{prop20}, the action $\GG_1\curvearrowright \M\ov\boxtimes \N$ is implemented by $(U^{\M}\rest_{\GG_1})\otop (U^{\N}\rest_{\GG_1})$, hence claim $(1)$ follows from
\[\begin{split}
&\quad\;
(U^{\M}\rest_{\GG_1})\otop (U^{\N}\rest_{\GG_1})=
(U^{\M}\rest_{\GG_1})_{[12]} (U^{\N}\rest_{\GG_1})_{[13]}=
(\id\otimes \phi_{\M}\wh{\theta}_{\GG_1\subseteq \GG}\otimes 
\phi_{\N}\wh{\theta}_{\GG_1\subseteq \GG})
(\id\otimes \Delta_{\whG_1}^{u,op})(\wW^{\GG_1})\\
&=
(\id\otimes (\phi_{\M}\otimes \phi_{\N}) \Delta_{\whG}^{u,op}
\wh{\theta}_{\GG_1\subseteq \GG})(\wW^{\GG_1})=
(U^{\M}\otop U^{\N})\rest_{\GG_1}.
\end{split}\]
Case $(2)$ is analogous.
\end{proof}

Next we want to prove an associativity of braided tensor product; $\M\ov\boxtimes (\N\ov\boxtimes \oon{P})=(\M\ov\boxtimes \N)\ov\boxtimes \oon{P}$. The minimal situation when both sides of this equality make sense, is when $\M,\N,\oon{P}$ are von Neumann algebras equipped with actions $\GG_1\curvearrowright \M, \GG\curvearrowright\N,\GG_2\curvearrowright\oon{P}$. Then we have canonical actions $\GG_2\curvearrowright \N\ov\boxtimes\oon{P}, \GG_1\curvearrowright \M\ov\boxtimes\N$ and the second braided tensor product is well defined. Choose arbitrary representations and implementations of actions on $\msf{H}_{\M},\msf{H}_{\N},\msf{H}_{\oon{P}}$, then both $\M\ov\boxtimes (\N\ov\boxtimes \oon{P})$ and $(\M\ov\boxtimes \N)\ov\boxtimes \oon{P}$ are subspaces of $\B(\msf{H}_{\M}\otimes \msf{H}_{\N}\otimes \msf{H}_{\oon{P}})$.

\begin{proposition}\label{prop11}
The von Neumann algebras $\M\ov\boxtimes (\N\ov\boxtimes \oon{P})$ and $(\M\ov\boxtimes \N)\ov\boxtimes\oon{P}$ in $\B(\msf{H}_{\M}\otimes\msf{H}_{\N}\otimes\msf{H}_{\oon{P}})$ are equal. Furthermore the associated embeddings of $\M,\N,\oon{P}$ are equal.
\end{proposition}

\begin{remark}\noindent
\begin{enumerate}
\item If $\M$ carries an action of $\GG$, then we can construct $\M\ov\boxtimes (\N\ov\boxtimes \oon{P}), (\M\ov\boxtimes \N)\ov\boxtimes\oon{P}$ also in another way; using respectively $\wh{\oon{R}}_{2}, \wh{\oon{R}}_{2}$ and $\wh{\oon{R}},\wh{\oon{R}}_2$. In this situation (the obvious modification of) Proposition \ref{prop11} still holds, by Propositions \ref{prop19}, \ref{prop21}. A similar remark applies when $\N$ carries an action of $\GG$ or both $\M,\N$ are equipped with an action of $\GG$.
\item When $\GG\curvearrowright \M$ (or $\GG\curvearrowright \oon{P}$ or both), then Proposition \ref{prop20} equips the iterated braided tensor product with a canonical action of $\GG_2$ (or $\GG_1$ or $\GG$). Since choosing different ordering of parentheses gives us the same algebra \emph{and} the same embeddings, this action is also uniquely determined.
\end{enumerate}
\end{remark}

\begin{proof}[Proof of Proposition \ref{prop11}]
The action $\GG_2\curvearrowright \N\ov\boxtimes\oon{P}$ is implemented by $(U^{\N}\rest_{\GG_2})\otop U^{\oon{P}}=
(\id\otimes (\phi_{\N}\wh{\theta}_{\GG_2\subseteq \GG}\otimes\phi_{\oon{P}})\Delta_{\whG_2}^{u,op})\wW^{\GG_2}$, hence $\phi_{\N\ov\boxtimes\oon{P}}=(\phi_{\N}\wh{\theta}_{\GG_2\subseteq \GG}\otimes\phi_{\oon{P}})\Delta_{\whG_2}^{u,op}$. Using this and \eqref{eq53} we have
\[\begin{split}
&\quad\;
\M\ov\boxtimes (\N\ov\boxtimes \oon{P})\\
&=
\ov{\lin}^{\,\swot}\{
(\phi_{\M}\otimes \phi_{\N\ov\boxtimes\oon{P}})(\wh{\oon{R}}^u_{12})
m_{[1]}
(\phi_{\M}\otimes \phi_{\N\ov\boxtimes\oon{P}})(\wh{\oon{R}}^u_{12})^*
x_{[23]}
\mid m\in\M,x\in\N\ov\boxtimes\oon{P}\}\\
&=
\ov{\lin}^{\,\swot}\{
(\phi_{\M}\otimes \phi_{\N}\wh{\theta}_{\GG_2\subseteq\GG}\otimes \phi_{\oon{P}})(\wh{\oon{R}}^u_{12[12]}\wh{\oon{R}}^u_{12[13]})
m_{[1]}
(\phi_{\M}\otimes \phi_{\N}\wh{\theta}_{\GG_2\subseteq\GG}\otimes \phi_{\oon{P}})
(\wh{\oon{R}}^{u *}_{12[13]}\wh{\oon{R}}^{u *}_{12[12]})
x_{[23]}\\
&\quad\quad\quad\quad
\quad\quad\quad\quad
\quad\quad\quad\quad
\quad\quad\quad\quad
\quad\quad\quad\quad
\quad\quad\quad\quad
\quad\quad\quad\quad
\quad\quad\quad\quad
\quad\quad\quad\quad
\mid m\in\M,x\in\N\ov\boxtimes\oon{P}\}\\
&=
\ov{\lin}^{\,\swot}\{
(\phi_{\M}\otimes \phi_{\N}\wh{\theta}_{\GG_2\subseteq\GG}\otimes \phi_{\oon{P}})(\wh{\oon{R}}^u_{12[12]}\wh{\oon{R}}^u_{12[13]})
m_{[1]}
(\phi_{\M}\otimes \phi_{\N}\wh{\theta}_{\GG_2\subseteq\GG}\otimes \phi_{\oon{P}})
(\wh{\oon{R}}^{u *}_{12[13]}\wh{\oon{R}}^{u *}_{12[12]})
\\&\quad\quad\quad\quad
\quad\quad\quad\quad\quad
\quad\quad\quad\quad\quad
\quad\quad\quad\quad\quad
(\phi_{\N}\otimes\phi_{\oon{P}})(\wh{\oon{R}}_{2}^u)_{[23]}n_{[2]}
(\phi_{\N}\otimes\phi_{\oon{P}})(\wh{\oon{R}}_2^u)^*_{[23]}
p_{[3]}
\mid m\in\M,n\in\N,p\in\oon{P}\}\\
&=
\ov{\lin}^{\,\swot}\{
(\phi_{\M}\otimes \phi_{\N}\otimes \phi_{\oon{P}})(\wh{\oon{R}}^u_{1[12]}\wh{\oon{R}}^u_{12[13]})
m_{[1]}
(\phi_{\M}\otimes \phi_{\N}\otimes \phi_{\oon{P}})
(\wh{\oon{R}}^{u *}_{12[13]}\wh{\oon{R}}^{u *}_{1[12]})
\\&\quad\quad\quad\quad
\quad\quad\quad\quad\quad
\quad\quad\quad\quad\quad
\quad\quad\quad\quad\quad
(\phi_{\N}\otimes\phi_{\oon{P}})(\wh{\oon{R}}_{2}^u)_{[23]}n_{[2]}
(\phi_{\N}\otimes\phi_{\oon{P}})(\wh{\oon{R}}_2^u)^*_{[23]}
p_{[3]}
\mid m\in\M,n\in\N,p\in\oon{P}\}
\end{split}\]
and on the other hand, since $\GG_1\curvearrowright \M\ov\boxtimes\N$ is implemented by $U^{\M}\otop(U^{\N}\rest_{\GG_1})$ with $\phi_{\M\ov\boxtimes\N}=(\phi_{\M}\otimes \phi_{\N}\wh{\theta}_{\GG_1\subseteq \GG})\Delta_{\whG_1}^{u,op}$,
\[\begin{split}
&\quad\;
(\M\ov\boxtimes \N)\ov\boxtimes \oon{P}\\
&=
\ov{\lin}^{\,\swot}\{
(\phi_{\M\ov\boxtimes\N}\otimes \phi_{\oon{P}})(\wh{\oon{R}}^u_{12})
y_{[12]}
(\phi_{\M\ov\boxtimes\N}\otimes\phi_{\oon{P}})(\wh{\oon{R}}^u_{12})^*
p_{[3]}
\mid y\in \M\ov\boxtimes \N, p\in \oon{P}\}\\
&=
\ov{\lin}^{\,\swot}\{
(\phi_{\M}\otimes \phi_{\N}\wh{\theta}_{\GG_1\subseteq \GG}\otimes \phi_{\oon{P}})(\wh{\oon{R}}^u_{12[23]}\wh{\oon{R}}^u_{12[13]})
y_{[12]}
(\phi_{\M}\otimes \phi_{\N}\wh{\theta}_{\GG_1\subseteq \GG}\otimes \phi_{\oon{P}})(\wh{\oon{R}}^{u *}_{12[13]}
\wh{\oon{R}}^{u *}_{12[23]})
p_{[3]}\\
&\quad\quad\quad\quad
\quad\quad\quad\quad
\quad\quad\quad\quad
\quad\quad\quad\quad
\quad\quad\quad\quad
\quad\quad\quad\quad
\quad\quad\quad\quad
\quad\quad\quad\quad
\quad\quad\quad\quad
\mid y\in \M\ov\boxtimes \N, p\in \oon{P}\}\\
&=
\ov{\lin}^{\,\swot}\{
(\phi_{\M}\otimes \phi_{\N}\wh{\theta}_{\GG_1\subseteq \GG}\otimes \phi_{\oon{P}})(\wh{\oon{R}}^u_{12[23]}\wh{\oon{R}}^u_{12[13]})
(\phi_{\M}\otimes \phi_{\N})(\wh{\oon{R}}^u_{1})_{[12]}
m_{[1]}
(\phi_{\M}\otimes \phi_{\N})(\wh{\oon{R}}^u_{1})_{[12]}^*n_{[2]}\\
&\quad\quad\quad\quad
\quad\quad\quad\quad\quad
\quad\quad\quad\quad\quad
\quad\quad\quad\quad\quad
(\phi_{\M}\otimes \phi_{\N}\wh{\theta}_{\GG_1\subseteq \GG}\otimes \phi_{\oon{P}})(\wh{\oon{R}}^{u *}_{12[13]}
\wh{\oon{R}}^{u *}_{12[23]})
p_{[3]}
\mid m\in\M,n\in\N, p\in \oon{P}\}\\
&=
\ov{\lin}^{\,\swot}\{
(\phi_{\M}\otimes \phi_{\N}\otimes \phi_{\oon{P}})(\wh{\oon{R}}^u_{2[23]}\wh{\oon{R}}^u_{12[13]} \wh{\oon{R}}^u_{1[12]})
m_{[1]}
(\phi_{\M}\otimes \phi_{\N})(\wh{\oon{R}}^u_{1})_{[12]}^* n_{[2]}\\
&\quad\quad\quad\quad
\quad\quad\quad\quad\quad
\quad\quad\quad\quad\quad
\quad\quad\quad\quad\quad
(\phi_{\M}\otimes \phi_{\N}\otimes \phi_{\oon{P}})(\wh{\oon{R}}^{u *}_{12[13]}
\wh{\oon{R}}^{u *}_{2[23]})
p_{[3]}
\mid m\in\M,n\in\N, p\in \oon{P}\}.
\end{split}\]
Observe (using Lemma \ref{lemma11}) that we have a version of the Yang-Baxter equation (c.f.~Proposition \ref{prop1}):
\[\begin{split}
&\quad\;
\wh{\oon{R}}^u_{2[23]}\wh{\oon{R}}^u_{12[13]} \wh{\oon{R}}^u_{1[12]}=
(\id\otimes \wh{\theta}_{\GG_1\subseteq\GG}\otimes \id)(\wh{\oon{R}}^u_{12[23]}\wh{\oon{R}}^u_{12[13]}) \wh{\oon{R}}^u_{1[12]}=
(\id\otimes \wh{\theta}_{\GG_1\subseteq\GG}\otimes \id)(\Delta_{\whG_1}^{u, op}\otimes \id)(\wh{\oon{R}}^u_{12}) \wh{\oon{R}}^u_{1[12]}\\
&=
\wh{\oon{R}}_{1[12]}^u
(\id\otimes \wh{\theta}_{\GG_1\subseteq\GG}\otimes \id)
(\Delta_{\whG_1}^{u}\otimes \id)(\wh{\oon{R}}^u_{12})
\wh{\oon{R}}_{1[12]}^{u *}
 \wh{\oon{R}}^u_{1[12]}=
 \wh{\oon{R}}^u_{1[12]}
(\id\otimes\wh{\theta}_{\GG_1\subseteq\GG}\otimes\id)(
\wh{\oon{R}}^u_{12[13]} \wh{\oon{R}}^u_{12[23]})\\
&=
 \wh{\oon{R}}^u_{1[12]}
\wh{\oon{R}}^u_{12[13]} \wh{\oon{R}}^u_{2[23]}.
\end{split}\]
Using this twice, we further have
\[\begin{split}
&\quad\;
(\M\ov\boxtimes \N)\ov\boxtimes \oon{P}\\
&=
\ov{\lin}^{\,\swot}\{
(\phi_{\M}\otimes \phi_{\N}\otimes \phi_{\oon{P}})(
 \wh{\oon{R}}^u_{1[12]}
\wh{\oon{R}}^u_{12[13]}
\wh{\oon{R}}^u_{2[23]}
)
m_{[1]}
(\phi_{\M}\otimes \phi_{\N})(\wh{\oon{R}}^u_{1})_{[12]}^* n_{[2]}\\
&\quad\quad\quad\quad
\quad\quad\quad\quad\quad
\quad\quad\quad\quad\quad
\quad\quad\quad\quad\quad
(\phi_{\M}\otimes \phi_{\N}\otimes \phi_{\oon{P}})(\wh{\oon{R}}^{u *}_{12[13]}
\wh{\oon{R}}^{u *}_{2[23]})
p_{[3]}
\mid m\in\M,n\in\N, p\in \oon{P}\}\\
&=
\ov{\lin}^{\,\swot}\{
(\phi_{\M}\otimes \phi_{\N}\otimes \phi_{\oon{P}})(
 \wh{\oon{R}}^u_{1[12]}
\wh{\oon{R}}^u_{12[13]}
)
m_{[1]}
(\phi_{\M}\otimes \phi_{\N}\otimes \phi_{\oon{P}})(
\wh{\oon{R}}^{u *}_{12[13]} \wh{\oon{R}}^{u *}_{1[12]}
)\\
&\quad\quad\quad\quad
(\phi_{\M}\otimes \phi_{\N}\otimes \phi_{\oon{P}})(
\wh{\oon{R}}^{u }_{1[12]}
\wh{\oon{R}}^{u }_{12[13]} 
\wh{\oon{R}}^u_{2[23]}
\wh{\oon{R}}^{u *}_{1[12]})n_{[2]}
(\phi_{\M}\otimes \phi_{\N}\otimes \phi_{\oon{P}})(\wh{\oon{R}}^{u *}_{12[13]}
\wh{\oon{R}}^{u *}_{2[23]})
p_{[3]}
\mid m\in\M,n\in\N, p\in \oon{P}\}\\
&=
\ov{\lin}^{\,\swot}\{
(\phi_{\M}\otimes \phi_{\N}\otimes \phi_{\oon{P}})(
 \wh{\oon{R}}^u_{1[12]}
\wh{\oon{R}}^u_{12[13]}
)
m_{[1]}
(\phi_{\M}\otimes \phi_{\N}\otimes \phi_{\oon{P}})(
\wh{\oon{R}}^{u *}_{12[13]} \wh{\oon{R}}^{u *}_{1[12]}
)\\
&\quad\quad\quad\quad
\quad\quad\quad\quad
(\phi_{\M}\otimes \phi_{\N}\otimes \phi_{\oon{P}})(
\wh{\oon{R}}^u_{2[23]}
\wh{\oon{R}}^{u }_{12[13]} )n_{[2]}
(\phi_{\M}\otimes \phi_{\N}\otimes \phi_{\oon{P}})(\wh{\oon{R}}^{u *}_{12[13]}
\wh{\oon{R}}^{u *}_{2[23]})
p_{[3]}
\mid m\in\M,n\in\N, p\in \oon{P}\}\\
&=
\ov{\lin}^{\,\swot}\{
(\phi_{\M}\otimes \phi_{\N}\otimes \phi_{\oon{P}})(
 \wh{\oon{R}}^u_{1[12]}
\wh{\oon{R}}^u_{12[13]}
)
m_{[1]}
(\phi_{\M}\otimes \phi_{\N}\otimes \phi_{\oon{P}})(
\wh{\oon{R}}^{u *}_{12[13]} \wh{\oon{R}}^{u *}_{1[12]}
)\\
&\quad\quad\quad\quad
\quad\quad\quad\quad
\quad\quad\quad\quad
(\phi_{\M}\otimes \phi_{\N}\otimes \phi_{\oon{P}})(
\wh{\oon{R}}^u_{2[23]} )n_{[2]}
(\phi_{\M}\otimes \phi_{\N}\otimes \phi_{\oon{P}})(
\wh{\oon{R}}^{u *}_{2[23]})
p_{[3]}
\mid m\in\M,n\in\N, p\in \oon{P}\},
\end{split}\]
which shows equality $\M\ov\boxtimes (\N\ov\boxtimes\oon{P})=(\M\ov\boxtimes \N)\ov\boxtimes \oon{P}$. Note that our calculation shows that generators in both descriptions are equal; we have an equality of embeddings.
\end{proof}

\begin{corollary}\label{cor1}
For a finite family $\M_1,\dotsc,\M_n\,(n\ge 2)$ of von Neumann algebras equipped with actions of $\GG$, we can define their braided tensor product $\M_1\ov\boxtimes \cdots \ov\boxtimes \M_n$. It is a von Neumann algebra equipped with embeddings $\iota_{\M_i}\colon \M_i\rightarrow \M_1\ov\boxtimes \cdots \ov\boxtimes \M_n\,(1\le i \le n)$ and the unique action of $\GG$ such that $\iota_{\M_i}\,(1\le i \le n)$ are equivariant.
\end{corollary}

In other words, we can define $\M_1\ov\boxtimes \cdots \ov\boxtimes \M_n$ and the rest of the structure, by putting parentheses in an arbitrary order; this can be formally proved by an induction on $n$.

\section{Infinite braided tensor product}\label{sec:inf}
Let $(\GG,\wh{\oon{R}}$) be a quasi-triangular locally compact quantum group and $(\M_{-n},\omega_{-n})\,(n\in\NN)$ $\GG$-von Neumann algebras together with chosen faithful, normal, $\GG$-invariant states. In this section, we will define infinite braided tensor product $\ov{\boxtimes}_{n=\infty}^{1} (\M_{-n},\omega_{-n})$. This will be a von Neumann algebra carrying a canonical action of $\GG$, together with equivariant embeddings.

Represent $\M_{-n}$ in the standard way on $\LL^2(\M_{-n})$ and let $\Omega_{-n}$ be the unique unit vector in the standard positive cone, such that $\omega_{-n}(m)=\la \Omega_{-n}| m\Omega_{-n}\ra $ for $m\in\M_{-n}$ (\cite[Lemma 2.10]{Haagerup}). Let $\msf{H}$ be the infinite tensor product $\otimes_{n=\infty}^{1}(\LL^2(\M_{-n}),\Omega_{-n})$, see \cite[Section 1 XIV]{TakesakiIII}. For $N\in\NN$, consider $\M_{-N}\ov\boxtimes\cdots\ov\boxtimes \M_{-1}$ represented on $\LL^2(\M_{-N})\otimes\cdots\otimes \LL^2(\M_{-1})$ and its amplification
\[
\oon{P}_N= \{\I^{\otimes \infty}\otimes x\mid 
x\in \M_{-N}\ov\boxtimes\cdots\ov\boxtimes \M_{-1}\}\subseteq \B\bigl( \otimes_{n=\infty}^{1} (\LL^2(\M_{-n}),\Omega_{-n})\bigr).
\]
The sequence $(\oon{P}_N)_{N=1}^{\infty}$ is increasing. Indeed, this follows directly from the definition; one of the generators of $\M_{-N-1}\ov\boxtimes (\M_{-N}\ov\boxtimes\cdots\ov\boxtimes\M_{-1})$ is $\I\otimes x$ for $x\in \M_{-N}\ov\boxtimes\cdots\ov\boxtimes\M_{-1}$.

\begin{definition}\label{def3}
We define the infinite braided tensor product of the family $(\M_{-n},\omega_{-n})\,(n\in \NN)$ as
\[
\ov\boxtimes_{n=\infty}^{1} (\M_{-n},\omega_{-n})=
\ov{\bigcup_{N=1}^{\infty} \{\I^{\otimes\infty} \otimes x\mid 
x\in \M_{-N}\ov\boxtimes\cdots\ov\boxtimes \M_{-1} \}}^{\,\swot}\subseteq 
\B\bigl( \otimes_{n=\infty}^{1} (\LL^2(\M_{-n}),\Omega_{-n})\bigr).
\]
\end{definition}

Next, for $ 1\le k \le K$ we define embeddings
\begin{equation}\label{eq39}
\iota_{-k}\colon \M_{-k}\ni x \mapsto \I^{\otimes \infty}\otimes 
\iota_{\M_{-k}}^{\M_{-K}\ov\boxtimes\cdots\ov\boxtimes\M_{-1}}(x)\in \ov{\boxtimes}_{n=\infty}^{1} (\M_{-n},\omega_{-n})
\end{equation}
where $\iota_{\M_{-k}}^{\M_{-K}\ov\boxtimes\cdots\ov\boxtimes\M_{-1}}$ is the canonical embedding of $\M_{-k}$ into $\M_{-K}\ov\boxtimes\cdots\ov\boxtimes \M_{-1}$ (see Corollary \ref{cor1}).

\begin{remark}\noindent
\begin{enumerate}
\item If the action $\GG\curvearrowright \M_{-n}$ is trivial for all $n\in \NN$, then $\ov\boxtimes_{n=\infty}^{1} (\M_{-n},\omega_{-n})$ coincides with the usual infinite tensor product, compare Example \ref{example1} (the same is true, if $\GG$ is classical and $\wh{\oon{R}}=\I$). Consequently, the structure of the von Neumann algebra $\ov\boxtimes_{n=\infty}^{1} (\M_{-n},\omega_{-n})$ depends on the choice of states $\omega_{-n}$ (\cite[Section XIV]{TakesakiIII}).
\item The definition of the infinite braided tensor product depends on the order of the algebras, and the order structure of $-\NN$; this shouldn't come as a surprise in light of Example \ref{example5}. 
\item In the algebraic context, a version of infinite braided tensor product was considered in \cite{InfiniteMajid}.
\end{enumerate}
\end{remark}

\begin{lemma}\label{lemma8}
The embeddings $\iota_{-k}\,(k\in \NN)$ are well defined. 
\end{lemma}

\begin{proof}
Fix $1\le k\le K$. We want to show that $\iota_{-k}$ does not depend on $K$. This is an easy consequence of Corollary \ref{cor1}, as follows:
\[
\iota_{\M_{-k}}^{\M_{-K}\ov\boxtimes \cdots\ov\boxtimes \M_{-1}}(x)=
\iota_{\M_{-k}}^{(\M_{-K}\ov\boxtimes \cdots\ov\boxtimes \M_{-k-1}\!)\,\ov\boxtimes ( \M_{-k}\ov\boxtimes\cdots\ov\boxtimes \M_{-1})}(x)=
\I^{\otimes (K-k)}\otimes 
\iota_{\M_{-k}}^{\M_{-k}\ov\boxtimes \cdots\ov\boxtimes \M_{-1}}(x)
\]
for $x\in \M_{-k}$.
\end{proof}

\begin{proposition}\label{prop12}
There exists a unique action of $\GG$ on $\ov{\boxtimes}_{n=\infty}^{1} (\M_{-n},\omega_{-n})$ which makes the embeddings $\iota_{-k}\,(k\in \NN)$ equivariant.
\end{proposition}

Before we prove this proposition, we need a lemma concerning standard implementations. We can be more general: let $\alpha^{\N}\colon \HH\curvearrowright \N$ be an action of a locally compact quantum group on a von Neumann algebra, $\rho$ strictly positive, self-adjoint operator affiliated with $\LL^{\infty}(\HH)$ and $\omega$ an n.s.f.~weight on $\N$ which is $\rho$-invariant (see \cite[Definition 2.3]{VaesUnitary}). Assume furthermore $\Delta_{\HH}(\rho)=\rho\otimes\rho$, or equivalently $\Delta_{\HH}(\rho^{it})=\rho^{it}\otimes \rho^{it}\,(t\in\RR)$. Represent $\N$ on its standard Hilbert space $\LL^2(\N)$, identified with the GNS Hilbert space $\msf{H}_\omega$ for $\omega$. Let $U^{\N}$ be the \emph{standard} implementation of action, i.e.~the one introduced in\footnote{More precisely, because we want $U^{\N}$ to be a representation of $\HH$, we take the adjoint of the unitary from \cite{VaesUnitary}.} \cite[Definition 3.6]{VaesUnitary}).

\begin{lemma}\label{lemma9}
For $n\in \mf{N}_{\omega}$ and $x\in \mf{N}_{\vp}$ such that $\Lambda_{\vp}(x)\in \Dom(\rho^{1/2})$ we have $\alpha^{\N}(n)(y\otimes \I)\in \mf{N}_{\vp\otimes\omega}$ and
\begin{equation}\label{eq42}
U^{\N *}(\rho^{1/2} \Lambda_{\vp}(x)\otimes \Lambda_{\omega}(n))=
\Lambda_{\vp\otimes\omega}(\alpha^{\N}(n)(x\otimes \I)).
\end{equation}
In particular, if $\omega$ is a $\GG$-invariant faithful normal state on $\N$ and $\Omega_{\N}\in \msf{H}_{\omega}$ is the unit vector in the standard positive cone which implements $\omega$, then $U^{N}(\xi\otimes \Omega_{\N})=\xi\otimes\Omega_{\N}$ for $\xi\in \LL^2(\HH)$.
\end{lemma}

\begin{proof}
The first part can be proved by first establishing the existence of a unitary representation satisfying \eqref{eq42} along the lines of \cite[Th{\'e}or{\`e}me 2.9]{Enock}, and then by reasoning analogous to \cite[Proposition 4.3]{VaesUnitary} (which is exactly our claim for $\rho=\delta^{-1}$). For the second part, if $\rho=\I$ and $\omega=\omega_{\Omega_{\N}}$, then for $x\in \mf{N}_{\vp}$ we have
\[
U^{\N *}( \Lambda_{\vp}(x)\otimes \Omega_{\N})=
U^{\N *}( \Lambda_{\vp}(x)\otimes \Lambda_{\omega}(\I))=
\Lambda_{\vp\otimes \omega}(\alpha^{\N}(\I) (x\otimes \I))=
\Lambda_{\vp}(x)\otimes \Omega_{\N},
\]
and the claim follows by approximation.
\end{proof}

\begin{proof}[Proof of Proposition \ref{prop12}]
The images of the maps $\iota_{-k}(k\in\NN)$ generate the whole von Neumann algebra, hence uniqueness of the action $\GG\curvearrowright \ov\boxtimes_{n=\infty}^{1} (\M_{-n},\omega_{-n})$ is clear. We need to show its existence. In the rest of the proof, legs of $\LL^{\infty}(\GG)\bar\otimes \B(\otimes_{n=\infty}^{1} (\LL^2(\M_{-n}),\Omega_{-n}) )$ will be numbered by $1, \dotsc,-3,-2,-1$. Let $U^{\M_{-n}}\in \LL^{\infty}(\GG)\bar\otimes \B(\LL^2(\M_{-n}))$ be the standard implementation of the action $\alpha^{\M_{-n}}$. We claim that the sequence of unitaries $( U^{\M_{-N}}_{[1,-N]}\cdots U^{\M_{-1}}_{[1,-1]})_{N=1}^{\infty}$ converges in $\sots$ to a unitary $U$. Recall that the unitary group is \sots-complete (\cite[Remark 4.10]{TakesakiI}), thus it is enough to show that for any vector $\Theta\in \LL^2(\GG)\otimes (\otimes_{n=\infty}^{1}(\LL^2(\M_{-n}),\Omega_{-n}))$, the sequences $\bigl( U^{\M_{-N}}_{[1,-N]}\cdots U^{\M_{-1}}_{[1,-1]} \Theta\bigr)_{N=1}^{\infty}$ and $\bigl( U^{\M_{-1} *}_{[1,-1]} \cdots U^{\M_{-N} *}_{[1,-N]}\Theta\bigr)_{N=1}^{\infty}$ are Cauchy. Since we are dealing with bounded sequences, it is enough to consider vectors in a linearly dense subset, so take $\Theta=\xi\otimes (\cdots\otimes \Omega_{-k-1}\otimes \zeta)$ for some $\xi\in \LL^2(\GG),k\in\NN, \zeta\in \LL^2(\M_{-k})\otimes\cdots\otimes \LL^2(\M_{-1})$. In this case, both sequences are eventually constant. Indeed, for $N>k$ we have using Lemma \ref{lemma9}
\[\begin{split}
&\quad\;
U^{\M_{-N} }_{[1,-N]} \cdots U^{\M_{-1} }_{[1,-1]}\Theta =
U^{\M_{-N} }_{[1,-N]} \cdots U^{\M_{-k-1}}_{[1,-k-1]}
\bigl(
U^{\M_{-k}}_{[1,-k]}\cdots U^{\M_{-1} }_{[1,-1]} ( 
 \xi\otimes (\cdots\otimes \Omega_{-k-1}\otimes \zeta))\bigr) \\
 &=
U^{\M_{-k}}_{[1,-k]}\cdots U^{\M_{-1} }_{[1,-1]} ( 
 \xi\otimes (\cdots\otimes \Omega_{-k-1}\otimes \zeta)).
\end{split}\]
In particular, this shows that $\bigl( U^{\M_{-N}}_{[1,-N]}\cdots U^{\M_{-1}}_{[1,-1]} \Theta\bigr)_{N=1}^{\infty}$ is Cauchy. Similarly, using again Lemma \ref{lemma9}
\[\begin{split}
&\quad\;
U^{\M_{-1} *}_{[1,-1]} \cdots U^{\M_{-N} *}_{[1,-N]}\Theta =
U^{\M_{-1} *}_{[1,-1]} \cdots U^{\M_{-k} *}_{[1,-k]}
\bigl(
U^{\M_{-k-1} *}_{[1,-k-1 ]}\cdots U^{\M_{-N} *}_{[1,-N]} ( 
 \xi\otimes (\cdots\otimes \Omega_{-k-1}\otimes \zeta))\bigr) \\
 &=
U^{\M_{-1} *}_{[1,-1]}\cdots U^{\M_{-k} *}_{[1,-k]} ( 
 \xi\otimes (\cdots\otimes \Omega_{-k-1}\otimes \zeta)).
\end{split}\]
We can conclude that there is a unitary $U$ such that $U^{\M_{-N} }_{[1,-N]} \cdots U^{\M_{-1} }_{[1,-1]}\xrightarrow[N\to\infty]{}U$ in \sots. Using the bicommutant theorem, we easily see that $U\in \LL^{\infty}(\GG)\bar\otimes \B(\otimes_{n=\infty}^{1}(\LL^2(\M_{-n}),\Omega_{-n}))$. Just as the tensor product of representations is a representation, we see that $U$ is a representation. Recall that multiplication on bounded sets is jointly continuous in \sot. Take $k\in \NN$ and $x\in \M_{-k}$. Using Corollary \ref{cor1} we have
\[\begin{split}
&\quad\;
U^* (\I\otimes \iota_{-k}(x))U=
\underset{N\to\infty}{\sot{-}\lim}\,
U^{\M_{-1} *}_{[1,-1]} \cdots U^{\M_{-N} *}_{[1,-N]}
(\I\otimes \iota_{-k}(x))
U^{\M_{-N} }_{[1,-N]} \cdots U^{\M_{-1} }_{[1,-1]}\\
&=
U^{\M_{-1} *}_{[1,-1]} \cdots U^{\M_{-k} *}_{[1,-k]}
(\I\otimes \I^{\otimes \infty} \otimes\iota_{-k}^{\M_{-k}\ov\boxtimes\cdots\ov\boxtimes\M_{-1}}(x))
U^{\M_{-k} }_{[1,-k]} \cdots U^{\M_{-1} }_{[1,-1]}\\
&=
\alpha^{\M_{-k}\ov\boxtimes\cdots\ov\boxtimes\M_{-1}} \bigl(\iota_{-k}^{\M_{-k}\ov\boxtimes\cdots\ov\boxtimes\M_{-1}} (x)\bigr)_{[1, -k,\dotsc,-1]}=
(\id\otimes \iota_{-k}^{\M_{-k}\ov\boxtimes\cdots\ov\boxtimes\M_{-1}}) (\alpha^{\M_{-k}}(x))_{[1, -k,\dotsc,-1]}=
(\id\otimes \iota_{-k}) (\alpha^{\M_{-k}}(x)).
\end{split}\]
This shows that the map
\[
\ov{\boxtimes}_{n=\infty}^{1} (\M_{-n},\omega_{-n})\ni z\mapsto 
U^* (\I\otimes z)U\in 
\LL^{\infty}(\GG)\bar\otimes \bigl(
\ov{\boxtimes}_{n=\infty}^{1} (\M_{-n},\omega_{-n})\bigr)
\]
is well defined, and defines an action $\GG\curvearrowright \ov{\boxtimes}_{n=\infty}^{1} (\M_{-n},\omega_{-n})$ for which the embeddings $\iota_{-k}(k\in \NN)$ are equivariant.
\end{proof}

\begin{remark}\label{remark3}\noindent
\begin{enumerate}
\item It follows from the proof of Proposition \ref{prop12}, that if $U^{\M_{-n}}$ are the standard implementations of actions $\alpha^{\M_{-n}}\colon\GG\curvearrowright \M_{-n}$, then the action $\GG\curvearrowright \ov{\boxtimes}_{n=\infty}^{1} (\M_{-n},\omega_{-n})$ is implemented by the unitary representation
\begin{equation}\label{eq36}
\underset{N\to\infty}{\sots{-}\lim}\, U^{\M_{-N}}_{[1,-N]}\cdots U^{\M_{-1}}_{[1,-1]}\in \LL^{\infty}(\GG)\bar\otimes \B\bigl(\otimes_{n=\infty}^{1} (\LL^2(\M_{-n}),\Omega_{-n})\bigr).
\end{equation}
\item If $\GG$ is a classical locally compact group, we can take $\wh{\oon{R}}=\I$ as the $\oon{R}$-matrix. In this case Proposition \ref{prop12} recovers the classical notion of the diagonal action on the infinite tensor product.
\end{enumerate}
\end{remark}

\section{Examples}\label{sec:examples}
In this section, we present several (classes of) examples of braided tensor products. First, we briefly discuss trivial cases, where the braided tensor product is easily seen to be equal to the usual tensor product. Next, less trivially, we obtain the same conclusion if one of the actions is inner. Next we show that construction of Houdayer \cite{Houdayer} can be realised as a braided tensor product. In the fourth example, we present a construction of a braided tensor product $\M\ov\boxtimes\N$ and normal functionals $\omega,\nu$ for which one cannot define functional ``$\omega\boxtimes\nu$''. In the last subsection we show how one can realise the crossed product as a  braided tensor product, and deduce from this several interesting phenomena.

\subsection{Example 1}\label{example1}
Let $\HH\curvearrowright\M,\GG\curvearrowright\N$ be locally compact quantum groups acting on von Neumann algebras, and $\wh{\mc{X}}\in \LL^{\infty}(\whH)\bar\otimes\LL^{\infty}(\whG)$ a bicharacter. Recall that to construct the braided tensor product, we choose arbitrary representations $\M\subseteq \B(\msf{H}_{\M}),\N\subseteq \B(\msf{H}_{\N})$ and implementations $U^{\M}=(\id\otimes\phi_{\M})\wW^{\HH}, U^{\N}=(\id\otimes \phi_{\N})\wW^{\GG}$ of these actions. Then
\[
\M\ov\boxtimes\N=
\ov{\lin}^{\,\swot}\{
(\phi_{\M}\otimes \phi_{\N})(\wh{\mc{X}}^u)
(m\otimes\I)
(\phi_{\M}\otimes \phi_{\N})(\wh{\mc{X}}^u)^*
(\I\otimes n)\mid m\in\M,n\in\N\}\subseteq \B(\msf{H}_{\M}\otimes \msf{H}_{\N}),
\]
which is a von Neumann algebra independent (up to canonical isomorphism) of the above choices. If the action $\HH\curvearrowright \M$ is trivial, then we can take $\phi_{\M}$ to be the counit, and by Remark \ref{remark2} we obtain $\M\ov\boxtimes\N=\M\bar\otimes \N$. Similarly, if the action $\GG\curvearrowright \N$ is trivial, then $\M\ov\boxtimes\N=\M\bar\otimes \N$.  The same conclusion holds for $\wh{\mc{X}}^u=\I$ and arbitrary actions.

\subsection{Example 2}
Let $\HH\curvearrowright\M,\GG\curvearrowright\N$ be locally compact quantum groups acting on von Neumann algebras, and $\wh{\mc{X}}\in \LL^{\infty}(\whH)\bar\otimes\LL^{\infty}(\whG)$ a bicharacter. Assume that the action $\GG\curvearrowright\N$ is inner. More precisely, let $U^{\N}=(\id\otimes \phi_{\N})\wW^{\GG}$ be a representation s.t.~$U^{\N}\in \LL^{\infty}(\GG)\bar\otimes \N$ and $\alpha^{\N}(n)=U^{\N *}(\I\otimes n)U^{\N}$ is an action of $\GG$ on $\N$. In other words, the action $\alpha^{\N}$ is \emph{cocycle equivalent} to the trivial action (\cite[Definition 1.5]{VaesUnitary}). For example, we can take $U^{\N}=\ww^{\GG}$ and $\N=\LL^{\infty}(\whG)$ or $\N=\B(\LL^2(\GG))$. Represent $\N$ on $\msf{H}_{\N}$, $\M$ on $\msf{H}_{\M}$ and let $U^{\M}=(\id\otimes \phi_{\M})\wW^{\HH}$ be an implementation of the action $\alpha^{\M}\colon\HH\curvearrowright\M$.

\begin{proposition}\label{prop23}
In this situation we have $\M\ov\boxtimes\N=\M\bar\otimes\N$, as an equality of subsets of $\B(\msf{H}_{\M}\otimes\msf{H}_{\N})$.
\end{proposition}

\begin{proof}
Recall that $\wh{\mc{X}}^u=(\Phi\otimes \id)\WW^{\GG}=(\id\otimes\wh\Phi)\WW^{\whH}$ (Lemma \ref{lemma1}). First we perform an auxilliary calculation
\[\begin{split}
\M\ov\boxtimes \N& =
\ov{\lin}^{\, \swot}\{
(\phi_{\M}\otimes \phi_{\N})(\wh{\mc{X}}^u)(m\otimes \I)
(\phi_{\M}\otimes \phi_{\N})(\wh{\mc{X}}^u)^* (\I\otimes n)\mid 
m\in\M,n\in\N\}\\
&=
\ov{\lin}^{\, \swot}\{
(\phi_{\M}\otimes \phi_{\N}\wh\Phi)(\WW^{\whH})(m\otimes \I)
(\phi_{\M}\otimes \phi_{\N}\wh\Phi)(\WW^{\whH})^* (\I\otimes n)\mid 
m\in\M,n\in\N\}\\
&=
\ov{\lin}^{\, \swot}\{
\bigl(
(\phi_{\N}\wh\Phi\otimes \phi_{\M})(\WW^{\HH *})(\I\otimes m)
(\phi_{\N}\wh\Phi\otimes \phi_{\M})(\WW^{\HH}) (n\otimes \I)
\bigr)_{[21]}\mid 
m\in\M,n\in\N\}\\
&=
\ov{\lin}^{\, \swot}\{
\bigl(
(\phi_{\N}^{\vN}\wh{\Phi}^{\vN}\otimes \id)
\alpha^{\M,u}(m) (n\otimes \I)
\bigr)_{[21]}\mid 
m\in\M,n\in\N\}.
\end{split}\]
Note that $\phi_{\N}^{\vN}\wh{\Phi}^{\vN}\colon \mrm{C}_0^{u}(\HH)^{**}\rightarrow \N$, hence using the universal version of the Podleś condition (Proposition \ref{prop7}) we can continue
\[\begin{split}
\M\ov\boxtimes \N &=
\ov{\lin}^{\, \swot}\{
\bigl(
(\phi_{\N}^{\vN}\wh{\Phi}^{\vN}\otimes \id)
\alpha^{\M,u}(m) (
\phi_{\N}^{\vN}\wh{\Phi}^{\vN} (x)
n\otimes \I)
\bigr)_{[21]}\mid 
m\in\M,x\in \mrm{C}_0^u(\HH)^{**}, n\in\N\}\\
&=
\ov{\lin}^{\, \swot}\{
\bigl(
(\phi_{\N}^{\vN}\wh{\Phi}^{\vN}\otimes \id)
(\alpha^{\M,u}(m) 
(x\otimes \I) )\,
(n\otimes \I)
\bigr)_{[21]}\mid 
m\in\M,x\in \mrm{C}_0^u(\HH)^{**}, n\in\N\}\\
&=
\ov{\lin}^{\, \swot}\{
\bigl(
(\phi_{\N}^{\vN}\wh{\Phi}^{\vN}\otimes \id)
(x\otimes m )\,
(n\otimes \I)
\bigr)_{[21]}\mid 
m\in\M,x\in \mrm{C}_0^u(\HH)^{**}, n\in\N\}\\
&=
\ov{\lin}^{\, \swot}\{
m\otimes \phi_{\N}^{\vN}\wh{\Phi}^{\vN}(x) n
\mid 
m\in\M,x\in \mrm{C}_0^u(\HH)^{**}, n\in\N\}
=\M\bar\otimes \N.
\end{split}\]
\end{proof}

By Proposition \ref{prop22}, a claim similar to Proposition \ref{prop23} holds if $\HH\curvearrowright\M$ is inner.

\subsection{Example 3}\label{example3} Let $\M,\N$ be von Neumann algebras equipped with actions $\alpha^{\M}\colon \wh{\RR}\curvearrowright \M$, $\alpha^{\N}\colon \RR\curvearrowright \N$, and $\wh{\mc{X}}=\ww^{\RR}\in \LL^{\infty}(\RR)\bar\otimes \LL^{\infty}(\wh{\RR})$ be the Kac-Takesaki operator. Then $\M\ov\boxtimes \N$ can be identified with the construction of Houdayer \cite[Definition 1.1]{Houdayer}, called the crossed product of $\M$ and $\N$. Indeed, to see this, represent $\M$ and $\alpha^{\M}$ in an arbitrary way. Next, represent $\N$ on $\LL^2(\RR)\bar\otimes \N$ via the (injective) normal map $\alpha^{\N}$, and take for the implementation of the action $\alpha^{\N}$ the unitary $\ww^{\RR}\otimes \I=
(\id\otimes (\pi_{\wh\RR}\otimes \I))(\wW^{\RR})$. Then for $m\in\M,n\in\N$ we have
\[\begin{split}
\iota_{\M}(m)\iota_{\N}(n)=
(\phi_{\M}\otimes \id)(\ww^{\RR})_{[12]}
(m\otimes \I\otimes \I)
(\phi_{\M}\otimes \id)(\ww^{\RR})_{[12]}^*
\alpha^{\N}(n)_{[23]}=
\alpha^{\M}(m)_{[21]} \alpha^{\N}(n)_{[23]}.
\end{split}\]
Hence we see that, up to passing between left and right actions, $\M\ov\boxtimes\N$ is isomorphic with the crossed product of Houdayer. This construction was used in \cite{Houdayer} to produce an uncountable family of pairwise non-isomorphic type $\oon{III}_1$ factors which are non-full, and have the same $\tau$ invariant \cite[Theorem 3.9]{Houdayer}.

\subsection{Example 4}
In this example we use infinite (braided) tensor products to show that even though the braided tensor product ``$\omega\boxtimes\nu$'' of equivariant functionals is always well defined (Proposition \ref{prop8}), it might not exist without the equivariance assumption (see Proposition \ref{prop13}).

Let us identify $\wh{\RR}$ with $\RR$ via the pairing $\la \xi,x\ra=e^{2\pi i \xi x}$. For $a\in \RR$, let $(\RR,\wh{\oon{R}}_a)$ be a quasi-triangular (classical) locally compact quantum group $\RR$ with an $\oon{R}$-matrix $\wh{\oon{R}}_a(\xi,\eta)=e^{2\pi i a \xi \eta}\,(\xi,\eta\in\wh{\RR})$. For $n\in-\NN$, consider $\M_n=\LL^{\infty}(\TT)$, for $\TT \subseteq \CC$ the unit circle, with the usual faithful normal state $\omega_n=\int_{\TT}$ given by integration with respect to the normalised Lebesgue measure. Fix the action $\alpha^{\M_n}\colon \RR\curvearrowright \M_n$ given by rotation, i.e.~$\alpha^{\M_n}(f)(x,\lambda)=f(\lambda e^{2\pi i x})$ for $x\in \RR,\lambda\in \TT, f\in\M_n=\LL^{\infty}(\TT)$, and observe that the state $\omega_n$ is invariant. Let $U^{\M_n}\in \LL^{\infty}(\RR)\bar\otimes\B(\LL^2(\TT))$ be the standard implementation and $\Omega_n=\I\in\LL^2(\TT)$ the positive unit vector which implements $\omega_n$.

We have a ``diagonal'' action of $\RR$ on the (non- or trivially- braided) infinite tensor product $\M=\bar\otimes_{n=-\infty}^{-1} (\M_n,\omega_n)$ (see \cite[Definition 1.6 XIV]{TakesakiIII}). Indeed, take $a=0$, then $\wh{\oon{R}}_0=\I$ and $\M$ is equal to the infinite braided tensor product $\ov\boxtimes_{n=-\infty}^{-1} (\M_n,\omega_n)$ and the canonical embeddings $\M_k\hookrightarrow \ov\boxtimes_{n=-\infty}^{-1} (\M_n,\omega_n)$ are the usual embeddings onto $k$'th tensor factor (Definition \ref{def3}, Remark \ref{remark3}). Proposition \ref{prop12} gives us an action $\alpha^{\M}\colon \RR\curvearrowright \M$ for which embeddings are equivariant; this is the diagonal action. Furthermore, by Remark \ref{remark3} we know that $\alpha^{\M}$ is implemented by representation
\[
U^{\M}=\underset{N\to\infty}{\sots{-}\lim}\, U^{\M_{-N}}_{[1,-N]}\cdots U^{\M_{-1}}_{[1,-1]}\in \LL^{\infty}(\RR)\bar\otimes \B\bigl(\otimes_{n=-\infty}^{-1} (\LL^2(\M_n),\Omega_n)\bigr).
\]
Let $\N=\M=\bar\otimes_{n=-\infty}^{-1} (\N_n,\nu_n)$ be another copy of $\bar\otimes_{n=-\infty}^{-1} (\LL^{\infty}(\TT),\int_{\TT})$ with the same diagonal action of $\RR$ implemented by $U^{\M}$.\\

Fix now $a\in \RR\setminus \QQ$. We will be interested in the braided tensor product $\M\ov\boxtimes\M$, constructed using the above structure. Let us denote by $\iota^1,\iota^2$ the canonical embeddings of $\M$ into $\M\ov\boxtimes\M$. Since $a$ is irrational, we can choose $k_1=0$ and $k_2,k_3,\dotsc\in\NN$ such that the numbers $e^{2\pi i a k_n} ( n\in\NN)$ are pairwise distinct. For $N\in\NN$, let $0<\delta_N\le\tfrac{1}{2N}$ be small enough so that the arc $\Omega_N=\{e^{2\pi i s}\mid -\delta_N\le s \le \delta_N\}\subseteq \TT$ satisfies
\begin{equation}\label{eq38}
e^{-2\pi ia k_n}\Omega_N\cap e^{-2\pi i a k_{n'}}\Omega_N=\emptyset\quad(1\le n\neq n' \le N).
\end{equation}
Set $\boldsymbol{\lambda}\in \LL^{\infty}(\TT)$ to be the identity function, define $g_N= \tfrac{1}{N} \sum_{n=1}^{N} \boldsymbol{\lambda}^{k_n}$ and for $N \in\NN,K\in\! -\NN$
\begin{equation}\label{eq40}
F_{N,K}=\I^{\otimes \infty}\otimes 1_{\Omega_N}\otimes \I^{\otimes (|K|-1)}\in \M,\quad G_{N,K}=\I^{\otimes\infty}\otimes g_N\otimes\I^{\otimes (|K|-1)}\in \M.
\end{equation}

\begin{lemma}\label{lemma10}
For every $N\in\NN,K\in \!-\NN$, we have $\|F_{N,K}\|=\|1_{\Omega_N}\|=\|G_{N,K}\|=\|g_N\|=1$, $\omega_{K}(1_{\Omega_N})\in \left[0,\tfrac{1}{N}\right]$, $\omega_{K}(g_N) = \tfrac{1}{N}$ and $\|\iota^1(F_{N,K})\iota^2(G_{N,K})\|\le \tfrac{1}{N}$.
\end{lemma}

\begin{proof}
The first claim is straightforward: $\|F_{N,K}\|=\|1_{\Omega_N}\|=1$ and $\|G_{N,K}\|=\|g_N\|=1$ follows from evaluation at $1$. Next, $\omega_K(1_{\Omega_N})\le \tfrac{1}{N}$ follows from $\delta_N\le \tfrac{1}{2N}$ and
\[
\omega_K(g_N)=\tfrac{1}{N} \sum_{n=1}^{N} \int_{\TT} \lambda^{k_n} \md \lambda =\tfrac{1}{N}.
\]
To calculate the norm of $\iota^1(F_{N,K})\iota^2(G_{N,K})$ it will be most convenient to use Proposition \ref{prop4}. Denote by $\iota^1_k$ the embedding of $\M_k=\LL^{\infty}(\TT)$ onto the $k$'th tensor factor, for the first copy of $\M=\bar\otimes_{n=-\infty}^{-1} (\LL^{\infty}(\TT),\omega_n)$ and similarly $\iota^2_k$ for the second copy of $\M$. We have
\[
\|\iota^1(F_{N,K})\iota^2(G_{N,K})\|=
\|\wh{\oon{R}}_{a [13]} \alpha^{\M}(F_{N,K})_{[12]}
\wh{\oon{R}}_{a [13]}^* \alpha^{\M}(G_{N,K})_{[34]}\|,
\]
where the second norm is calculated in $\B(\LL^2(\RR))\bar\otimes \M\bar\otimes \B(\LL^2(\RR))\bar\otimes \M$. Since $\alpha^{\M}(F_{N,K})=(\id\otimes \iota^1_{K})\alpha^{\M_{K}}(1_{\Omega_{N}})\in \LL^{\infty}(\RR)\bar\otimes( \bar\otimes_{n=-\infty}^{-1} (\M_n,\omega_n))$ has non-trivial factors only on legs $1$, $K$ (and $\I$ on the remaining legs), and similarly $\alpha^{\M}(G_{N,K})=(\id\otimes\iota^2_{K})\alpha^{\M_{K}}(g_N)$, we have in fact
\begin{equation}\begin{split}\label{eq37}
\|\iota^1(F_{N,K})\iota^2(G_{N,K})\|&=
\|\wh{\oon{R}}_{a [13]} \alpha^{\M_{K}}(1_{\Omega_N})_{[12]}
\wh{\oon{R}}_{a [13]}^* \alpha^{\M_{K}}(g_N)_{[34]}\|\\
&=
\| \alpha^{\M_{K}}(1_{\Omega_N})_{[12]}
\wh{\oon{R}}_{a [13]}^* \alpha^{\M_{K}}(g_N)_{[34]} \wh{\oon{R}}_{a [13]}\|,
\end{split}\end{equation}
with norm calculated in $\B(\LL^2(\RR))\bar\otimes \M_{K}\bar\otimes \B(\LL^2(\RR))\bar\otimes \M_{K}$. Next we identify $\wh{\oon{R}}_{a [13]}^* \alpha^{\M_{K}}(g_N)_{[34]} \wh{\oon{R}}_{a [13]}$. To do this, let us introduce the Fourier transform $\mc{F}_{\RR}$ defined via $\mc{F}_{\RR}(f)(\xi)=\int_{\RR} \ov{\la \xi,x\ra} f(x)\md x$, which is a unitary identification between $\LL^2(\RR)$ and $\LL^2(\wh\RR)$ (in abstract setting, it is taken as the identity). Under this correspondence, $\lambda_x^{\RR}\in \oon{vN}(\RR)$ is mapped to $\ov{\la\cdot,x\ra}\in \LL^{\infty}(\wh{\RR})$ and $\la \xi,\cdot\ra =e^{2\pi i \xi \bullet}\in \LL^{\infty}(\RR)$ to $\lambda^{\wh\RR}_\xi\in\oon{vN}(\wh\RR)$ \cite[Page 102]{Folland} (we write additional superscripts to avoid confusion). Using this identification we see that
\[
\wh{\oon{R}}_{a }^* (\I\otimes \la \xi , \cdot\ra ) \wh{\oon{R}}_{a }\simeq 
\wh{\oon{R}}_{a }^* (\I\otimes \lambda^{\wh\RR}_\xi) \wh{\oon{R}}_{a }=
\wh{\oon{R}}_{a }^* \wh{\oon{R}}_{a }(\cdot, -\xi + \cdot)
(\I\otimes\lambda_\xi^{\wh\RR})
=
\la \cdot,- a \xi\ra \otimes \lambda_\xi^{\wh\RR}\simeq 
\lambda_{a\xi}^{\RR}\otimes \la \xi , \cdot\ra \in \oon{vN}(\RR)\bar\otimes \LL^{\infty}(\RR)
\]
for $\xi\in \RR$, hence as $\alpha^{\M_{K}}(\boldsymbol{\lambda})=\la 1 , \cdot \ra \otimes \boldsymbol{\lambda}$, we obtain
\[\begin{split}
&\quad\;
\wh{\oon{R}}_{a[13]}^* \alpha^{\M_{K}}(g_N)_{[34]}
\wh{\oon{R}}_{a[13]}=
\tfrac{1}{N}\sum_{n=1}^{N}
\wh{\oon{R}}_{a[13]}^* \alpha^{\M_{K}}(\boldsymbol{\lambda}^{k_n})_{[34]}
\wh{\oon{R}}_{a[13]}\\
&=
\tfrac{1}{N}\sum_{n=1}^{N}
\wh{\oon{R}}_{a[13]}^* 
(\la k_n,\cdot\ra \otimes \boldsymbol{\lambda}^{k_n})_{[34]}
\wh{\oon{R}}_{a[13]}=
\tfrac{1}{N}\sum_{n=1}^{N}
(\lambda_{a k_n}^{\RR} \otimes \la k_n,\cdot\ra \otimes \boldsymbol{\lambda}^{k_n})_{[134]}.
\end{split}\]
We can plug this into \eqref{eq37}:
\[\begin{split}
&\quad\;
\|\iota^1(F_{N,K})\iota^2(G_{N,K})\|=
\bigl\|
\tfrac{1}{N}\sum_{n=1}^{N}
\alpha^{\M_{K}}(1_{\Omega_{N}})_{[12]}
(\lambda_{a k_n}^{\RR} \otimes \la k_n,\cdot\ra \otimes \boldsymbol{\lambda}^{k_n})_{[134]}\bigr\|.
\end{split}\]
To calculate this norm, take any $h\in \LL^2(\RR)\otimes \LL^2(\TT)\otimes\LL^2(\RR)\otimes\LL^2(\TT)$ with norm $1$. Using \eqref{eq38}, we can bound
\[\begin{split}
&\quad\;
\bigl\|
\tfrac{1}{N}\sum_{n=1}^{N}
\alpha^{\M_{K}}(1_{\Omega_{N}})_{[12]}
(\lambda_{a k_n}^{\RR} \otimes \la k_n,\cdot\ra \otimes \boldsymbol{\lambda}^{k_n})_{[134]}(h)\bigr\|^2\\
&=
\tfrac{1}{N^2} \int_{\RR}\int_{\TT}\int_{\RR}\int_{\TT}
\bigl| \sum_{n=1}^{N} 
\alpha^{\M_{K}}(1_{\Omega_{N}})_{[12]}
(\lambda_{a k_n}^{\RR} \otimes \la k_n,\cdot\ra \otimes \boldsymbol{\lambda}^{k_n})_{[134]}(h)\bigr|^2(x,\lambda,y,\rho)
\md \rho\md y \md \lambda \md x \\
&=
\tfrac{1}{N^2} \int_{\RR}\int_{\TT}\int_{\RR}\int_{\TT}
\bigl| \sum_{n=1}^{N} 
1_{\Omega_{\N}}(\lambda e^{2\pi i x})
(\lambda_{a k_n}^{\RR} \otimes \la k_n,\cdot\ra \otimes \boldsymbol{\lambda}^{k_n})_{[134]}(h)
(x,\lambda,y,\rho)
\bigr|^2
\md \rho\md y \md \lambda \md x \\
&=
\tfrac{1}{N^2} \int_{\RR}\int_{\TT}\int_{\RR}\int_{\TT}
\bigl| \sum_{n=1}^{N} 
1_{\Omega_{\N}}(\lambda e^{2\pi i x})
\la k_n,y\ra \rho^{k_n} h(x-ak_n,\lambda,y,\rho)
\bigr|^2
\md \rho\md y \md \lambda \md x \\
&=
\tfrac{1}{N^2} \int_{\RR}\int_{\TT}\int_{\RR}\int_{\TT}
\bigl| \sum_{n=1}^{N} 
1_{\Omega_{\N}}(\lambda e^{2\pi i (x+ak_n)})
\la k_n,y\ra \rho^{k_n} h(x,\lambda,y,\rho)
\bigr|^2
\md \rho\md y \md \lambda \md x \\
&=
\tfrac{1}{N^2} \int_{\RR}\int_{\TT}\int_{\RR}\int_{\TT}
 \sum_{n=1}^{N}
1_{e^{-2\pi i a k_n} \Omega_{\N}}(\lambda e^{2\pi i x})
\bigl| 
\la k_n,y\ra \rho^{k_n} h(x,\lambda,y,\rho)
\bigr|^2
\md \rho\md y \md \lambda \md x \\
&=
\tfrac{1}{N^2} \int_{\RR}\int_{\TT}\int_{\RR}\int_{\TT}
1_{\bigcup_{n=1}^{N}e^{-2\pi i a k_n} \Omega_{\N}}(\lambda e^{2\pi i x})
|h(x,\lambda,y,\rho)|^2
\md \rho\md y \md \lambda \md x 
\le 
\tfrac{1}{N^2}  \|h\|_2^2=\tfrac{1}{N^2}.
\end{split}\]
This ends the proof.
\end{proof}

With this lemma in hand, we can prove the following interesting property of $\M\ov\boxtimes\M$; recall that $\iota^1,\iota^2$ denote the canonical embeddings $\M\rightarrow \M\ov\boxtimes\M$.

\begin{proposition}\label{prop13}
There are bounded, normal functionals $\rho,\nu\in \M_*$ such that there is no $\kappa\in (\M\ov\boxtimes\M)^*$ with $\kappa(\iota^1(x)\iota^2(y))=\rho(x)\nu(y)$ for $x,y\in \M$.
\end{proposition}

\begin{proof}
For $N\in\NN$, let $F_{5^N,-N},G_{5^N,-N}$ be the functions defined in \eqref{eq40}. Since $\|1_{\Omega_{5^N}}\|=\|g_{5^N}\|=1$, we can find states $\rho_{-N},\nu_{-N}\in \LL^1(\TT)$ such that $\rho_{-N}(1_{\Omega_{5^N}}) \ge \tfrac{1}{2}$, $\nu_{-N}(g_{5^N})\ge \tfrac{1}{2}$. Define functionals $\rho,\nu$ via norm convergent series
\[
\rho=\sum_{n=-\infty}^{-1} \tfrac{1}{2^{|n|}} \cdots\otimes\omega_{n-1}\otimes 
\rho_{n}\otimes \cdots\otimes \rho_{-1}\in \M_*,\quad
\nu=\sum_{n=-\infty}^{-1} \tfrac{1}{2^{|n|}} \cdots\otimes\omega_{n-1}\otimes 
\nu_{n}\otimes \cdots\otimes \nu_{-1}\in \M_*
\]
and assume by contradiction that there is $\kappa\in (\M\ov\boxtimes\M)^*$ as in the claim (note that we do not assume that $\kappa$ is normal). Then using Lemma \ref{lemma10}
\begin{equation}\label{eq41}
\bigl|\kappa\bigl( \iota^1(F_{5^N,-N})\iota^2(G_{5^N,-N})\bigr)\bigr|\le 
\|\kappa\| \|\iota^1(F_{5^N,-N})\iota^2(G_{5^N,-N})\|\le \tfrac{\|\kappa\|}{5^N}
\end{equation}
and on the other hand, as $\rho_n,\omega_n,\nu_n$ are states, using again Lemma \ref{lemma10}
\[\begin{split}
&\quad\;
\kappa\bigl( \iota^1(F_{5^N,-N})\iota^2(G_{5^N,-N})\bigr)=
\rho(F_{5^N,-N}) \nu(G_{5^N,-N})\\
&=
\bigl(\sum_{n=-\infty}^{-N}
\tfrac{1}{2^{|n|}} \rho_{-N}(1_{\Omega_{5^N}})
+
\sum_{n=-N+1}^{-1}
\tfrac{1}{2^{|n|}} \omega_{-N}(1_{\Omega_{5^N}})
\bigr)\,
\bigl(\sum_{n=-\infty}^{-N}
\tfrac{1}{2^{|n|}} \nu_{-N}(g_N)
+
\sum_{n=-N+1}^{-1} \tfrac{1}{2^{|n|}}\omega_{-N}(g_N)
\bigr)\\
&\ge 
\bigl(\sum_{n=-\infty}^{-N}
\tfrac{1}{2^{|n|}} \rho_{-N}(1_{\Omega_{5^N}})
\bigr)\,
\bigl(\sum_{n=-\infty}^{-N}
\tfrac{1}{2^{|n|}} \nu_{-N}(g_N)
\bigr)\ge 
\bigl(\sum_{n=-\infty}^{-N}
\tfrac{1}{2^{|n|}} \tfrac{1}{2}
\bigr)\,
\bigl(\sum_{n=-\infty}^{-N}
\tfrac{1}{2^{|n|}} \tfrac{1}{2}\bigr)=
\tfrac{1}{4^N},
\end{split}\]
which contradicts \eqref{eq41} and $\|\kappa\|<+\infty$.
\end{proof}

\subsection{Example 5}
In this subsection we show how one can realise crossed products as braided tensor products.

Let $\M\subseteq \B(\msf{H}_{\M})$ be a von Neumann algebra with an action $\alpha^{\M}\colon\GG\curvearrowright \M$ of a locally compact quantum group $\GG$, and let $U^{\M}=(\id\otimes\phi_{\M})(\wW^{\GG})$ be an implementation. Consider the bicharacter $\wh{\mc{X}}=\ww^{\GG}\in \LL^{\infty}(\GG)\bar\otimes \LL^{\infty}(\whG)$. On $\LL^{\infty}(\whG)$ we have a translation action $\Delta_{\whG}\colon \whG\curvearrowright \LL^{\infty}(\whG)$, hence we can consider the braided tensor product $\LL^{\infty}(\whG)\ov\boxtimes \M$ defined using this structure.

\begin{proposition}\label{prop17}
The von Neumann algebras $\LL^{\infty}(\whG)\ov\boxtimes \M$ and $\GG\ltimes \M$ are unitarily equivalent.
\end{proposition}

\begin{proof}
Note that the action $\Delta_{\whG}$ is implemented by $\ww^{\whG}=(\id\otimes\pi_{\GG})\wW^{\whG}$. By definition
\[
\LL^{\infty}(\whG)\ov\boxtimes\M=
\ov{\lin}^{\,\swot}
\{
(\pi_{\GG}\otimes \phi_{\M})(\WW^{\GG})(\wh{x}\otimes \I)
(\pi_{\GG}\otimes \phi_{\M})(\WW^{\GG *})(\I\otimes m)\mid
\wh{x}\in \LL^{\infty}(\whG),m\in\M\}
\]
and conjugating by the unitary $(\pi_{\GG}\otimes\phi_{\M})(\WW^{\GG *})$ gives
\[\begin{split}
&\quad\;
\ov{\lin}^{\,\swot}
\{
(\wh{x}\otimes \I)
(\pi_{\GG}\otimes \phi_{\M})(\WW^{\GG *})(\I\otimes m)
(\pi_{\GG}\otimes \phi_{\M})(\WW^{\GG})\mid
\wh{x}\in \LL^{\infty}(\whG),m\in\M\}\\
&=
\ov{\lin}^{\,\swot}
\{
(\wh{x}\otimes \I)
U^{\M *}(\I\otimes m)U^{\M}\mid
\wh{x}\in \LL^{\infty}(\whG),m\in\M\}\\
&=
\ov{\lin}^{\,\swot}
\{
(\wh{x}\otimes \I)
\alpha^{\M}(m)\mid
\wh{x}\in \LL^{\infty}(\whG),m\in\M\}=\GG\ltimes \M.
\end{split}\]
\end{proof}

As a consequence, we obtain several concrete examples with interesting behaviour.

\begin{corollary}\label{cor2}
There exist braided tensor products of the form $\LL^{\infty}(\whG)\ov\boxtimes \M$ with the following properties:
\begin{enumerate}
\item $\LL^{\infty}(\whG)$ is of type II, $\M$ is of type I and $\LL^{\infty}(\whG)\ov\boxtimes \M$ is of type I,
\item $\LL^{\infty}(\whG)$ is of type III, $\M$ is of type I and $\LL^{\infty}(\whG)\ov\boxtimes \M$ is of type I,
\item $\LL^{\infty}(\whG)$ and $\M$ have $w^*$ CBAP but $\LL^{\infty}(\whG)\ov\boxtimes \M$ doesn't have $w^*$ CBAP,
\item\label{cor24} $\LL^{\infty}(\whG)$ and $\M$ are factors, but $\LL^{\infty}(\whG)\ov\boxtimes \M$ has diffuse center, in particular is not a factor.
\end{enumerate}
\end{corollary}

Note that such examples don't exist for the usual tensor product.

\begin{proof}
It is well known, that if $\GG$ acts on $\LL^{\infty}(\GG)$ with respect to the translation action $\Delta_{\GG}$, then $\GG\ltimes \LL^{\infty}(\GG)$ is isomorphic with the type I factor $\B(\LL^2(\GG))$. This proves existence of examples (1) and (2), when we take for $\GG$ an appropriate classical group. Concretely, for type $\oon{II}_1$ take the free group on two generators $F_2$, for $\oon{II}_{\infty}$ the product group $F_2\times (ax+b)$, where $(ax+b)$ is the ``$ax+b$'' group (see \cite[Page 260]{Folland}), and for types $\oon{III}_\lambda\,(0\le \lambda\le 1)$ e.g.~the examples constructed in \cite{Sutherland} (in the case of types $\oon{II}_{\infty}$, $\oon{III}_{\lambda}$ one can take here also \emph{discrete, quantum} examples, see \cite[Proposition 4.8]{KrajczokWasilewski} and \cite[Corollary 4.4, Theorem 4.5]{KrajczokSoltanExamples}).

It is known that although $\oon{SL}(2,\ZZ)$ and $\ZZ^2$ are weakly amenable with Cowling-Haagerup constant equal $1$, their semidirect product is not weakly amenable (\cite{HaagerupnonCBAP}). By \cite[Proposition 2.2]{Sutherland} we know that the group von Neumann algebra of a semidirect product of groups is isomorphic with the respective crossed product. For discrete groups,  the Cowling-Haagerup constant is equal to the constant of its group von Neumann algebra (\cite[Theorem 2.6]{HaagerupnonCBAP}), which shows existence of example (3). Existence of example (4) follows from a work of Vaes (\cite[Proposition B]{VaesFactoriality}), who constructed a crossed product $G\ltimes \B(\msf{H})$ with diffuse center, for a Hilbert space $\msf{H}$ and a locally compact group $G$ such that $\LL^{\infty}(\wh{G})=\oon{vN}(G)$ is a factor.
\end{proof}

\subsection{Example 5}\label{example5}

Let $\TT$ be the circle group and $D(\TT)=\TT\times \ZZ$ its Drinfeld double, which is a quasi-triangular quantum group with $\oon{R}$-matrix $\wh{\oon{R}}=(\gamma_{\ZZ\subseteq D(\TT)}\otimes \gamma_{\TT\subseteq D(\TT)})\ww^{\TT}=\ww^{\TT}_{[23]}$ (Proposition \ref{prop14}). Set $\M=\N=\LL^{\infty}(\TT)$, fix $\zeta\in \TT$ and equip $\M,\N$ with actions
\[
\alpha^{\M}_{\TT}\colon \TT\curvearrowright \M\colon \alpha^{\M}_{\TT}=\Delta_{\TT},\quad 
\alpha^{\N}_{\ZZ}\colon \ZZ\curvearrowright\N\colon \alpha^{\N}_{\ZZ}(f)(k,\lambda)=f(\zeta^{k}\lambda)
\]
and let $\alpha^{\M}_{\ZZ}\colon \ZZ\curvearrowright \M$, $\alpha^{\N}_{\TT}\colon \TT\curvearrowright \N$ be the trivial actions. Since $\TT,\ZZ$ are classical, the  Yetter-Drinfeld condition is met and by Proposition \ref{prop15} we obtain actions $\alpha^{\M},\alpha^{\N}$ of $D(\TT)$. For the implementation of the action $\alpha^{\M}_{\TT}$ we can take the identity, and for the implementation of the trivial actions the respective counits. We let $U^{\N,\ZZ}, \phi_{\N,\ZZ}$ be the standard implementation of $\alpha^{\N}_{\ZZ}$. By Proposition \ref{prop16} we obtain implementations $\phi_{\M},\phi_{\N}$ of the actions of $D(\TT)$. In this situation, we can consider two braided tensor products $\M\ov\boxtimes \N, \N\ov\boxtimes\M$, both constructed with respect to $\wh{\oon{R}}$.

\begin{proposition}\label{prop18}
$\M\ov\boxtimes\N$ is equal to the tensor product $\M\bar\otimes\N=\LL^{\infty}(\TT)\bar\otimes \LL^{\infty}(\TT)$, while $\N\ov\boxtimes\M$ is isomorphic with $\ZZ\ltimes \N$. In particular, if $\zeta=e^{2\pi i x}$ with $x\in \RR\setminus \QQ$, then $\N\ov\boxtimes \M$ is isomorphic with the injective $\oon{II}_1$ factor and $\M\ov\boxtimes \N\nsimeq \N\ov\boxtimes \M$.
\end{proposition}

Compare this result with Proposition \ref{prop22}.

\begin{proof}
Since all the involved quantum groups are coamenable, we have
\[\begin{split}
\M\ov\boxtimes\N&=
\ov{\lin}^{\,\swot}\{
(\phi_{\M}\otimes \phi_{\N})(\wh{\oon{R}}^u)
(f\otimes \I)
(\phi_{\M}\otimes \phi_{\N})(\wh{\oon{R}}^u)^*
(\I\otimes g)\mid f\in \M,g\in \N\}\\
&=
\ov{\lin}^{\,\swot}\{
(\id\otimes \eps_{\TT}\otimes \eps_{\ZZ}\otimes \phi_{\N,\ZZ})(\ww^{\TT}_{[23]})
(f\otimes \I)
(\id\otimes \eps_{\TT}\otimes \eps_{\ZZ}\otimes \phi_{\N,\ZZ})(\ww^{\TT}_{[23]})^*
(\I\otimes g)\mid f\in \M,g\in \N\}\\
&=
\ov{\lin}^{\,\swot}\{
f\otimes g\mid f\in \M,g\in \N\}=\M\bar\otimes \N.
\end{split}\]
If we take the braided tensor product with swapped factors, we get
\[\begin{split}
\N\ov{\boxtimes}\M&=
\ov{\lin}^{\,\swot}\{
(\phi_{\N}\otimes \phi_{\M})(\wh{\oon{R}}^u)
(g\otimes \I)
(\phi_{\N}\otimes \phi_{\M})(\wh{\oon{R}}^u)^*
(\I\otimes f)\mid g\in \N,f\in \M\}\\
&=
\ov{\lin}^{\,\swot}\{
(\eps_{\ZZ}\otimes \phi_{\N,\ZZ}\otimes \id\otimes \eps_{\TT})(\ww^{\TT}_{[23]})
(f\otimes \I)
(\eps_{\ZZ}\otimes \phi_{\N,\ZZ}\otimes \id\otimes \eps_{\TT})(\ww^{\TT}_{[23]})^*
(\I\otimes g)\mid f\in \M,g\in \N\}\\
&=
\ov{\lin}^{\,\swot}\{
(\phi_{\N,\ZZ}\otimes \id)(\ww^{\TT})
(f\otimes \I)
( \phi_{\N,\ZZ}\otimes \id)(\ww^{\TT})^*
(\I\otimes g)\mid f\in \M,g\in \N\}.
\end{split}\]
If we apply the flip, we get the isomorphic von Neumann algebra
\begin{equation}\label{eq48}
\ov{\lin}^{\,\swot}\{
(\id\otimes \phi_{\N,\ZZ})(\ww^{\ZZ })^*
(\I\otimes f)
( \id\otimes \phi_{\N,\ZZ})(\ww^{\ZZ})
(g\otimes \I)\mid f\in \M,g\in \N\}.
\end{equation}
Since $\M=\N=\LL^{\infty}(\TT)=\LL^{\infty}(\wh{\ZZ})$, we recognise in \eqref{eq48} the crossed product $\ZZ\ltimes \N$.

Assume now that $\zeta=e^{2\pi i x }$ with $x\in \RR\setminus \QQ$. Then it is well known that the action $\ZZ\curvearrowright \LL^{\infty}(\TT)$ (called the irrational rotation) is free and ergodic, hence $\ZZ\ltimes \LL^{\infty}(\TT)$ is a factor \cite[Corollary XIII 1.6]{TakesakiIII}. It is type $\oon{II}_1$ by \cite[Theorem XIII 1.7]{TakesakiIII} and injective by \cite[Theorem XV 3.16]{TakesakiIII}.
\end{proof}

\section{Appendix }
\subsection{Implementation of an isomorphism}\label{sec:AppendixImplementation}\noindent

Let $\GG,\HH$ be isomorphic locally compact quantum groups. As explained in subsection \ref{sec:Independence}, this means that we have compatible isomorphisms $\theta\colon \mrm{C}_0^u(\GG)\rightarrow \mrm{C}_0^u(\HH)$, $\theta_r\colon \mrm{C}_0(\GG)\rightarrow \mrm{C}_0(\HH)$, $\gamma\colon \LL^{\infty}(\GG)\rightarrow \LL^{\infty}(\HH)$ which respect to the comultiplications. By taking the dual morphism, we get isomorphisms $\wh\theta\colon \mrm{C}_0^u(\whH)\rightarrow \mrm{C}_0^u(\whG)$, $\wh{\theta}_r\colon \mrm{C}_0(\whH)\rightarrow \mrm{C}_0(\whG)$, $\wh\gamma\colon \LL^{\infty}(\whH)\rightarrow \LL^{\infty}(\whG)$. Since our von Neumann algebras $\LL^{\infty}(\GG)\subseteq \B(\LL^2(\GG))$, $\LL^{\infty}(\HH)\subseteq \B(\LL^2(\HH))$ are in standard position, by \cite[Theorem IX 1.14]{TakesakiII} there is a unique unitary $v\colon \LL^2(\GG)\rightarrow \LL^2(\HH)$ which satisfies $J_{\HH} v =v J_{\GG}, v \mf{P}_{\vp_{\GG}}=\mf{P}_{\vp_{\HH}}$ and implements $\gamma$ in the sense that $\gamma(x)=vxv^*\,(x\in \LL^{\infty}(\GG))$. Similarly, there is a unique unitary $\wh{v}\colon\LL^2(\HH)\rightarrow \LL^2(\GG)$ which implements $\wh\gamma$ and satisfies $J_{\whG} \wh{v}=\wh{v} J_{\whH}$, $\wh{v}\mf{P}_{\vp_{\whH}}=\mf{P}_{\vp_{\whG}}$. Here $\vp_{\GG}$ is the left Haar integral on $\GG$ and $\mf{P}_{\vp_{\GG}}\subseteq \LL^2(\GG)$ is the associated positive cone \cite[Section 2.23]{Stratila}, similarly for other quantum groups.

We begin with a general technical lemma of independent interest.

\begin{lemma}\label{lemma6}
Let $\KK$ be a locally compact quantum group, $x\in \mf{N}_{\vp_{\KK}}$ and $\omega\in \LL^1(\KK)$ such that $(\omega\otimes\id)\ww^{\KK}\in \mf{N}_{\vp_{\wh{\KK}}}$. Then $\la \Lambda_{\vp_{\KK}}(x)\,|\,\Lambda_{\vp_{\wh{\KK}}}(\,(\omega\otimes\id)\ww^{\KK}\,) \ra = \omega(x^*)$.
\end{lemma}

\begin{proof}
Define the subspace $\mscr{J}=\{\nu\in\LL^1(\KK)\mid \exists_{C>0}\forall_{y\in \mf{N}_{\vp_{\KK}}}|\nu(y^*)|\le C \|\Lambda_{\vp_{\KK}}(y)\|\}$, and similarly its dual version $\wh{\mscr{J}}\subseteq\LL^1(\wh{\KK})$. By the construction of the dual Haar integral \cite[Page 75]{KustermansVaesVN}, we have that $\{(\nu\otimes\id)\ww^{\KK}\mid \nu\in \mscr{J}\}$ forms a $\ssots\times\|\cdot\|$ core for $\Lambda_{\vp_{\wh\KK}}$ and $\la \Lambda_{\vp_{\KK}}(z) | \Lambda_{\vp_{\wh\KK}}((\nu\otimes \id)\ww^{\KK})\ra=\nu(z^*)$ for $z\in \mf{N}_{\vp_{\KK}},\nu\in\mscr{J}$. Consequently we can find nets $\omega_i\in \mscr{J}\,(i\in I)$, $\wh{\omega}_j\in\wh{\mscr{J}}(j\in J)$ such that
\[
(\omega_i\otimes\id)\ww^{\KK}\xrightarrow[i\in I]{\ssots} 
(\omega\otimes \id)\ww^{\KK},\quad
\Lambda_{\vp_{\wh\KK}}(\, (\omega_i\otimes\id)\ww^{\KK}\,)\xrightarrow[i\in I]{} 
\Lambda_{\vp_{\wh\KK}}(\, (\omega\otimes \id)\ww^{\KK}\,)
\]
and
\[
(\wh\omega_j\otimes\id)\ww^{\wh\KK}\xrightarrow[j\in J]{\ssots} 
x,\quad
\Lambda_{\vp_{\KK}}(\, (\wh\omega_j\otimes\id)\ww^{\wh\KK}\,)\xrightarrow[j\in J]{} 
\Lambda_{\vp_{\KK}}(x ).
\]
Using this, we calculate
\[\begin{split}
&\quad\;
\la \Lambda_{\vp_{\KK}}(x)\,|\,\Lambda_{\vp_{\wh{\KK}}}(\,(\omega\otimes\id)\ww^{\KK}\,) \ra =
\lim_{j\in J}\lim_{i\in I} \la 
\Lambda_{\vp_{\KK}}( (\wh\omega_j\otimes\id)\ww^{\wh\KK}) \,|\,
\Lambda_{\vp_{\wh{\KK}}}(\,(\omega_i\otimes\id)\ww^{\KK}\,)\ra \\
&=
\lim_{j\in J}\lim_{i\in I}
\omega_i \bigl(  ((\wh\omega_j\otimes\id)\ww^{\wh\KK})^*\bigr)=
\lim_{j\in J}\lim_{i\in I}
(\ov{\wh{\omega}_j}\otimes \omega_i)\ww^{\wh\KK *}=
\lim_{j\in J}\lim_{i\in I}
\ov{\wh{\omega}_j}( (\omega_i\otimes \id)\ww^{\KK} )\\
&=
\lim_{j\in J}
\ov{\wh{\omega}_j}((\omega\otimes\id)\ww^{\KK}) =
\lim_{j\in J}
\ov{ (\ov \omega\otimes \wh{\omega}_j)\ww^{\KK *} }=
\lim_{j\in J}
\ov{\ov\omega (\,  (\wh{\omega}_j\otimes \id  )\ww^{\wh\KK} )}=
\ov{\ov{\omega}(x)}=
\omega(x^*).
\end{split}\]
\end{proof}

\begin{proposition}\label{prop5}
We have $\wh{v}=v^*$.
\end{proposition}

\begin{proof}
Observe first that $\vp_{\HH}\circ \gamma$ is a n.s.f.~weight on $\LL^{\infty}(\GG)$ and for $x\in \LL^{\infty}(\GG)^+,\omega\in\LL^1(\GG)^+$ we have, using the strong form of left invariance \cite[Proposition 3.1]{KustermansVaesVN}
\[
\vp_{\HH}\circ\gamma ( (\omega\otimes \id)\Delta_{\GG}(x))=
\vp_{\HH} ( (\omega\gamma^{-1}\otimes \id)\Delta_{\HH}(\gamma (x)))=
\omega(\gamma^{-1}(\I)) \vp_{\HH}(\gamma(x))=
\omega(\I) \vp_{\HH}(\gamma(x)),
\]
which shows that $\vp_{\HH}\circ\gamma$ is left invariant. Uniqueness of Haar integrals \cite[Theorem 3.5]{DaeleVN} gives $t>0$ such that $\vp_{\HH}\circ\gamma=t^{-2} \vp_{\GG}$. In particular $\gamma$ restricts to a bijection $\mf{N}_{\vp_{\GG}}\rightarrow \mf{N}_{\vp_{\HH}}$ and $\gamma\circ\sigma^{\vp_{\GG}}_s=\sigma^{\vp_{\HH}}_s\circ\gamma\,(s\in \RR)$. Define
\[
v_0\colon \LL^2(\GG)\rightarrow \LL^2(\HH)\colon v_0 \Lambda_{\vp_{\GG}}(x) = t \Lambda_{\vp_{\HH}}(\gamma(x))\quad(x\in \mf{N}_{\vp_{\GG}}).
\]
One easily sees that $v_0$ is well defined and extends to a unitary map. We claim that $v_0=v$. For $x_1\in \mf{N}_{\vp_{\GG}}$ such that $x_1\in  \Dom(\sigma^{\vp_{\GG}}_{i/2})$, $\sigma^{\vp_{\GG}}_{i/2}(x)^*\in \mf{N}_{\vp_{\GG}}$ we have
\[
J_{\HH} v_0 \Lambda_{\vp_{\GG}}(x_1)=
t \Lambda_{\vp_{\HH}}(\sigma^{\vp_{\HH}}_{i/2}(\gamma(x_1))^*)=
t \Lambda_{\vp_{\HH}}(\gamma ( \sigma^{\vp_{\GG}}_{i/2}(x_1)^*))=
v_0J_{\GG}\Lambda_{\vp_{\GG}}(x_1).
\]
Density of the space of vectors $\Lambda_{\vp_{\GG}}(x_1)$ \cite[Theorem 10.20]{StratilaZsido} shows $J_{\HH} v_0 = v_0 J_{\GG}$. For $x_2\in\Linf,x_3\in \mf{N}_{\vp_{\GG}}$ we have
\[
\gamma(x_2) t \Lambda_{\vp_{\HH}}(\gamma( x_3))=
t \Lambda_{\vp_{\HH}}(\gamma(x_2 x_3))=
v_0 x_2v^*_0 v_0 \Lambda_{\vp_{\GG}}(x_3)=
v_0 x_2v^*_0 t\Lambda_{\vp_{\HH}}(x_3)
\]
which shows $\gamma(x_2)=v_0 x_2 v_0^*$, i.e.~$v_0$ implements $\gamma$. Finally, by \cite[Equation (3), page 355]{StratilaZsido} we have
\[
\mf{P}_{\vp_{\GG}}=\ov{\{x J_{\GG} \Lambda_{\vp_{\GG}}(x)\,|\,x\in \mf{N}_{\vp_{\GG}}\cap {\mf{N}_{\vp_{\GG}}}^*\}}
\]
and similarly for $\mf{P}_{\vp_{\HH}}$. Consequently
\[\begin{split}
&\quad\;
v_0 \mf{P}_{\vp_{\GG}}=\ov{\{v_0 x J_{\GG} \Lambda_{\vp_{\GG}}(x)\,|\,x\in \mf{N}_{\vp_{\GG}}\cap {\mf{N}_{\vp_{\GG}}}^*\}}=
\ov{\{\gamma(x) J_{\HH}v_0 \Lambda_{\vp_{\GG}}(x)\,|\,x\in \mf{N}_{\vp_{\GG}}\cap {\mf{N}_{\vp_{\GG}}}^*\}}\\
&=
\ov{\{\gamma(x) J_{\HH} \Lambda_{\vp_{\HH}}(\gamma(x))\,|\,x\in \mf{N}_{\vp_{\GG}}\cap {\mf{N}_{\vp_{\GG}}}^*\}}=
\ov{\{y J_{\HH} \Lambda_{\vp_{\HH}}(y)\,|\,y\in \mf{N}_{\vp_{\HH}}\cap {\mf{N}_{\vp_{\HH}}}^*\}}=\mf{P}_{\vp_{\HH}}.
\end{split}\]
Uniqueness of the standard implementation \cite[Theorem 2.3]{Haagerup} allows us to conclude $v_0=v$. An analogous argument gives
\begin{equation}\label{eq12}
\wh{v} \Lambda_{\vp_{\whH}}(\wh y)=r  \Lambda_{\vp_{\whG}}(\wh{\gamma}(\wh y))\quad(\wh y\in \mf{N}_{\vp_{\whH}})
\end{equation}
for some $r>0$. Next, for $x_3\in \mf{N}_{\vp_{\GG}},\omega_1\in\LL^1(\HH)$ we have
\[\begin{split}
&\quad\;
(\omega_1\otimes \id)(\ww^{\HH *})t\Lambda_{\vp_{\HH}}(\gamma(x_3))=
t\Lambda_{\vp_{\HH}}( (\omega_1\otimes \id)\Delta_{\HH}(\gamma(x_3)) )=
t\Lambda_{\vp_{\HH}}( \gamma ((\omega_1\gamma \otimes \id)\Delta_{\GG}(x_3) ))\\
&=
v \Lambda_{\vp_{\GG}} ((\omega_1\gamma\otimes \id)\Delta_{\GG}(x_3))=
v (\omega_1\gamma\otimes \id)(\ww^{\GG *}) \Lambda_{\vp_{\GG}}(x_3)=
v \wh{\gamma}\bigl( (\omega_1\otimes \id)(\ww^{\HH *})\bigr)v^* v \Lambda_{\vp_{\GG}}(x_3)\\
&=
v \wh{\gamma}\bigl( (\omega_1\otimes \id)(\ww^{\HH *})\bigr)v^* t \Lambda_{\vp_{\HH}}(\gamma(x_3)),
\end{split}\]
hence $(\omega_1\otimes \id)(\ww^{\HH *})= v \wh{\gamma}\bigl( (\omega_1\otimes \id)(\ww^{\HH *})\bigr) v^*$. A density argument shows that $v^*$ implements $\wh\gamma$.

 Take $\wh{x}\in \mf{N}_{\vp_{\whG}}$ and $x\in\mf{N}_{\vp_{\GG}}$ of the form $x=(\wh\omega\otimes\id)\ww^{\whG}$ (see the previous lemma or \cite[Page 75]{KustermansVaesVN}). Using Lemma \ref{lemma6} we calculate
\[\begin{split}
&\quad\;
\la v \Lambda_{\vp_{\whG}}(\wh x) | v \Lambda_{\vp_{\GG}}(x)\ra =
\la \Lambda_{\vp_{\whG}}(\wh x) | \Lambda_{\vp_{\GG}}(x)\ra =
\wh\omega(\wh{x}^*)=
\wh\omega\circ \wh{\gamma} (\wh{\gamma}^{-1}(\wh{x})^*)=
\bigl\la \Lambda_{\vp_{\whH}}(\wh{\gamma}^{-1}(\wh{x})) \big|
\Lambda_{\vp_{\HH}}\bigl(
(\wh{\omega}\circ\wh{\gamma}\otimes \id)\ww^{\whH}
\bigr)\bigr\ra\\
&=
\bigl\la \Lambda_{\vp_{\whH}}(\wh{\gamma}^{-1}(\wh{x})) \big|
\Lambda_{\vp_{\HH}}\bigl(
\gamma \bigl( (\wh{\omega}\otimes \id)\ww^{\whG}\bigr)
\bigr)\bigr\ra=
\bigl\la t^{-1}\Lambda_{\vp_{\whH}}(\wh{\gamma}^{-1}(\wh{x})) \big|
t\Lambda_{\vp_{\HH}}(\gamma (x))\bigr\ra=
\bigl\la t^{-1}\Lambda_{\vp_{\whH}}(\wh{\gamma}^{-1}(\wh{x})) \big|
v \Lambda_{\vp_{\GG}}(x)\bigr\ra.
\end{split}\]
Density of the space of vectors of the form $\Lambda_{\vp_{\GG}}(x)$ shows $v\Lambda_{\vp_{\whG}}(\wh x)=t^{-1} \Lambda_{\vp_{\HH}} (\wh{\gamma}^{-1} (\wh x))$. Equation \eqref{eq12} gives $\Lambda_{\vp_{\HH}}(\wh\gamma^{-1}(\wh x))=r^{-1} \wh{v}^* \Lambda_{\vp_{\whG}}(\wh x)$, hence again by a density argument we find $v=t^{-1} r^{-1} \wh{v}^* $. Taking norms gives $t=r^{-1}$ and $v=\wh{v}^*$ as claimed.
\end{proof}


\subsection{Comment about left-right conventions}
In our paper, we have decided to work with left actions. Right actions of locally compact quantum groups on von Neumann algebras are defined in a completely analogous way, as injective, unital normal $*$-homomorphisms $\M\rightarrow \M\bar\otimes \LL^{\infty}(\GG)$ satisfying the natural coaction condition. Working with left or right actions is a matter of choice. There are also ways of passing between them, hence typically one can translate results  between these conventions. For example, if $\alpha^{\M}\colon \M\rightarrow \LL^{\infty}(\GG)\bar\otimes \M$ is a left action, then $\chi\alpha \colon \M\rightarrow \M\bar\otimes \LL^{\infty}(\GG^{op})$ is a right action of $\GG^{op}$ on $\M$, and $(j_{\M}\otimes R_{\GG})\chi \alpha^{\M} j_{\M}^{-1}\colon \M^{op}\rightarrow \M^{op}\bar\otimes \LL^{\infty}(\GG)$ is a right action of $\GG$ on $\M^{op}$ ($j_{\M}\colon \M\rightarrow \M^{op}$ is the canonical $*$-linear, antimultiplicative map which can be spatially realised as $j_{\M}=J_{\M}(\cdot)^* J_{\M}$ when $\M\subseteq \B(\LL^2(\M))$ and $J_{\M}$ is the modular conjugation).

While left actions are implemented by (left) representations, right actions are implemented by right representations. More precisely, a (unitary) right representation of $\GG$ on $\msf{H}$ is a unitary element $U\in \M(\mc{K}(\msf{H}_U)\otimes \mrm{C}_0(\GG))$ satisfying $(\id\otimes \Delta_{\GG})U=U_{[12]}U_{[13]}$. Any such $U$ can be written as $U=(\phi_{U}\otimes \id)\Vv^{\GG}$ for a non-degenerate $*$-homomorphism $\phi_{U}\colon \mrm{C}_0^u(\widecheck{\GG})\rightarrow \B(\msf{H}_{U})$ ($\widecheck{\GG}=\whG'$ is the ``right dual''). Then we say that $U$ (or $\phi_{U}$) implements the right action $\alpha\colon \M\curvearrowleft \GG$ if $\M\subseteq \B(\msf{H}_U)$ and $\alpha(m)=U(m\otimes \I) U^*\,(m\in\M)$.

Let us (partially) indicate how our results change when passing to right conventions. First, with a bicharacter $\wh{\mc{X}}\in \LL^{\infty}(\whH)\bar\otimes \Linfd$ we associate its right version $\widecheck{\mc{X}}=(j_{\whH}\otimes j_{\whG})(\wh{\mc{X}})^*\in \LL^{\infty}(\widecheck{\HH})\bar\otimes \LL^{\infty}(\widecheck{\GG})$. It has the lift $\widecheck{\mc{X}}^u\in \M(\mrm{C}_0^u(\widecheck{\HH})\otimes \mrm{C}_0^u(\widecheck{\GG}))$. Next, in the right setting, the definition of braided flip operators and braided tensor product change. Let $U=(\phi_{U}\otimes \id)\Vv^{\HH},V=(\phi_{V}\otimes \id)\Vv^{\GG}$ be right representations, then we define
\[
\braid{U}{V}=
(\phi_{V}\otimes \phi_{U})(\widecheck{\mc{X}}^u)^* \Sigma\colon 
\msf{H}_U\otimes \msf{H}_V\rightarrow
\msf{H}_V\otimes \msf{H}_U.
\]

If $\M\curvearrowleft \HH,\N\curvearrowleft \GG$ are right actions implemented by $U^{\M}=(\phi_{\M}\otimes \id)\Vv^{\HH}$, $U^{\N}=(\phi_{\N}\otimes \id)\Vv^{\GG}$, then one defines
\[
\M\ov\boxtimes \N=\ov{\lin}^{\swot}\,\{\iota_{\M}(m)\iota_{\N}(n)\mid m\in\M,n\in\N\}
\]
where (writing $\braid{\N}{\M}=\braid{U^{\N}}{U^{\M}}$)
\[
\iota_{\M}(m)=m\otimes \I\;(m\in\M),\quad 
\iota_{\N}(n)=
\braid{\N}{\M}(n\otimes \I) (\braid{\N}{\M})^*=
(\phi_{\M}\otimes \phi_{\N})(\widecheck{\mc{X}}^u)^*
(\I\otimes n)
(\phi_{\M}\otimes \phi_{\N})(\widecheck{\mc{X}}^u)\;(n\in\N).
\]
\section{Acknowledgements}
The work of JK was partially supported by FWO grant 1246624N and EPSRC grant no EP/K032208/1. Furthermore, JK would like to thank the Isaac Newton Institute for Mathematical Sciences, Cambridge, for support and hospitality during the programme ``Quantum information, quantum groups and operator algebras'' where part of the work on this paper was undertaken. The work of KDC was supported by Fonds voor Wetenschappelijk Onderzoek (FWO), grant G032919N. Both authors thank Stefaan Vaes for informing them about the examples in Corollary \ref{cor2}.\eqref{cor24}, and for informing them about the paper \cite{Houdayer}.

\bibliographystyle{plain}
\bibliography{bibliography}

\end{document}